\documentclass{article}
\usepackage[utf8]{inputenc}
\usepackage[margin=1 in]{geometry}
\usepackage{amsmath}
\usepackage{amsfonts}
\usepackage{amssymb}
\usepackage{amsthm}
\usepackage{bbm}
\usepackage{mathrsfs}
\usepackage{enumitem}
\usepackage{hyperref}
\usepackage{xcolor}
\usepackage{graphicx}
\usepackage{tikz}
\usetikzlibrary{graphs}

\newtheorem{theorem}{Theorem}
\numberwithin{theorem}{subsection}
\newtheorem{lemma}[theorem]{Lemma}
\newtheorem{prop}[theorem]{Proposition}
\newtheorem{cor}[theorem]{Corollary}
\newtheorem{conjecture}[theorem]{Conjecture}

\newtheorem{question}[theorem]{Question}

\theoremstyle{definition}
\newtheorem{definition}[theorem]{Definition}
\newtheorem{remark}[theorem]{Remark}
\newtheorem{example}[theorem]{Example}

\newcommand{\tb}{\textbf}

\newcommand{\tn}{\textnormal}
\newcommand{\se}{\subseteq}
\newcommand{\ol}{\overline}
\newcommand{\lam}{\lambda}
\newcommand{\tcb}{\textcolor{blue}}
\newcommand{\Aut}{\tn{Aut}}
\newcommand{\Lam}{\Lambda}
\newcommand{\til}{\widetilde}

\usepackage[maxbibnames=10,style=alphabetic,maxalphanames=10]{biblatex}
\title{Distinguishability and linear independence for $H$-chromatic symmetric functions}
\author{Shao Yuan Lin \\ \href{mailto:syl54@cantab.ac.uk}{syl54@cantab.ac.uk} \and Laura Pierson \\ \href{mailto:lcpierson73@gmail.com}{lcpierson73@gmail.com}}
\begin{document}

\maketitle

\begin{abstract}
    We study the $H$-chromatic symmetric functions $X_G^H$ (introduced in \cite{EFHKY22} as a generalization of the chromatic symmetric function (CSF) $X_G$), which track homomorphisms from the graph $G$ to the graph $H$. We focus first on the case of self-chromatic symmetric functions (self-CSFs) $X_G^G$, making some progress toward a conjecture from \cite{EFHKY22} that the self-CSF, like the normal CSF, is always different for different trees. In particular, we show that the self-CSF distinguishes trees from non-trees with just one exception, we check using Sage \cite{sage} that it distinguishes all trees on up to 12 vertices, and we show that it determines the number of legs of a spider and the degree sequence of a caterpillar given its spine length. We also show that the self-CSF detects the number of connected components of a forest, again with just one exception. Then we prove some results about the power sum expansions for $H$-CSFs when $H$ is a complete bipartite graph, in particular proving that the conjecture from \cite{EFHKY22} about $p$-monotonicity of $\omega(X_G^H)$ for $H$ a star holds as long as $H$ is sufficiently large compared to $G$. We also show that the self-CSFs of complete multipartite graphs form a basis for the ring $\Lam$ of symmetric functions, and we give some construction of bases for the vector space $\Lam^n$ of degree $n$ symmetric functions using $H$-CSFs $X_G^H$ where $H$ is a fixed graph that is not a complete graph, answering a question from \cite{EFHKY22} about whether such bases exist. However, we show that there generally do not exist such bases with $G$ fixed, even with loops, answering another question from \cite{EFHKY22}. We also define the $H$-chromatic polynomial as an analogue of the chromatic polynomial, and ask when it is the same for different graphs.
\end{abstract}

\section{Introduction}\label{sec:intro}

In \cite{EFHKY22}, Eagles, Foley, Huang, Karangozishvili, and Yu define the \emph{\tb{\tcb{$H$-chromatic symmetric function ($H$-CSF)}}} of a graph $G$ as $$X_G^H(x_1,x_2,\dots,x_{|V(H)|}) := \sum_{f:G\to H} \ \ \sum_{\phi:V(H)\to\{1,2,\dots,|V(H)|\}} \ \ \prod_{v\in V(G)}x_{\phi(f(v))},$$ where $f:G\to H$ is a \emph{\tb{\tcb{graph homomorphism}}} from $G$ to the graph $H$ (meaning $f$ sends every vertex of $G$ to a vertex of $H$, such that adjacent vertices of $G$ map to adjacent vertices of $H$), and $\phi:V(H)\to\{1,2,\dots,|V(H)|\}$ is a labeling of the vertices of $H$ with distinct positive integers from 1 to $|V(H)|$. This can be extended to a symmetric function in a countably infinite set of variables $x_1,x_2,\dots$ by writing it in the basis of \emph{\tb{\tcb{monomial symmetric functions}}} in the variables $x_1,\dots,x_{|V(H)|}$, $$m_\lam(x_1,x_2,\dots,x_{|V(H)|}):=\sum_{i_1,\dots,i_{\ell(\lam)}\in\{1,2,\dots,|V(H)|\}\tn{ all distinct}}x_{i_1}^{\lam_1}\dots x_{i_{\ell(\lam)}}^{\lam_{\ell(\lam)}},$$ and then replacing each $m_\lam(x_1,\dots,x_{|V(H)|})$ with the monomial symmetric function $$m_\lam(x_1,x_2,\dots):=\sum_{i_1,\dots,i_{\ell(\lam)}\in\mathbb{N}\tn{ all distinct}}x_{i_1}^{\lam_1}\dots x_{i_{\ell(\lam)}}^{\lam_{\ell(\lam)}}$$ in the countably infinite set of variables $x_1,x_2,\dots.$ Stanley's \emph{\tb{\tcb{chromatic symmetric function (CSF)}}} from \cite{stanley1995symmetric} is a scalar multiple of the special case where $H$ is a complete graph. The CSF is defined by $$X_G(x_1,x_2,\dots) := \sum_{\kappa\tn{ a proper coloring}} \ \prod_{v\in V(G)}x_{\kappa(v)},$$ where a \emph{\tb{\tcb{proper set coloring}}} is a function $\kappa:V(G)\to \{1,2,3,\dots\}$ assigning a positive-integer-valued color to each vertex such that adjacent vertices get distinct colors. 

The authors of \cite{EFHKY22} show that $$X_G^H = \sum_\lam d_\lam m_\lam^{|V(H)|},$$ where $d_\lam$ counts homomorphisms $f:G\to H$ of \emph{\tb{\tcb{type $\lam$}}}, meaning the sizes of the preimages for the vertices in $H$ are equal to the parts of $\lam$, and $m_\lam^{|V(H)|}$ is a scalar multiple of the monomial symmetric function $m_\lam$, specifically, $$m_\lam^{n} := \frac{n!}{\binom{n}{r_1(\lam),r_2(\lam),\dots,n-\ell(\lam)}}m_\lam,$$ where $r_i(\lam)$ is the number of parts of $\lam$ of size $i$.

\bigskip 

\noindent In \S\ref{sec:self-CSF} of this paper, we focus on the special case of \emph{\tb{\tcb{self-CSFs}}} $X_G^G$, where $G=H$:
\begin{itemize}
    \item In \S\ref{subsec:self-CSF_general}, we note that for large graphs, almost all pairs with the same number of vertices actually have the same self-CSF. However, we also show that the self-CSF is not a strictly weaker invariant than the CSF by giving an example of two graphs with the same CSF but different self-CSF.
    
    \item In \S\ref{subsec:self-CSF_bipartite}, we show that for connected bipartite graphs, the self-CSF determines the first few terms of the star sequence, while for disconnected bipartite graphs, it determines the multiset of nonzero differences between the part sizes in the connected components.
    
    \item In \S\ref{subsec:self-CSF_trees}, we check that the conjecture from \cite{EFHKY22} that the self-CSF distinguishes trees holds for trees on up to 12 vertices. We also show that the self-CSF distinguishes trees from non-trees except for the case $X_{P_3}^{P_3}=X_{K_1\sqcup K_2}^{K_1\sqcup K_2}$, and we show how to use the $m_{2,1^{n-2}}^n$ coefficient to learn some information about the leaves of the tree. Then in \S\ref{subsec:spiders} and \S\ref{subsec:caterpillars}, we show that additional information can be learned in the special cases of \emph{\tb{\tcb{spiders}}} and \emph{\tb{\tcb{caterpillars}}}.
    
    \item In \S\ref{subsec:self-CSF_forests}, we show that the self-CSF can detect the number of connected components of a forest, again with $X_{P_3}^{P_3}=X_{K_1\sqcup K_2}^{K_1\sqcup K_2}$ as a special exception.
\end{itemize}
In \S\ref{sec:p-expansions}, we look at expansion formulas for $X_G^H$ in the basis of \emph{\tb{\tcb{power sum symmetric functions}}} $p_\lam$ for several special cases:
\begin{itemize}
    \item In \S\ref{subsec:p_monotonicity_stars}, we give a formula for $X_G^{S_{n+1}}$ (where $S_{n+1}$ is the $(n+1)$-vertex \emph{\tb{\tcb{star graph}}}) that holds as long as $G$ is bipartite (since otherwise $X_G^{S_{n+1}}=0$) and $n$ is larger than both parts in any bipartition of $G$. We use this to prove the conjecture from \cite{EFHKY22} that $\omega(X_G^{S_{n+1}})$ always has all nonnegative or all nonpositive $p$-coefficients for those cases.
    
    \item In \S\ref{subsec:p_K_2,n}, we give a formula for the $p$-expansion of $X_G^{K_{2,n}}$, again assuming $G$ is bipartite and $n$ is larger than both parts in any bipartition of $G$, and in \S\ref{subsec:p_K_m,n}, we extend these results to prove some general facts about the $p$-expansion of $X_G^{K_{m,n}}$ for $n$ sufficiently large compared to $G$.
\end{itemize}
In \S\ref{sec:bases}, we give some examples of bases for the ring $\Lam$ of symmetric functions such that all the basis elements are $H$-CSFs:
\begin{itemize}
    \item In \S\ref{subsec:self-CSF_basis}, we show that the self-CSFs of complete multipartite graphs $K_\lam$ form a basis for $\Lam$.

    \item In \S\ref{subsec:clique_containing_basis}, we gives several general classes of graphs $H$  for which one can construct a basis $X_{G_\lam}^H$ for the vector space $\Lam^n$ of symmetric functions of degree $n$ using $H$-CSFs with $H$ fixed, answering a question from \cite{EFHKY22} about whether such bases can exist where $H$ is not a complete graph. We show that the fixed graph $H$ can be any graph containing an $n$-vertex clique, or a disjoint union of cliques without too many singletons, or if one allows loops, a disjoint union of vertices with loops attached to them, or a path with a loop on one endpoint. On the other hand, we show that if $H$ is a disjoint union of a small number of edges, or is formed from $K_n$ by deleting disjoint edges, no such basis can exist.
    \item In \S\ref{subsec:G_fixed_no_bases}, we show that for $n\ge 12$, it is not possible to form a basis $X_G^{H_\lam}$ for $\Lam^n$ with $G$ fixed even if the graphs $G$ and $H_\lam$ are allowed to contain loops, answering another question from \cite{EFHKY22}.
\end{itemize}
In \S \ref{sec:recursive}, we look at a few simple ways to compute $H$-CSFs in terms of other $H$-CSFs:
\begin{itemize}
    \item In \S \ref{subsec:disjoint_union_H}, we show how to compute the $H$-CSF of a graph by breaking down $H$ into its connected components.
    \item In \S\ref{subsec:K_n^1}, we show how to compute $X_G^{K_n^1}$ in terms of the normal $X_G$ and CSFs of weighted graphs formed by contracting certain subsets of the vertices of $G$, where $K_n^1$ is the $n$-vertex clique with a loop on one vertex.
\end{itemize}
Finally, in \S\ref{sec:H-chromatic_poly}, we define the \emph{\tb{\tcb{$H$-chromatic polynomial}}} $\chi_G^H(k)$ as the specialization of $X_G^H$ to $x_1=\dots=x_k=1$ and $x_i=0$ for $i>k$.
\begin{itemize}
    \item In \S\ref{subsec:polynomiality}, we show that this is in fact a polynomial, and give an explicit formula for it along with an example.
    \item In \S\ref{subsec:equal_poly}, we show that for every $n$, there exist graphs with the same $S_n^1$-chromatic polynomial but different $S_n^1$-CSF, where $S_n^1$ is the star graph with a loop at the center vertex.
\end{itemize}

\section{Self-chromatic symmetric functions (self-CSFs) \texorpdfstring{$X_G^G$}{XGG}}\label{sec:self-CSF}

For the special case $G=H$, we know from \cite{EFHKY22} that the self-CSF $X_G^G$ can detect $n:=|V(G)|$ based on the degree, and it can detect whether or not $G$ is bipartite, because $G$ is bipartite if and only if $X_G^G$ contains a monomial $m_\lam^n$ with $\ell(\lam)=2$. However, it \emph{cannot} detect whether or not $G$ is connected, because the authors of \cite{EFHKY22} note that $X_{P_3}^{P_3} = X_{K_1\sqcup K_2}^{K_1\sqcup K_2},$ where $P_3$ is the 3-vertex path, and $K_1\sqcup K_2$ is the graph consisting of an isolated vertex plus two other vertices joined by an edge.

\subsection{Self-CSFs for general graphs}\label{subsec:self-CSF_general}

For the self-CSFs of general graphs $G$, we note that if the only endomorphism of $G$ is the identity automorphism, then $X_G^G=m^{|V(G)|}_{1^{|V(G)|}}$. In fact, this is the case for most graphs:

\begin{prop}\label{prop:random_graphs}
    Let $G(n,p)$ denote a random graph with $n$ vertices, each pair of which forms an edge with probability $p$. Then for every $p\in(n^{-1/3}\log^2n,1-n^{-1/3}\log^2n)$, $|\tn{End}(G)|=1$ asymptotically (as $n\rightarrow\infty$) almost surely for all $G\in G(n,p)$. 
\end{prop}

\begin{proof}
    As shown in Theorem 2 of \cite{bonato2009greatcores}, a random graph $G\in G(n,p)$ is a core (i.e. $\tn{End}(G)=\tn{Aut}(G)$) asymptotically almost surely for $p\in(n^{-1/3}\log^2n,1-n^{-1/3}\log^2n)$. On the other hand, Theorem 2 in \cite{erdos1963asymmetric} tells us that a random graph $G$ with $n$ vertices is asymmetric (i.e. $|\tn{Aut}(G)|=1$) asymptotically almost surely. Combining the two results and noting that the intersection of asymptotically almost sure events is asymptotically almost sure, the result follows.
\end{proof}

Since $(n^{-1/3}\log^2n,1-n^{-1/3}\log^2n)\rightarrow(0,1)$ as $n\rightarrow\infty$, the asymptotic almost-sureness extends to the set of all random graphs with $n$ vertices, allowing us to deduce the following.

\begin{cor}\label{cor:same_self-CSF}
    Almost all pairs of finite graphs with the same number of vertices have the same self-CSF.
\end{cor}

On the other hand, there do exist graphs that are distinguished by the self-CSF but not by the CSF, as the following example illustrates:

\begin{example}\label{ex:same_CSF_different_self-CSF}
    Stanley \cite{stanley1995symmetric} gives the following example of two graphs with the same CSF:
    
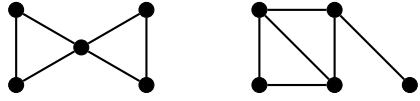
\begin{center} \begin{tikzpicture}
    \graph[nodes={draw, circle, fill=black, inner sep=2pt}, empty nodes, no placement, edges={thick}]{a[x=0,y=0]--b[x=0.866,y=0.5]--c[x=0.866,y=-0.5]--a--d[x=-0.866,y=0.5]--e[x=-0.866,y=-0.5]--a};
    \end{tikzpicture} \hspace{30pt} \begin{tikzpicture}
    \graph[nodes={draw, circle, fill=black, inner sep=2pt}, empty nodes, no placement, edges={thick}]{a[x=1,y=0]--b[x=0,y=1]--c[x=-1,y=1]--d[x=-1,y=0]--e[x=0,y=0]--b;c--e};
    \end{tikzpicture} \end{center}
    However, we can check that these two graphs have different self-CSF. We know that $\big[m_{1^{|V(G)|}}^{|V(G)|}\big]X_G^G$ counts the number of automorphisms of $G$. For the graphs above, the left graph has 8 automorphisms, because we can swap the 2 triangles or not, and then flip each triangle or not, while the right graph only has 2 automorphisms, because the only symmetry is to swap the 2 vertices forming the diagonal of the square. Thus, the self-CSFs of the two graphs have different $m_{1^5}^5$ coefficients, so they are not the same.
\end{example}

Example \ref{ex:same_CSF_different_self-CSF} shows that the self-CSF is not a strictly weaker invariant than the CSF. However, in light of Corollary \ref{cor:same_self-CSF}, $X_G^G$ is a bad tool for distinguishing graphs in general, especially large ones. This is why we restrict our attention in the remainder of this section to certain families of graphs, such as trees, where Corollary \ref{cor:same_self-CSF} does not hold. In fact, the authors of \cite{EFHKY22} conjecture that the self-CSF distinguishes all trees from each other, so we will focus here on making some progress toward that conjecture (\S\ref{subsec:self-CSF_trees}, \S\ref{subsec:spiders}, and \S\ref{subsec:caterpillars}), and on studying what information the self-CSF can tell us about a graph in the related cases of bipartite graphs (\S\ref{subsec:self-CSF_bipartite}) and forests (\S\ref{subsec:self-CSF_forests}).

\subsection{Self-CSFs for bipartite graphs}\label{subsec:self-CSF_bipartite}

We focus first on the case of bipartite graphs. We show that for connected bipartite graphs, the self-CSF can be used to recover the first part of the star sequence:

\begin{prop}\label{prop:bipartite_degrees}
    If $G$ is bipartite and is known to be connected, then for each for each $1\le \ell\le k$, $X_G^G$ determines the number $\sum_{v\in V(G)}\binom{\deg(v)}{\ell}$ of induced $(\ell+1)$-vertex stars in $G$, or equivalently, the power sums $\sum_{v\in V(G)}\deg(v)^\ell$, where $k$ is the size of the smaller part in the bipartition of $G$.
\end{prop}

\begin{proof}
    Let $V(G)=V_k\sqcup V_{n-k}$ be the bipartition, with $|V_k|=k$ and $|V_{n-k}|=n-k.$ This bipartition is unique since $G$ is connected. Also, the values $n-k$ and $k$ can be determined from $X_G^G$, since $m_{n-k,k}^n$ is the only length 2 monomial that shows up in $X_G^G$.
    
    Suppose $k\le n-k$, and consider the coefficient $$[m^n_{n-k,k-\ell+1,1^{\ell-1}}]X_G^G.$$ Let $f:G\to G$ be any graph homomorphism from $G$ to itself. Note that if a particular vertex $v\in V(G)$ maps to some vertex $f(v)\in V_k,$ then all neighbors of $v$ must map to vertices in $V_{n-k}$, and then all their neighbors must map to vertices in $V_k$, and so on. Thus, all vertices at even distance from $v$ must map to vertices in $V_k$, while all vertices at odd distance from $v$ must map to vertices in $V_{n-k}$. The vertices at even distance from $v$ are either those in $V_k$ or those in $V_{n-k}$, so either $f(V_k)\se V_k$ and $f(V_{n-k})\se V_{n-k}$, or $f(V_k)\se V_{n-k}$ and $f(V_{n-k})\se V_k$.
    
    Assume first that $n-k\ne k,$ so $n-k>k.$ Since $n-k > k > k-\ell$, the only way for $v\in V(G)$ to be the image under $f$ of $n-k$ distinct vertices in $G$ is if those $n-k$ vertices are precisely the vertices in $V_{n-k}$, since the vertices in $f^{-1}(v)$ must all belong to $V_k$ or all to $V_{n-k}$. Thus, the map $f:G\to G$ must send all vertices in $V_{n-k}$ to the same vertex $v=f(V_{n-k})\in V(G)$ (although $v$ need not be in $V_{n-k}$). Then, the vertices in $V_k$ all need to map to neighbors of $v$, such that $k-\ell+1$ of them map to the same neighbor of $v$, and the others map to distinct neighbors of $v$. For $\ell< k$, the number of ways to do this is $$[m_{n-k,k-\ell+1,1^{\ell-1}}^n]X_G^G=\ell!\binom{k}{\ell-1}\sum_{v\in V(G)}\binom{\deg(v)}{\ell},$$ since we can choose the $\ell$ neighbors of $v$ in the image of $f$ in $\binom{\deg(v)}{\ell}$ ways, we can choose the $\ell-1$ vertices in $V_k$ mapping to distinct neighbors of $v$ in $\binom{k}{\ell-1}$ ways, and we can choose which of the $\ell$ chosen neighbors of $v$ to map each of the $\ell-1$ vertices in $V_k$ to (and which of those $\ell$ vertices to map the remaining $k-\ell+1$ vertices in $V_k$ to) in $\ell!$ ways. For $\ell=k$, the set of $k-\ell+1$ leftover vertices in $V_k$ actually also has size 1, so it should not be distinguished from the other vertices in $V_k$. Thus, we instead get $$[m_{n-k,1^k}^n]X_G^G=k!\sum_{v\in V(G)}\binom{\deg(v)}{k},$$ since we just need to choose $k$ neighbors of $v$ to map all vertices of $V(G)$ to, and then choose the order of which vertex in $V_k$ maps to which neighbor of $v$.

    If $n-k=k$, we get almost the same thing, except that we could have either $f(V_k)=v$ or $f(V_{n-k})=v$, so the above coefficients would all get multiplied by 2 to account for these 2 possibilities. However, we can determine whether or not we are in this case since we know the values of $k$ and $n-k$, so we know whether or not to divide the coefficients by 2.

    It follows that we can learn $\sum_{v\in V(G)}\binom{\deg(v)}{\ell}$ for each $1 \le \ell\le k,$ so by taking linear combinations of those values, we can also learn $\sum_{v\in V(G)}\deg(v)^\ell$.
\end{proof}

We can also use $X_G^G$ to learn some information about $G$ when $G$ is a disconnected bipartite graph:

\begin{prop}\label{prop:diff_sequences}
    If $G$ is bipartite, the multiset of nonzero differences between the two part sizes in the bipartition of each connected component of $G$ can be determined from $X_G^G$.
\end{prop}

\begin{proof}
    Suppose the $i^{\text{th}}$ connected component has bipartition $V_i=V_{k_i}\sqcup V_{\ell_i}$, with $|V_{k_i}|=k_i\le |V_{\ell_i}|=\ell_i,$ and let $d_i=\ell_i-k_i.$ The claim is that the set of nonzero $d_i$ values can be recovered from $X_G^G$ (although we cannot necessarily learn how many $d_i$'s are equal to 0). Suppose $d_1 < d_2 <\dots < d_j$ are the nonzero differences that show up, listed in nondecreasing order, and let $c_i$ be the number of components with difference $d_i$.

    Consider the length 2 monomials $m_{n-k,k}^n$ that show up in $X_G^G$. Each of these corresponds to mapping all the edges of $G$ onto a single edge $v_1v_2$, but for each connected component $V_i$, we can choose which part of its bipartition maps onto $v_1$, and which part maps onto $v_2.$ Thus, the difference between $n-k$ and $k$ must be some sum of the form $\pm n_1 d_1 \pm n_2 d_2 \dots \pm n_j d_j,$ with $0\le n_i \le c_i.$ By finding the largest value of $n-k$ such that the monomial $m_{n-k,k}^n$ shows up, we can thus learn the value of $S=d_1+d_2+\dots+d_j$. The number of times this monomial shows up is $|E(G)|\cdot 2^{e+1},$ where $e$ is the number of connected components where the parts are equal in size, because there are $|E(G)|$ ways to choose which edge to map all the vertices onto, and then for each component with equal part sizes, we can map either of its parts to either of $v_1$ and $v_2,$ while for all the components with unequal part sizes, we must map the larger parts all to $v_1$ or all to $v_2$. We can thus learn the value of $|E(G)|\cdot 2^{e+1}$, although we cannot necessarily learn $e$ and $|E(G)|$ themselves.

    The next largest value of $n-k$ will correspond to $(n-k)-k=S-2d_1,$ since $d_1$ is the smallest difference, so the next farthest apart $n-k$ and $k$ can be is if we choose a component with difference $d_1$ and swap which vertex its larger part maps onto. Thus, we can learn the value of $d_1$. However, we can also learn how many copies of $d_1$ there are from the coefficient of this monomial. If there are $c_1$ copies, the coefficient is $|E(G)|\cdot 2^{e+1}\cdot c_1$, because again we must choose an edge to map all the vertices onto, then for each component with equal part sizes, we can independently choose which endpoint of the edge to match each of its parts to, while for the components with different part sizes, we must map all their larger parts to the same endpoint, except we must choose one of the $c_1$ components with difference $d_1$ to map to the opposite endpoint. Thus, we can learn the values of $c_1$ and $d_1.$

    If $c_1>1,$ we will also get length 2 monomials corresponding to the differences $S-2n_1d_1$ for each $2\le n_1 \le c_1$, by mapping exactly $n_1$ of the larger parts that are $d_1$ vertices larger than their smaller parts onto the opposite endpoint from all the other larger parts. However, we can count exactly how many such monomials we get, because we know the value of $c_1$, and there will be $|E(G)|\cdot 2^{e+1}\cdot \binom{c_1}{n_1}$ such monomials. Then if we eliminate those length 2 monomials and look only at the remaining ones, the largest remaining difference $(n-k)-k$ must equal the next smallest difference $d_2$, so we can learn $d_2$, and also $c_2$ from the coefficient of that monomial. 
    
    Continuing in this manner, once we know the values $d_1,d_2,\dots,d_i$ and $c_1,c_2,\dots,c_i$ for each $i$, we can count exactly how many length 2 monomial terms we get of each size from assigning some subset of the larger parts corresponding to differences among $d_1,d_2,\dots,d_i$ to one endpoint of the edge, and all other larger parts to the opposite endpoint. Then we can remove all those monomials from the sum, and recursively learn $d_{i+1}$ and $c_{i+1}$ from looking at the largest remaining difference in part sizes that shows up in a length 2 monomial, since that difference will be $S-2d_i,$ and the coefficient will be $|E(G)|\cdot 2^{e+1}\cdot c_i.$ Thus, by induction, we can learn the full sequence of nonzero differences.
\end{proof}

\subsection{Self-CSFs for trees}\label{subsec:self-CSF_trees}

Now we focus in on self-CSFs of trees. The authors of \cite{EFHKY22} conjectured that the self-CSF distinguishes all trees from each other, i.e. $X_{T_1}^{T_1}\ne X_{T_2}^{T_2}$ whenever $T_1$ and $T_2$ are nonisomorphic trees. They verified this conjecture for trees on up to 9 vertices, and we check that it also holds for trees up to 12 vertices:

\begin{prop}\label{prop:12-vertex-trees}
    If $T_1$ and $T_2$ are nonisomorphic trees with up to 12 vertices, then $X_{T_1}^{T_1}\ne X_{T_2}^{T_2}.$
\end{prop}

\begin{proof}
    We were able to use Sage \cite{sage} to check all 10 to 12-vertex trees except for some of the ones where the computation timed out. Our program computes $X_T^T$ in $O(|\tn{End}(T)|)$ time, and the presence of a degree $d$ vertex in $T$ would introduce at least $d^d$ endomorphisms, thus many trees containing a high-degree vertex (usually $d\geq8$) would lead to a timeout. In light of this, we also checked the sizes of the automorphism groups (i.e. the coefficient of $m^n_{1^n}$) and the degree-squared sums of all the trees (which by Proposition \ref{prop:bipartite_degrees} is a quantity we can deduce from the self-CSF provided the tree's bipartition does not contain a part of size 1, but we can always recognize such a case since the only trees containing a part of size 1 are the stars, and stars are the unique trees with the maximal number of automorphisms amongst all trees with the same number of vertices), and these two pieces of information was able to distinguish all but two of the trees (pictured below) whose self-CSF computations timed out.
    $$T_1=\vcenter
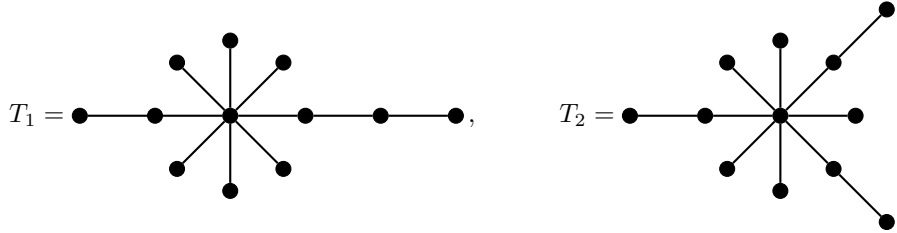
{\hbox{\begin{tikzpicture}
    \graph[nodes={draw, circle, fill=black, inner sep=2pt}, empty nodes, no placement, edges={thick}]{a[x=-2,y=0]--b[x=-1,y=0]--c[x=0,y=0]--d[x=1,y=0]--e[x=2,y=0]--f[x=3,y=0];g[x=-0.7071,y=0.7071]--c--h[x=0,y=1];i[x=0.7071,y=0.7071]--c--j[x=0.7071,y=-0.7071];k[x=0,y=-1]--c--l[x=-0.7071,y=-0.7071]};
    \end{tikzpicture}}}\,,\hspace{30pt}
    T_2=\vcenter{\hbox{\begin{tikzpicture}
    \graph[nodes={draw, circle, fill=black, inner sep=2pt}, empty nodes, no placement, edges={thick}]{a[x=-2,y=0]--b[x=-1,y=0]--c[x=0,y=0]--d[x=1,y=0];e[x=1.4142,y=1.4142]--f[x=0.7071,y=0.7071]--c--g[x=0.7071,y=-0.7071]--h[x=1.4142,y=-1.4142];i[x=0,y=1]--c--j[x=0,y=-1];k[x=-0.7071,y=0.7071]--c--l[x=-0.7071,y=-0.7071]};
    \end{tikzpicture}}}$$
    These two 12-vertex trees have the same degree sequence, so computing further power sums of the degrees will not distinguish them. We were able to verify that no other 12-vertex trees share the same number of automorphisms and degree-squared sum as these two trees, so it remains to distinguish $T_1$ from $T_2$. We were able to do this by manually computing the coefficient ratio $[m^{12}_{2,1^{10}}]/[m^{12}_{1^{12}}]$ for these two trees by utilizing Proposition \ref{prop:spider_m21n-2}, made possible by the fact that both $T_1$ and $T_2$ are spiders. Indeed, this ratio is 36 for $T_1$ and 43 for $T_2$.
\end{proof}

Another question one can ask is whether or not the self-CSF can distinguish trees from non-trees, and it turns out the answer is generally yes. In fact, the $X_{P_3}^{P_3} = X_{K_1\sqcup K_2}^{K_2\sqcup K_2}$ example actually turns out to be the \emph{only} case where the self-CSF fails to detect whether or not the graph is a tree:

\begin{prop}\label{prop:detect_trees}
    If $T$ is a tree and $X_T^T = X_G^G$, then $G$ is also a tree unless $T=P_3$ and $G=K_1\sqcup K_2.$
\end{prop}

\begin{proof}
    Suppose $T$ is a tree and $X_T^T = X_G^G$. From \cite{EFHKY22}, $G$ must be bipartite, or else $X_G^G$ would not have a length 2 monomial. Also, $|V(T)|=|V(G)|$, since $X_T^T$ and $X_G^G$ have the same degree. If $G$ is connected, then by Proposition \ref{prop:bipartite_degrees}, the sums of the degrees must match, so $$2\cdot |E(T)|=\sum_{v\in V(T)}\deg(v) = \sum_{v\in V(G)}\deg (v)=2\cdot |E(G)|.$$ Since $T$ is a tree, we get $|E(G)|=|E(T)|=|V(T)|-1=|V(G)|-1$. But then $G$ must also be a tree, since it is a connected graph with one fewer edge than it has vertices.

    Thus, assume $G$ is disconnected. The unique length 2 monomial appearing in $X_T^T$ must also be the only length 2 monomial appearing in $X_G^G$. This implies that $G$ cannot be the edgeless graph $\ol{K_n}$ unless $n=|V(T)|=|V(G)|\le 3$, since every length monomial $m_{n-k,k}^n$ shows up in $\ol{K_n}$. If $n=2,$ the only case to consider is $X_{K_2}^{K_2}$ and $X_{\ol{K_2}}^{\ol{K_2}}$, which are not equal because $X_{\ol{K_2}}^{\ol{K_2}}$ contains an $m_2^2$ term while $X_{K_2}^{K_2}$ does not, because for $\ol{K_2}$, both vertices can be mapped to the same vertex since there is no edge between them, while for $K_2$, they cannot because the edge needs to map to an edge. 
    
    If $n=3,$ we already know that $X_{P_3}^{P_3} = X_{K_1\sqcup K_2}^{K_2\sqcup K_2}$, and we can also check that $X_{P_3}^{P_3}\ne X_{\ol{K_3}}^{\ol{K_3}}$, because $X_{\ol{K_3}}^{\ol{K_3}}$ contains an $m_3^3$ term while $X_{P_3}^{P_3}$ does not, since in $\ol{K_3}$, the 3 vertices can all map onto one vertex, while in $P_3,$ they cannot.
    
    Thus, suppose $n\ge 4,$ and let the unique length 2 monomial in $X_T^T=X_G^G$ be $m_{n-k,k}^n$. Then $$[m_{n-k,k}^n]X_T^T = |E(T)|\cdot 2,$$ since this is the number of ways to choose an edge to map all the edges of $T$ onto, and then to choose which endpoint of that edge to map the vertices in one of the parts of the bipartition of $T$ onto (with the vertices in the other part of the bipartition then mapping onto the other endpoint). Also, $$[m_{n-k,k}^n]X_G^G=|E(G)|\cdot 2^{\kappa(G)},$$ where $\kappa(G)$ is the number of connected components of $G$, because all the edges of $G$ must map onto the same edge, and then for each connected component of size 2 or more, we have 2 choices to choose which endpoint of that edge to map each part of its bipartition onto, while for connected components of size 1, we also have 2 choices for which endpoint of the edge to map the lone vertex onto. Since $|E(T)|=n-1$, we get $$(n-1)\cdot 2 = |E(G)|\cdot 2^{\kappa(G)}.$$ Since $\kappa(G)>1,$ we must have $|E(G)|<n-1,$ so let $|E(G)|=n-\ell,$ with $\ell>1.$ Plugging that in and dividing by 2 gives $$n-1 = (n-\ell)\cdot 2^{\kappa(G)-1}.$$ Solving for $n$ gives $$\ell\cdot 2^{\kappa(G)-1}-1=n\cdot(2^{\kappa(G)-1}-1) \implies n = \frac{\ell\cdot 2^{\kappa(G)-1}-1}{2^{\kappa(G)-1}-1}=\ell + \frac{\ell-1}{2^{\kappa(G)-1}-1}.$$ Since $n$ is an integer, we must have $(2^{\kappa(G)-1}-1)\mid (\ell-1)$, so since $\ell>1,$ $$2^{\kappa(G)-1}-1\le \ell-1 \implies 2^{\kappa(G)-1}\le \ell.$$ But we must also have $n-\ell+\kappa(G)\ge n$, since if we start with the edgeless graph $\ol{K_n}$ and add the $n-\ell$ edges one by one, each additional edge can decrease the number of connected components by at most one, so the number of edges plus the number of connected components is at least $n$. Thus, $\ell \le \kappa(G)$, so $2^{\kappa(G)-1}\le \ell\le \kappa(G).$ This can only hold if $\kappa(G)\le 2$, since $2^{\kappa(G)-1}>\kappa(G)$ if $\kappa(G)\ge 3$. If $\kappa(G)=1$, then $G$ is a tree and we are done. If $\kappa(G)=2,$ then $2^{\kappa(G)-1}=2^{2-1}=2$, so we must also have $\ell=2$. Then solving for $n$ gives $$n = \ell + \frac{\ell-1}{2^{\kappa(G)-1}-1} = 2 + \frac{2-1}{2^{2-1}-1} = 3,$$ which takes us back to the case $G=P_3$ and $G=K_1\sqcup K_2.$ Thus, in all other cases, $G$ must be a tree.
\end{proof}

Next, we consider what information about a tree we can learn from some of its other self-CSF coefficients. First we give a general interpretation of the coefficients of any $H$-CSF, which will help us recover some information related to the leaves in the case of trees:

\begin{prop}\label{prop:coeff_in_general}
    For a set partition $\pi$ of $V(G)$, let $\tn{type}(\lam)$ be the integer partition of $|V(G)|$ denoting the size of its parts, and let $G/\pi$ denote the graph obtained by identifying vertices in $G$ whenever they come from the same part of $\pi$. Then $$[m^n_\lam]X_G^H=\sum_{\substack{\pi\vdash V(G)\\\tn{type}(\pi)=\lam\\}}S(G/\pi,H)\big|\tn{Aut}(G/\pi)\big|$$ where $S(G/\pi,H)$, defined similarly to in Proposition 3.21 of \cite{EFHKY22}, is the number of subgraphs of $H$ isomorphic to $G/\pi$.
\end{prop}

\begin{proof}
    Recall that $[m^n_\lam]X_G^H$ counts the elements of $\tn{Hom}(G,H)$ of type $\lam$. Observe that $f\in\tn{Hom}(G,H)$ has type $\lam$ if and only if there exists $\pi\vdash V(G)$ of type $\lam$ such that $f=\psi\circ g_\pi$ where:
    \begin{itemize}
        \item $g_\pi$ is the map $G\rightarrow G/\pi$ obtained by identifying the required vertices, and
        \item $\psi:G/\pi\rightarrow H$ is an embedding (injective homomorphism).
    \end{itemize}
    By Proposition 3.21 of \cite{EFHKY22}, the number of embeddings $\psi:G/\pi\rightarrow H$ is $S(G/\pi,H)\big|\tn{Aut}(G/\pi)\big|$. Summing over all possibilities of $\pi$ (or equivalently of $g_\pi$), the result follows.
\end{proof}

\begin{remark}
    Proposition 3.21 of \cite{EFHKY22} is the special case of the above with $\lam=1^{|V(G)|}$. Indeed, there is only one such $\pi$ in this case, for which $G/\pi=G$.
\end{remark}

For some special cases, there are nice necessary and/or sufficient conditions for $S(G/\pi,H)=0$ that allows us to be more specific about the sum above. For instance, going back to the setting of self-CSFs of trees, we obtain the following:

\begin{lemma}\label{lem:m21n-2_XTT}
    Let $T$ be a tree with $n$ vertices and let $T/uw$ denote the graph obtained by identifying the vertices $u$ and $w$ from $T$. Then $$[m^n_{2,1^{n-2}}]X_T^T=\sum_{\substack{\{u,w\}\subset V(T)\\d(u,w)=2}}S(T/uw,T)\big|\tn{Aut}(T/uw)\big|,$$ where $d(u,w)$ is the distance between $u$ and $w$, i.e. the minimum length of a path from $u$ to $w$ in $T$.
\end{lemma}

\begin{proof}
    Let $f\in\tn{End}(T)=\tn{Hom}(T,T)$. For $f(T)$ to have type $(2,1^{n-2})$, two distinct vertices, say $u$ and $w$, need to map to the same vertex through $f$, with $f$ otherwise being injective. As such, choosing $\pi\vdash V(G)$ such that $\tn{type}(\pi)=(2,1^{n-2})$ is equivalent to choosing $u$ and $w$. By Proposition \ref{prop:coeff_in_general}, it thus suffices to show that $d(u,w)=2$ is a necessary condition for there to exist valid endomorphisms $f$ of $T$ (i.e. to show that $S(T/uw,T)=0$ whenever $d(u,w)\neq2$). Indeed, $d(u,w)$ cannot be 1, since adjacent vertices cannot map onto the same vertex. On the other hand, $d(u,w)$ cannot exceed 2, for if otherwise, the path from $u$ to $w$ would map to a cycle of length $d(u,w)$ under $f$, which is impossible since $f(T)$ must also be a tree.
\end{proof}
In fact, we can be even more specific. Besides the distance 2 requirement between $u$ and $w$, we can also say a lot about what their degrees have to be, as well as the degree of their common neighbor:

\begin{prop}\label{prop:m21n-2_XTT}
    Let $T$ be an $n$-vertex tree and $D=L\sqcup J$, where $L=\{(1,a,b),(b,a,1)\mid a,b\in\mathbb N,\:a\geq2\}$ and $J=\{(k,k+1,2),(2,k+1,k)\mid k\in\mathbb N,\:k\geq2\}$. Then $$[m^n_{2,1^{n-2}}]X_T^T=\sum_{\substack{\{u,v,w\}\subset V(T)\\d(u,v)=d(v,w)=1\\(\deg(u),\deg(v),\deg(w))\in D}}S(T/uw,T)\big|\tn{Aut}(T/uw)\big|.$$
\end{prop}

\begin{proof}
Let $f\in\tn{End}(T)$ have type $(2,1^{n-1})$. Then the preimage under $f$ of one of the vertices, say $u'$, is empty, while another vertex, say $w'$, has preimage of size 2, whose elements we denote by $u$ and $w$. As established in Lemma \ref{lem:m21n-2_XTT}, $d(u,w)$ must equal 2, and therefore, they share a neighbor (which must be unique since $T$ is a tree) which we denote by $v$, and whose image under $f$ we denote by $v'$. Consider the subgraph of $T$ that includes all vertices but only the edges that are the image of an edge in $T$ under $f$. We see that in this subgraph:
\begin{itemize}
    \item $\deg(u')=0$ since $f^{-1}(u')=\emptyset$; 
    \item $\deg(v')=\deg(v)-1$ since both edges $uv$ and $vw$ get mapped to a single edge $v'w'$ and all other edges coming out of $v$ are unaffected; and
    \item $\deg(w')=\deg(u)+\deg(w)-1$ since the edges originally coming out of either $u$ or $w$ now come out of the same vertex, but the two edges that originally connected them to $v$ are now merged into one. 
\end{itemize}
Since $f(T)$ is a subgraph of $T$ with one less vertex and one less edge, we need the degree sequence of $f(T)$ to look like that of $T$, but with two vertices having their degrees lowered by 1, one of which must be lowered to zero since we are also removing a vertex (namely $u'$) from $T$ through the endomorphism. As such, $f(T)$ must be isomorphic to $T$ with a leaf (degree 1 vertex) removed, and the neighbor of that leaf must be the other vertex whose degree was lowered. Suppose this neighbor originally has degree $d$. Then the only changes to the degree sequence are $\{d,1\}\mapsto\{d-1,0\}$. Since $u$, $v$, and $w$ are the only vertices that might have their degrees changed through $f$, we need the move
$$\{\deg(u),\deg(v),\deg(w)\}\mapsto\{0,\deg(v)-1,\deg(u)+\deg(w)-1\}$$
to encompass these changes and these changes only. We now have two cases depending on the value of $d$. \\[8pt]
\tb{Case 1}: $d>2$. Then $d-1\neq1$, so we need the following (multi)set relations to hold:
\begin{itemize}
    \item [(i)] $\{d,1\}\subset\{\deg(u),\deg(v),\deg(w)\}$,
    \item [(ii)] $\{d-1,0\}\subset\{0,\deg(v)-1,\deg(u)+\deg(w)-1\}$, and
    \item [(iii)] $\{\deg(u),\deg(v),\deg(w)\}\setminus\{d,1\}=\{0,\deg(v)-1,\deg(u)+\deg(w)-1\}\setminus\{d-1,0\}$.
\end{itemize}
Since $v$ is a neighbor of at least two vertices (namely $u$ and $w$), we cannot have $\deg(v)=1$, so without loss of generality (no generality is lost because $u$ and $w$ play symmetric roles) suppose $\deg(u)=1$. Then conditions (i) and (ii) become the following:
\begin{itemize}
    \item [(a)] $d\in\{\deg(v),\deg(w)\}$, 
    \item [(b)] $d-1\in\{\deg(v)-1,\deg(w)\}$. 
\end{itemize}
These both hold if and only if $d=\deg(v)$, and we see that in all such cases, (iii) still holds no matter the values of $\deg(v)$ and $\deg(w)$ since the LHS and RHS both equal $\{\deg(w)\}$. Hence the only restriction we obtain from this case is $\deg(u)=1$, i.e. for one of the two vertices being identified to be a leaf. \\[8pt]
\tb{Case 2}: $d=2$. Then $d-1=1$, so only a single change, namely $d\mapsto0$, in the degree sequence is required for the set $\{d,1\}$ to map to $\{d-1,0\}$, though the case above is still a possibility. The new possibility requires:
\begin{itemize}
    \item [(i)] $2=d\in\{\deg(u),\deg(v),\deg(w)\}$,
    \item [(ii)] $0\in\{0,\deg(v)-1,\deg(u)+\deg(w)-1\}$ (which is automatic), and
    \item [(iii)] $\{\deg(u),\deg(v),\deg(w)\}\setminus\{d\}=\{0,\deg(v)-1,\deg(u)+\deg(w)-1\}\setminus\{0\}$ as multisets.
\end{itemize}
For (iii), notice that $\deg(v)-1$ must equal $\deg(u)$ or $\deg(w)$ since it cannot equal $\deg(v)$, so without loss of generality suppose that $\deg(v)-1=\deg(u)$. Condition (iii) thus reduces to
\begin{itemize}
    \item [(iii$'$)] $\{\deg(v),\deg(w)\}\setminus\{d\}=\{\deg(v)+\deg(w)-2\}$.
\end{itemize}
But regardless of whether $\deg(v)=d$ or $\deg(w)=d$, since $d=2$, (iii$'$) always holds. So to conclude, this edge case allows for the following extra possibilities for $(\{\deg(u),\deg(w)\},\deg(v))$:
\begin{itemize}
    \item $(\{1,\tn{anything}\},2)$, which is already covered in Case 1, so not a new possibility; and
    \item $(\{\deg(u),2\},\deg(u)+1)$ where $\deg(u)$ can be anything, which is a new possibility for any choice of $\deg(u)$ greater than 1.
\end{itemize}
Combining both cases, we see that the set of all possible values of $(\deg(u),\deg(v),\deg(w))$ is precisely the set $D$, with Case 1 corresponding to $L$ and Case 2 corresponding to $J$.
\end{proof}
\begin{cor}\label{cor:m21n-2_XTT}
    Let $T$ be a tree with $n$ vertices. If $\ell$ is a leaf of $T$, denote by $\tn{nei}(\ell)$ its unique neighbor and $T-\ell$ the subtree of $T$ obtained by deleting $\ell$ and the edge containing $\ell$ from $T$. Then, with $J$ defined as in Proposition \ref{prop:m21n-2_XTT}, we have: \begin{align*}[m^n_{2,1^{n-2}}]X_T^T&=\sum_{\substack{\{u,v,w\}\subset V(T)\\d(u,v)=d(v,w)=1\\(\deg(u),\deg(v),\deg(w))\in J}}S(T/uw,T)\big|\tn{Aut}(T/uw)\big|\\
    &+\bigg(\hspace{-2pt}\sum_{\substack{\ell\in V(T)\\\tn{deg}(\ell)=1}}\hspace{-4pt}(\tn{deg}(\tn{nei}(\ell))-1)S(T-\ell,T)\big|\tn{Aut}(T-\ell)\big|
    -\hspace{-10pt}\sum_{\substack{\{\ell_1,\ell_2\}\subset V(T)\\d(\ell_1,\ell_2)=2\\\tn{deg}(\ell_1)=\tn{deg}(\ell_2)=1}}\hspace{-10pt}S(T-\ell_1,T)\big|\tn{Aut}(T-\ell_1)\big|\bigg).\end{align*}
\end{cor}
\begin{proof}
    By splitting $D$ into $L\sqcup J$, the sum in the statement of Proposition \ref{prop:m21n-2_XTT} can be split into the sums over $(\deg(u),\deg(v),\deg(w))\in L$ and $(\deg(u),\deg(v),\deg(w))\in J$ respectively, so it suffices to show that former corresponds to the terms in parentheses in this corollary. Indeed, a choice of $u$, $v$ and $w$ whose corresponding degrees belong to $L$ is equivalent to a choice of a leaf $\ell$ (since one of $\{u,w\}$ must have degree 1) and a choice of a grandneighbor thereof. For a fixed leaf $\ell$, it has one neighbor $\tn{nei}(\ell)$, who in turn has $\deg(\tn{nei}(\ell))$ neighbors, 1 of which is $\ell$ itself and the other $\deg(\tn{nei}(\ell))-1$ of which are the grandneighbors of $\ell$.  This procedure would double-count the case where said grandneighbor is also a leaf, so we correct it by deducting one instance thereof, hence the subtracted sum. The resulting tree $T/uw$ is isomorphic to $T-\ell$, so the result follows.
\end{proof}

\subsection{Self-distinguishability of spiders}\label{subsec:spiders}

Using the results above, we can now make some partial progress towards proving that the self-CSF distinguishes a family of trees known as spiders.

A \emph{\tb{\tcb{spider}}} is a tree containing exactly one vertex of degree 3 or more. Following the notation of \cite{MMW08}, a spider can be thought of as a union of edge-disjoint path graphs (known as the \emph{\tb{\tcb{legs}}}) joined at a common vertex $t$ (the \emph{\tb{\tcb{torso}}}). As such, an $n$-vertex spider can be uniquely identified by a partition $\lam\vdash n-1$ whose parts are the lengths of its legs, and we let $T_\lam$ denote this spider. In particular, since the torso must have degree 3 or more, we have $\ell(\lam)\geq3$.

\begin{prop}\label{prop:spider_legs}
    $X_{T_\lam}^{T_\lam}$ determines the number of legs $\ell(\lam)$ of a spider $T_\lam$.
\end{prop}
\begin{proof}
    The degree sequence of $T_\lam$ is $(\ell(\lam),2^{n-1-\ell(\lam)},1^{\ell(\lam)})$ where $n=|V(T_\lam)|$, and in particular, this is determined by the single parameter $\ell(\lam)$. We know from Proposition \ref{prop:bipartite_degrees} that the self-CSF determines $\sum_{v\in T_\lam}\tn{deg}(v)^i$ for every $1\leq i\leq k$ where $k$ is the size of the smaller part in the bipartition of $T_\lam$. In particular, for $i=2$, this sum is $\ell(\lam)^2+4(n-1-\ell(\lam))+\ell(\lam)=\ell(\lam)^2-3\ell(\lam)+4(n-1)$. This is an increasing, and therefore injective, function in $\ell(\lam)$ when $\ell(\lam)\geq3$, and hence uniquely determines $\ell(\lam)$. We can deduce this information so long as $k\geq2$, which is true for every spider except for the star. However, we could easily identify such a case, since the star is the only tree that has an $m^n_{n-1,1}$ term in its self-CSF, as well as the only spider with $\ell(\lam)=n-1$.
\end{proof}

\begin{prop}\label{prop:spider_auts}
    For an $n$-vertex spider $T_\lam$, $[m^n_{1^n}]X_{T_\lam}^{T_\lam}=\prod_{i=1}^{n-1}r_i(\lam)!$.
\end{prop}
\begin{proof}
    Recall that $[m^n_{1^n}]X_{T_\lam}^{T_\lam}=|\tn{Aut}(T_\lam)|$. An automorphism necessarily preserves degrees, so a path beginning at a leaf (the only kind of vertex with degree 1) and ending at the torso (the only one with degree 3 or more) must map to a path beginning at a leaf and ending at the torso, and as such, a leg must map to a leg. Furthermore, as automorphisms preserve lengths of paths, a leg can only map to a leg of the same length. Since permuting legs of the same length leaves the structure of $T_\lam$ unchanged, all such permutations, and compositions thereof, are valid automorphisms. With $r_i(\lam)$ legs of length $i$, there are $r_i(\lam)!$ ways to permute these legs. Since legs of different lengths permute independently, the total number of automorphisms is the product of the $r_i(\lam)$'s.
\end{proof}

\begin{prop}\label{prop:spider_m21n-2}
    For an $n$-vertex spider $T_\lam$ with $\ell(\lam)\geq4$, we have
    \begin{align*}\frac{[m^n_{2,1^{n-2}}]X_{T_\lam}^{T_\lam}}{[m^n_{1^n}]X_{T_\lam}^{T_\lam}}=r_1(\lam)\ell(\lam)+\ell(\lam)-\frac32r_1(\lam)-\frac12r_1(\lam)^2+\sum_{j=2}^{n-1}r_{j-1}(\lam)r_j(\lam).\end{align*}
\end{prop}
\begin{proof}
    Suppose a leaf $v$ (we avoid the notation $\ell$ to prevent confusion with the length $\ell(\lam)$ of the partition $\lam$) belongs to a leg of length $\lam_i$. Then, using the notation of Corollary \ref{cor:m21n-2_XTT}, we see that:
    \begin{itemize}
        \item $\deg(\tn{nei}(v))-1=
        \begin{cases}
            1&\tn{if }\lam_i>1,\\
            \ell(\lam)-1&\tn{if }\lam_i=1,
        \end{cases}$
        \item $S(T_\lam-v,T_\lam)=r_{\lam_i}(\lam)$, and
        \item $|\tn{Aut}(T_\lam-v)|=\displaystyle\prod_{j=1}^{n-1}r_j(\lam^-)!$ where for each $1\leq j\leq\ell(\lam)$ we have $\lam_j^-=
        \begin{cases}
            \lam_j&\tn{if }j\neq i,\\
            \lam_j-1&\tn{if }j=i,
        \end{cases}$\\
        or equivalently, for each $1\leq j\leq n-1$, $r_j(\lam^-)=
        \begin{cases}
            r_j(\lam)&\tn{if }j\notin\{\lam_i,\lam_i-1\},\\
            r_j(\lam)+1&\tn{if }j=\lam_i-1\tn{ with }\lam_i\geq2,\\
            r_j(\lam)-1&\tn{if }j=\lam_i.
        \end{cases}$
    \end{itemize}
    Since each leaf uniquely identifies the leg it belongs to and vice versa, summing over the leaves $v$ is equivalent to summing over the parts $\lam_i$ corresponding to each leg. The observations above allow us to perform such a reindexing of the sum through expressing the functions of the $v$'s in terms of the $\lam_i$'s. For $\lam_i\geq2$, note that
    \begin{align*}
        r_{\lam_i}(\lam)\prod_{j=1}^{n-1}r_j(\lam^-)!
        &=\tcb{r_{\lam_i}(\lam)}\left(\prod_{j=1}^{\lam_i-2}r_j(\lam)!\right)\textcolor{red}{(r_{\lam_i-1}(\lam)+1)!}\tcb{(r_{\lam_i}(\lam)-1)!}\left(\prod_{j=\lam_i+1}^{n-1}r_j(\lam)!\right)
        \\&=\textcolor{red}{(r_{\lam_i-1}(\lam)+1)}\left(\prod_{j=1}^{\lam_i-2}r_j(\lam)!\right)\textcolor{red}{r_{\lam_i-1}(\lam)!}\,\tcb{r_{\lam_i}(\lam)!}\left(\prod_{j=\lam_i+1}^{n-1}r_j(\lam)!\right)
        \\&=(r_{\lam_i-1}(\lam)+1)\prod_{j=1}^{n-1}r_j(\lam)!
        =(r_{\lam_i-1}(\lam)+1)[m^n_{1^n}]X_{T_\lam}^{T_\lam},
    \end{align*}
    and for $\lam_i=1$,
    \begin{equation*}
        r_{\lam_i}(\lam)\prod_{j=1}^{n-1}r_j(\lam^-)!
        =r_1(\lam)(r_1(\lam)-1)!\prod_{j=2}^{n-1}r_j(\lam)!
        =r_1(\lam)!\prod_{j=2}^{n-1}r_j(\lam)!
        =\prod_{j=1}^{n-1}r_j(\lam)!
        =[m^n_{1^n}]X_{T_\lam}^{T_\lam}.
    \end{equation*}
    Combining the above, we get
    \begin{align*}
        &\sum_{\substack{v\in V(T_\lam)\\\tn{deg}(v)=1}}\hspace{-4pt}(\tn{deg}(\tn{nei}(v))-1)S(T_\lam-v,T_\lam)\big|\tn{Aut}(T_\lam-v)\big|
        -\hspace{-10pt}\sum_{\substack{\{v_1,v_2\}\subset V(T_\lam)\\d(v_1,v_2)=2\\\tn{deg}(v_1)=\tn{deg}(v_2)=1}}\hspace{-10pt}S(T_\lam-v_1,T_\lam)\big|\tn{Aut}(T_\lam-v_1)\big|
        \\&=\sum_{\substack{1\leq i\leq\ell(\lam)\\\lam_i\geq2}}(r_{\lam_i-1}(\lam)+1)[m^n_{1^n}]X_{T_\lam}^{T_\lam}
        +\sum_{\substack{1\leq i\leq\ell(\lam)\\\lam_i=1}}(\ell(\lam)-1)[m^n_{1^n}]X_{T_\lam}^{T_\lam}
        -\sum_{\substack{1\leq i<j\leq\ell(\lam)\\\lam_i=\lam_j=1}}[m^n_{1^n}]X_{T_\lam}^{T_\lam}
        \\&=\bigg(\ell(\lam)-r_1(\lam)+\sum_{\substack{1\leq i\leq\ell(\lam)\\\tcb{\lam_i}\geq2}}r_{\tcb{\lam_i}-1}(\lam)+r_1(\lam)(\ell(\lam)-1)-\binom{r_1(\lam)}{2}\bigg)[m^n_{1^n}]X_{T_\lam}^{T_\lam}
        \\&\hspace{70pt}\tn{now substitute }\tcb{j=\lam_i}\tn{ in the sum}
        \\&=\bigg(r_1(\lam)\ell(\lam)+\ell(\lam)-\frac32r_1(\lam)-\frac12r_1(\lam)^2+\sum_{\tcb{j}\geq2}r_{\tcb{j}-1}(\lam)\textcolor{red}{\#\{1\leq i\leq\ell(\lam)\mid\lam_i=j\}}\bigg)[m^n_{1^n}]X_{T_\lam}^{T_\lam}
        \\&=\bigg(r_1(\lam)\ell(\lam)+\ell(\lam)-\frac32r_1(\lam)-\frac12r_1(\lam)^2+\sum_{j=2}^{n-1}r_{j-1}(\lam)\textcolor{red}{r_j(\lam)}\bigg)[m^n_{1^n}]X_{T_\lam}^{T_\lam}.
    \end{align*}
    Lastly, we claim that when $\ell(\lam)\geq4$, the first term in Corollary \ref{cor:m21n-2_XTT} vanishes, i.e.
    $$\sum_{\substack{\{u,v,w\}\subset V(T_\lam)\\d(u,v)=d(v,w)=1\\(\deg(u),\deg(v),\deg(w))\in J}}\hspace{-15pt}S(T_\lam/uw,T_\lam)\big|\tn{Aut}(T_\lam/uw)\big|=0.$$
    Indeed, $T_\lam$ contains only one vertex of degree exceeding 2, but the only element of $J$ with at most one coordinate exceeding 2 is $(2,3,2)$. But when $\ell(\lam)\geq4$, there is no vertex of degree 3 in $T_\lam$, meaning that the set being summed over is empty. Hence, applying Corollary \ref{cor:m21n-2_XTT}, the result follows.
\end{proof}

\begin{prop}\label{prop:spider_endos}
    Let $T_\lam$ be a spider, $t$ be its torso, and $v_1,\dots,v_{\ell(\lam)}$ be its leaves such that the leg $L_i$ containing the leaf $v_i$ has length $\lam_i$. Then
    $$|\tn{End}(T_\lam)|=\sum_{u\in V(T_\lam)}\prod_{i=1}^{\ell(\lam)}\sum_{w_i\in V(T_\lam)}\#\{f\in\tn{Hom}(L_i,T_\lam)\mid f(t)=u,\,f(v_i)=w_i\}.$$
\end{prop}
\begin{proof}
    An endomorphism $f\in\tn{End}(T_\lam)$ is uniquely specified by the image $f(t)$ of the torso $t$, the images $f(v_i)$ of the leaves $v_i$, and the images of the remaining vertices $V(L_i)\setminus\{t,v_i\}$ of the legs $L_i$ for each $i$. If we fix $f(t)=u$ and $f(v_i)=w_i$, then there are precisely $\#\{f\in\tn{Hom}(L_i,T_\lam)\mid f(t)=u,\,f(v_i)=w_i\}$ ways to choose the image of the remainder of $L_i$, and since these choices are independent for each $i$ (due to the fact that distinct legs are disconnected if we remove the torso), taking their product over all the $i$'s gives the total number of endomorphisms where $f(t)=u$ and $f(v_i)=w_i$ for each $i$. Summing over all choices of $u,w_1,\dots,w_{\ell(\lam)}$, we see that the total number of endomorphisms is
    $$|\tn{End}(T_\lam)|=\sum_{u\in V(T_\lam)}\sum_{w_1\in V(T_\lam)}\cdots\sum_{w_{\ell(\lam)}\in V(T_\lam)}\prod_{i=1}^{\ell(\lam)}\#\{f\in\tn{Hom}(L_i,T_\lam)\mid f(t)=u,\,f(v_i)=w_i\}.$$
    Since all sums and products are finite, we can reorder them by using the operator identity $$\sum_{w_1\in V(T_\lam)}\cdots\sum_{w_{\ell(\lam)}\in V(T_\lam)}\prod_{i=1}^{\ell(\lam)}=\prod_{i=1}^{\ell(\lam)}\sum_{w_i\in V(T_\lam)}$$ which follows by distributivity of multiplication over addition. Hence the result follows.
\end{proof}

To see how the result above helps us, we begin by observing the following linear-algebraic interpretation.

\begin{lemma}\label{lem:A1u}
    Let $A$ be the adjacency matrix of $T_\lam$ and let $\mathbf{1}$ be the column vector of the appropriate length whose entries are all ones. Then $$\sum_{w_i\in V(T_\lam)}\#\{f\in\tn{Hom}(L_i,T_\lam)\mid f(t)=u,\,f(v_i)=w_i\}=(A^{\lam_i}\mathbf{1})_u.$$
\end{lemma}

\begin{proof}
    Observe that the quantity being summed over is precisely the number of walks of length $\lam_i$ (which is the length of the leg $L_i$) starting from vertex $u$ and ending at vertex $w_i$. By a classical result in algebraic graph theory (see e.g. Lemma 8.1.2 of \cite{godsilroyle}), this is precisely the $(u,w_i)^\tn{th}$ component of the $\lam_i^\tn{th}$ power of the adjacency matrix. For fixed $u$, summing this quantity over all the $w_i$'s is equivalent to summing all the entries of the $u^\tn{th}$ row of $A^{\lam_i}$, but that is the same as taking the $u^\tn{th}$ component of the vector resulting from postmultiplying by $\mathbf{1}$.
\end{proof}

Unfortunately, an exact expression for $A^{\lam_i}\tb{1}$ would be too complicated to be useful, but we can still approximate it by utilizing the eigendecomposition of $A$. To begin, we need a few results from spectral graph theory. The \emph{\tb{\tcb{spectrum}}} of a graph refers to the multiset of eigenvalues of its adjacency matrix, whereas its \emph{\tb{\tcb{spectral radius}}} refers to the maximum of the absolute values of said eigenvalues.

\begin{lemma}
    Let $G$ be an undirected graph and $A$ its adjacency matrix. Then $A$ is orthogonally diagonalizable (in particular, it has real eigenvalues and we can create an orthonormal basis of eigenvectors).
\end{lemma}
\begin{proof}
    $A$ is symmetric whenever $G$ is undirected.
\end{proof}

\begin{lemma}[cf. \cite{godsilroyle}, \S8.8] \label{lem:bipartite_spectra}
    Let $G$ be a bipartite graph with $A$ its adjacency matrix. For a vector $x$ indexed by the vertex set of $G$, write $x_0$ and $x_1$ to be the subvectors of $x$ whose indices belong to the two parts $V_0$ and $V_1$ respectively of the bipartition of $V(G)$. If $\alpha$ is an eigenvalue of $A$ with corresponding eigenvector $x$, then $-\alpha$ is an eigenvalue with corresponding eigenvector $y$ where $y_0=x_0$ and $y_1=-x_1$. In particular, this means that the multiset of nonzero eigenvalues of $A$ occur in pairs that differ only by sign, and so $\alpha$ and $-\alpha$ must have the same multiplicity. Furthermore, we have $x_0^Tx_0=x_1^Tx_1=\frac12x^Tx$ in such a case, i.e. the two parts in the bipartition contribute equally to the $l^2$-norm of any eigenvector.
\end{lemma}

\begin{proof}
    If we order the vertex set so that all the vertices in $V_0$ are listed before the vertices of $V_1$, the adjacency matrix would have the block form
    $$A=\begin{pmatrix}
        0_{|V_0|\times|V_0|} & B \\[4pt]
        B^T & 0_{|V_1|\times|V_1|}
    \end{pmatrix}$$
    where $B$ is a $|V_0|\times|V_1|$ matrix, so if $Ax=\alpha x$, then
    $$\begin{pmatrix}
        0_{|V_0|\times|V_0|} & B \\[4pt]
        B^T & 0_{|V_1|\times|V_1|}
    \end{pmatrix}\begin{pmatrix}
        x_0 \\[4pt] x_1
    \end{pmatrix}=\begin{pmatrix}
        Bx_1 \\[4pt] B^T x_0
    \end{pmatrix}=\begin{pmatrix}
        \alpha x_0 \\[4pt] \alpha x_1
    \end{pmatrix},$$
    and thus
    $$\begin{pmatrix}
        0_{|V_0|\times|V_0|} & B \\[4pt]
        B^T & 0_{|V_1|\times|V_1|}
    \end{pmatrix}\begin{pmatrix}
        x_0 \\[4pt] -x_1
    \end{pmatrix}=\begin{pmatrix}
        -Bx_1 \\[4pt] B^T x_0
    \end{pmatrix}=\begin{pmatrix}
        -\alpha x_0 \\[4pt] \alpha x_1
    \end{pmatrix}=-\alpha\begin{pmatrix}
        x_0 \\[4pt] -x_1
    \end{pmatrix},$$
    i.e. $Ay=-\alpha y$ as required. Lastly, observe that 
    $$\alpha x^Tx=x^TAx=x_0^TBx_1+x_1^TB^Tx_0=2x_0^TBx_1=2x_0^T(\alpha x_0)=2\alpha x_0^Tx_0,$$
    from which $x_0^Tx_0=x_1^Tx_1=\frac12x^Tx$ follows.
\end{proof}

\begin{lemma}
    Let $G$ be a graph, $A$ its adjacency matrix, and $\rho(A)$ its spectral radius. Then $\rho(A)$ is a simple eigenvalue (i.e. eigenvalue with multiplicity 1) of $A$ whose corresponding eigenvector can be made to have strictly positive entries.
\end{lemma}

\begin{proof}
    This is a variant of the Perron-Frobenius theorem; see Theorem 8.8.1 of \cite{godsilroyle}.
\end{proof}

Combining the three lemmata above, we conclude the following:

\begin{cor}\label{cor:spectral_radius_eigendecomp}
    Let $G$ be a bipartite graph with $A$ its adjacency matrix and $\rho(A)$ its spectral radius. Then $\rho(A)$ and $-\rho(A)$ are the only eigenvalues of $A$ of maximum absolute value, and both have multiplicity 1. Furthermore, following the notation of Lemma \ref{lem:bipartite_spectra}, if $x=(x_0,x_1)$ is the eigenvector with eigenvalue $\rho(A)$, then $(x_0,-x_1)$ is the eigenvector with eigenvalue $-\rho(A)$, and without loss of generality, we can make $x$ be a unit vector with strictly positive entries, and in such a case, we have $\|x_0\|_2=\|x_1\|_2=2^{-1/2}$.
\end{cor}

The whole reason for our interest in and extensive analysis of the spectral radius comes from the observation that when a matrix is raised to a large power, the result is dominated in magnitude by the terms containing the largest eigenvalue(s). Indeed, this is the basis for our approximation below:

\begin{prop}\label{prop:spider_endos_eigenexpansion}
    For an $n$-vertex spider $T_\lam$ with adjacency matrix $A$ and spectral radius $\rho$, we have 
    \begin{align*}
        |\tn{End}(T_\lam)|&=2^{\ell(\lam)}\rho^{n-1}\big(\|x_0\|_{\ell(\lam)}^{\ell(\lam)}\|x_0\|_1^{e(\lam)}\|x_1\|_1^{o(\lam)}+\|x_1\|_{\ell(\lam)}^{\ell(\lam)}\|x_0\|_1^{o(\lam)}\|x_1\|_1^{e(\lam)}\big)+o(\rho^{n-1})
    \end{align*}
    where $\|\cdot\|_p$ denotes the $p$-norm of a vector, $e(\lam)$ and $o(\lam)=\ell(\lam)-e(\lam)$ denote the number of even and odd parts respectively in $\lam$, while $x=(x_0,x_1)$ and $y=(x_0,-x_1)$ are unit eigenvectors of $A$ corresponding to the eigenvalues $\rho$ and $-\rho$ respectively. Once again, $x_0$ and $x_1$ denote the components of $x$ indexed by vertices in $V_0$ and $V_1$ respectively, the two parts in the bipartition of $V(T_\lam)$.
\end{prop}

\begin{proof}
    Expressing the all-ones vector $\tb{1}$ in the eigenvector basis, we see that the coefficient of an eigenvector $x$ is the component of $\tb{1}$ along $x$, namely $(x^Tx)^{-1}\tb{1}^Tx$, but since we are assuming without loss of generality that $x$ is a unit vector, we can omit the division by $x^Tx$. Using this, we see that
    $$A^{\lam_i}\tb{1}=A^{\lam_i}((\tb{1}^Tx)x+(\tb{1}^Ty)y+\cdots)=\rho^{\lam_i}(\tb{1}^Tx)x+(-\rho)^{\lam_i}(\tb{1}^Ty)y+\cdots,$$
    and that the terms omitted by the ellipsis are of order $o(\rho^{\lam_i})$ by Corollary \ref{cor:spectral_radius_eigendecomp}. As such, by Proposition \ref{prop:spider_endos} and Lemma \ref{lem:A1u}, we have
    \begin{align*}
        |\tn{End}(T_\lam)|&=\sum_{u\in V(T_\lam)}\prod_{i=1}^{\ell(\lam)}(A^{\lam_i}\tb{1})_u=\sum_{u\in V(T_\lam)}\prod_{i=1}^{\ell(\lam)}\rho^{\lam_i}\big((\tb{1}^Tx)x_u+(-1)^{\lam_i}(\tb{1}^Ty)y_u+o(1)\big)\\
        &=\rho^{\sum_{i=1}^{\ell(\lam)}\lam_i}\sum_{u\in V(T_\lam)}\prod_{i=1}^{\ell(\lam)}\big((\tb{1}^Tx_0+\tb{1}^Tx_1)x_u+(-1)^{\lam_i}(\tb{1}^Tx_0-\tb{1}^Tx_1)y_u+o(1)\big).
        \end{align*}
        Now since $y_u=\begin{cases}
            x_u & \tn{ when }u\in V_0\tn,\\
            -x_u & \tn{ when }u\in V_1,\\
        \end{cases}$ splitting the sum over $u$ into the two cases gives
        \begin{align*}
        & \rho^{n-1}\sum_{u\in V_0}\prod_{i=1}^{\ell(\lam)}x_u\big((1+(-1)^{\lam_i})\tb{1}^Tx_0+(1-(-1)^{\lam_i})\tb{1}^Tx_1\big)\\
        &\hspace{15pt}+\rho^{n-1}\sum_{u\in V_1}\prod_{i=1}^{\ell(\lam)}x_u\big((1-(-1)^{\lam_i})\tb{1}^Tx_0+(1+(-1)^{\lam_i})\tb{1}^Tx_1\big)+o(\rho^{n-1})\\
        &=\rho^{n-1}\sum_{u\in V_0}x_u^{\ell(\lam)}\prod_{i=1}^{\ell(\lam)}\big(2\mathbbm1\{\lam_i\tn{ even}\}\tb{1}^Tx_0+2\mathbbm1\{\lam_i\tn{ odd}\}\tb{1}^Tx_1\big)\\
        &\hspace{15pt}+\rho^{n-1}\sum_{u\in V_1}x_u^{\ell(\lam)}\prod_{i=1}^{\ell(\lam)}\big(2\mathbbm1\{\lam_i\tn{ odd}\}\tb{1}^Tx_0+2\mathbbm1\{\lam_i\tn{ even}\}\tb{1}^Tx_1\big)+o(\rho^{n-1})\\
        &=2^{\ell(\lam)}\rho^{n-1}\bigg(\sum_{u\in V_0}x_u^{\ell(\lam)}(\mathbf{1}^Tx_0)^{e(\lam)}(\mathbf{1}^Tx_1)^{o(\lam)}+\sum_{u\in V_1}x_u^{\ell(\lam)}(\mathbf{1}^Tx_0)^{o(\lam)}(\mathbf{1}^Tx_1)^{e(\lam)}\hspace{-2pt}\bigg)+o(\rho^{n-1}).
    \end{align*}
    (We are writing $\mathbbm{1}\{\lam_i\tn{ odd}\}$ here to denote the indicator variable for the event that $\lam_i$ is odd, and similarly for $\mathbbm{1}\{\lam_i\tn{ even}\}$.) Lastly, since $x_0$ and $x_1$ have strictly positive entries, we have $\tb{1}^Tx_i=\|x_i\|_1$ and $\sum_{u\in V_i}x_u^{\ell(\lam)}=\|x_i\|_{\ell(\lam)}^{\ell(\lam)}$ for $i=0,1$, from which the result follows. 
\end{proof}

This leads us to suggest the following weaker version of the conjecture from \cite{EFHKY22} that the self-CSF distinguishes all trees, which we think may be easier to prove:

\begin{conjecture}
    The self-CSF distinguishes all spiders of sufficiently large size.
\end{conjecture}

\begin{proof}[Supporting evidence]
    In Theorem 4 of \cite{Oliveira2018}, the authors showed that spiders with the same number of vertices are distinguished by their spectral radii, and in fact, an even stronger result holds: if $\rho_\lam$ is the spectral radius of $T_\lam$, then $\rho_\lam>\rho_\mu$ whenever $\tn{rev}(\lam)>\tn{rev}(\mu)$ in lexicographic order, where $\tn{rev}(\lam)$ denotes the sequence obtained by writing the parts of the partition $\lam$ in reverse (of the usual) order, i.e. in increasing rather than decreasing order. Since $\rho_\lam^{n-1}$ is, asymptotically speaking, the main factor determining the size of the right hand side of the expression in Proposition \ref{prop:spider_endos_eigenexpansion}, it may seem likely that the lexicographic ordering of the spectral radii extends to that of the size of the endomorphism monoid. Since $|\tn{End}(T_\lam)|$ can be obtained from the self-CSF through summing all of its coefficients in the $m^n$-basis, this would imply the self-distinguishability of spiders. 
    
    Unfortunately, this is not true in general, one reason being the fact that changing the bipartition sizes heavily impact the values of $\|x_0\|_{\ell(\lam)}$, $\|x_1\|_{\ell(\lam)}$, $\|x_0\|_1$, $\|x_1\|_1$, $e(\lam)$, and $o(\lam)$, with the combined effect of these changes enough for the terms in parentheses in Proposition \ref{prop:spider_endos_eigenexpansion} to hinder the desired monotonically increasing property of $|\tn{End}(T_\lam)|$ from holding in general. For spiders with the same value of the ordered pair $(|V_0|,|V_1|)$, our numerical tests show that the terms in the parentheses generally do not vary as much as between spiders with different bipartition sizes, but even then, it still is not generally true in such cases that $|\tn{End}(T_\lam)|>|\tn{End}(T_\mu)|$ whenever $\tn{rev}(\lam)>\tn{rev}(\mu)$. A counterexample is the pair of 16-vertex spiders $T_{6,4,4,1}$ and $T_{9,2,2,2}$, both of which have bipartition sizes 8 and 8: we have $|\tn{End}(T_{6,4,4,1})|=2633524>2628414=|\tn{End}(T_{9,2,2,2})|$ even though $\tn{rev}(6,4,4,1)<\tn{rev}(9,2,2,2)$, with the direction of the inequality remaining the same if we use the approximation given by Proposition \ref{prop:spider_endos_eigenexpansion} (the values for $T_{6,4,4,1}$ and $T_{9,2,2,2}$ are $3.38\times10^6$ and $3.13\times10^6$ respectively) rather than the exact value of $|\tn{End}(T_\lam)|$. However, these exceptions seem to be rare -- with the help of a few other pieces of information, such as the ones we can deduce using Propositions \ref{prop:spider_legs} through \ref{prop:spider_m21n-2}, we believe that  self-distinguishability of spiders remains to be a likely conclusion. \phantom\qedhere
\end{proof}

\subsection{Self-distinguishability of caterpillars}\label{subsec:caterpillars}

A \emph{\tb{\tcb{caterpillar}}} is a tree whose induced subtree obtained by deleting all the leaves is a non-trivial path, known as the \emph{\tb{\tcb{spine}}}. Once again, we follow the notation of \cite{MMW08}: let $v_0,\dots,v_s$ denote the vertices of the spine in this exact order (up to reversing the sequence), with $s\geq1$ being the length of the spine, and let $f_i$ be the number of leaves adjacent to $v_i$ (so $f_i=\deg(v_i)-1$ for $i=0$ and $i=s$, and $f_i=\deg(v_i)-2$ for $0<i<s$). In particular, we must have $f_0\geq1$ and $f_s\geq1$, since if otherwise, $v_0$ or $v_s$ would be a leaf, contradicting the assumption that they are part of the spine. Denote by $f$ the sequence $(f_0,\dots,f_s)$, and let $T_f$ denote the caterpillar with these $f_i$'s.

Clearly, for any $f$ we have $T_f\cong T_{\tn{rev}(f)}$, where $\tn{rev}(f)$ denotes the sequence $f$ written in reverse order, where we write $\cong$ to mean that the two graphs are isomorphic. As we shall show below, $\tn{rev}(f)$ plays a role in determining the order of $\tn{Aut}(T_f)$.

\begin{prop}\label{prop:caterpillar_auts}
    For a caterpillar with spine length $s$, $[m^n_{1^n}]X_{T_f}^{T_{f}}=(1+\mathbbm1\{\tn{rev}(f)=f\})\prod_{i=0}^{s}f_i!$, where $$\mathbbm1\{\tn{rev}(f)=f\}=
    \begin{cases}
        1&\tn{if }\tn{rev}(f)=f\\
        0&\tn{if }\tn{rev}(f)\neq f
    \end{cases}$$
    is the indicator function of the event $\tn{rev}(f)=f$.
\end{prop}
\begin{proof}
    Since automorphisms preserve degrees and map paths to paths, the spine, being the unique path of maximal length consisting of vertices with degree greater than 1, must map to itself under an automorphism $\sigma$. There are two ways to make this happen: the first being $\sigma(v_i)=v_i$, which is valid for any caterpillar, and the second being $\sigma(v_i)=v_{s-i}$, which is possible if and only if $v_{s-i}$ has the same number of leaves attached to it as $v_i$ for all $i$ (since automorphisms preserve degrees), which is the case when and only when $\tn{rev}(f)=f$. This explains the $1+\mathbbm1\{\tn{rev}(f)=f\}$ factor. Once the image of the spine is fixed, it remains to specify where the leaves map to. Since automorphisms preserve distances, each leaf can only map to other leaves attached to the same spine vertex. Since all $f_i!$ permutations of the leaves attached to vertex $v_i$ are possible, and permutations of different spine vertices are independent of each other, the total number of automorphisms is the product of all the aforementioned factors.
\end{proof}

Our next result shows that we can almost determine the multiset of the $f_i$'s through considering the degrees $d_i=\deg(v_i)$ of the spine vertices. The $d_i$'s and $f_i$'s are related via $d_i=f_i+2$ if $i\notin\{0,s\}$, while $d_0=f_0+1$ and $d_s=f_s+1$.

\begin{prop}\label{prop:caterpillar_degrees}
    $X_{T_f}^{T_f}$ uniquely determines the multiset $\{d_i\}_{0\leq i\leq s}$ for a caterpillar $T_f$ of spine length $s$, provided that $s$ is known beforehand.
\end{prop}
\begin{proof}
    If $T_f$ has $n$ total vertices and spine length $s$, then it has $s+1$ spine vertices and therefore $n-s-1$ leaf vertices. As such, the degree sequence of $T_f$ is $(d_0,d_1,\dots,d_s,1^{n-s-1})$, and in particular, this is fully determined by the $s+1$ parameters $d_0,\dots,d_s$. We know from Proposition \ref{prop:bipartite_degrees} that the self-CSF determines $\sum_{v\in V(T_\lam)}\deg(v)^j=n-s-1+\sum_{i=0}^sd_i^j$ for every integer $1\leq j\leq k$ where $k$ is the size of the smaller part in the bipartition of $T_f$. Since $n$ and $s$ are already known (recall that $\deg\big(X_{T_f}^{T_f}\big)=n$), this is equivalent to knowing $\sum_{i=0}^sd_i^j$ for these values of $j$. By Newton's identities, knowing the first $k$ power-sums of an $(s+1)$-element multiset $\{d_0,\dots,d_s\}$ is equivalent to knowing the first $k$ elementary symmetric polynomials $e_1(d_0,\dots,d_s),\dots,e_k(d_0,\dots,d_s)$ thereof, which by Vieta's formulae is in turn equivalent to knowing the first $k$ non-leading coefficients of the unique monic polynomial $$\prod_{i=0}^s(x-d_i)=x^{s+1}+\sum_{i=1}^{s+1}(-1)^ie_i(d_0,\dots,d_s)x^{s+1-i}$$ whose roots form the required multiset. As such, if $k\geq s+1$, then we are done, since we can deduce the polynomial and therefore its multiset of roots. It thus remains to prove the cases where $k<s+1$.

    For each spine vertex $v_i$ $(0\leq i\leq s)$, let $L_i$ denote the set of leaves attached to $v_i$. Then the vertex set of $T_f$ bipartitions into
    $$B_0=\bigg(\bigsqcup\limits_{i\tn{ even}}L_i\bigg)\sqcup\{v_i\mid i\tn{ odd}\}\hspace{30pt}\tn{ and }\hspace{30pt}B_1=\bigg(\bigsqcup\limits_{i\tn{ odd}}L_i\bigg)\sqcup\{v_i\mid i\tn{ even}\}.$$
    Letting $\ell_i=|L_i|-1=f_i-1$, we see that
    $$|B_0|=\sum_{i\tn{ odd}}|\{v_i\}|+\sum_{i\tn{ even}}|L_i|=\sum_{i\tn{ odd}}1+\sum_{i\tn{ even}}(1+\ell_i)=\sum_{0\leq i\leq s}1+\sum_{i\tn{ even}}\ell_i=s+1+\sum_{i\tn{ even}}\ell_i,$$
    and similarly,
    $$|B_1|=s+1+\sum_{i\tn{ odd}}\ell_i.$$
    By definition, $k=\min(|B_0|,|B_1|)$, so if $k=s+1-c$ with $c>0$, then either $\sum\limits_{i\tn{ odd}}\ell_i=-c$ or $\sum\limits_{i\tn{ even}}\ell_i=-c$. But the only way $\ell_i$ could be negative is if $f_i=0$, in which case $\ell_i=-1$. In either case, by the pigeonhole principle, this requires $f_i=0$ for at least $c$ values of $i$. Since $f_0\neq0\neq f_s$, these $c$ values of $i$ must be between 1 and $s-1$, for which $d_i=f_i+2=2$. Consequently, when $k=s+1-c$, we can already deduce the values of $c$ of the $s+1$ parameters $d_0,\dots,d_s$, and thus knowing the first $k$ power-sums of all the $d_i$'s in this case is equivalent to knowing the first $k$ power sums of the remaining $s+1-c$ unknowns. But since $s+1-c=k$, we know the values of as many power-sums as the size of the unknown multiset, so by the polynomial-root argument mentioned above, we are done.
\end{proof}

Combining both results above, we obtain the following corollary.

\begin{cor}\label{cor:caterpillar_f0fs}
    $X_{T_f}^{T_f}$ tells us the value of $(1+\mathbbm1\{\tn{rev}(f)=f\})f_0f_s$.
\end{cor}
\begin{proof}
    Observe that
    $$\prod_{i=0}^sf_i!=(d_0-1)!(d_s-1)!\prod_{i=1}^{s-1}(d_i-2)!=(d_0-1)(d_s-1)\prod_{i=0}^s(d_i-2)!=f_0f_s\prod_{i=0}^s(d_i-2)!,$$
    and note that by Proposition \ref{prop:caterpillar_auts} the left hand side is $(1+\mathbbm1\{\tn{rev}(f)=f\})^{-1}[m^n_{1^n}]X_{T_f}^{T_f}$, while the product on the right hand side can be deduced from knowledge of the multiset $\{d_0,\dots,d_s\}$, which we know from Proposition \ref{prop:caterpillar_degrees} to be obtainable from $X_{T_f}^{T_f}$.
\end{proof}

\subsection{Self-CSFs for forests}\label{subsec:self-CSF_forests}

For forests, we show that the self-CSF can usually detect the number of connected components, again with the $X_{P_3}^{P_3}=X_{K_1\sqcup K_2}^{K_1\sqcup K_2}$ example as a special exception:

\begin{prop}\label{prop:forest_components}
    If $F_1$ and $F_2$ are forests with $X_{F_1}^{F_1}=X_{F_2}^{F_2}$, then $\kappa(F_1)=\kappa(F_2)$ unless $F_1=K_1\sqcup K_2$ and $F_2=P_3$, where $\kappa(G)$ is the number of connected components of $G$.
\end{prop}

\begin{proof}
    Consider the length 2 monomial $m_{n-k,k}^n$ showing up in $X_{F_1}^{F_1}=X_{F_2}^{F_2}$ such that $n-k$ is maximal. By the argument in the proof of Proposition \ref{prop:detect_trees}, $[m_{n-k,k}^n]X_{F_1}^{F_1} = 2^{e_1+1}\cdot |E(F_1)|$, where $e_1$ is the number of components in $F_1$ with equal part sizes, and $[m_{n-k,k}^n]X_{F_2}^{F_2} = 2^{e_2+1}\cdot |E(F_2)|$, with $e_2$ defined similarly. By Proposition \ref{prop:diff_sequences}, the number of components with unequal parts must be the same in $F_1$ as in $F_2$, so we can multiply both coefficients by the same power of 2 to get $$2^{\kappa(F_1)}\cdot |E(F_1)|=2^{\kappa(F_2)}\cdot |E(F_2)|.$$ Since $F_1$ and $F_2$ are forests, we have $|E(F_1)|=n-\kappa(F_1)$ and $|E(F_2)|=n-\kappa(F_2),$ so $$(n-\kappa(F_1))\cdot 2^{\kappa(F_1)} = (n-\kappa(F_2))\cdot 2^{\kappa(F_2)}.$$ Assuming without loss of generality that $\kappa(F_1)>\kappa(F_2),$ we can solve for $n$ to get $$n = \frac{\kappa(F_1)\cdot 2^{\kappa(F_1)}-\kappa(F_2)\cdot 2^{\kappa(F_2)}}{2^{\kappa(F_1)}-2^{\kappa(F_2)}}=\kappa(F_1)+\frac{\kappa(F_1)-\kappa(F_2)}{2^{\kappa(F_1)-\kappa(F_2)}-1}.$$ For $n$ to be an integer, we thus need $2^{\kappa(F_1)-\kappa(F_2)}-1 \le \kappa(F_1)-\kappa(F_2),$ which is only possible if $\kappa(F_1)-\kappa(F_2)=1.$ In that case, we get $$n = \kappa(F_1)+\frac{1}{2^1-1}=\kappa(F_1)+1,$$ so $\kappa(F_1)=n-1$ and $\kappa(F_2)=n-2.$ This implies that $F_1$ has a single edge and $n-2$ lone vertices, while $F_2$ has 2 edges, so it is either a length 3 path plus $n-3$ lone vertices, or 2 disconnected edges plus $n-4$ lone vertices.

    However, we know that $$|\tn{Aut}(F_1)|=[m_{1^n}^n]X_{F_1}^{F_1}=[m_{1^n}^n]X_{F_2}^{F_2}=|\tn{Aut}(F_2)|,$$ where $|\Aut(G)|$ is the number of automorphisms of a graph $G$. Thus, $F_1$ and $F_2$ must have the same number of automorphisms. We have $$|\Aut(F_1)|=2\cdot (n-2)!,$$ since the endpoints of the edge can map to themselves or each other, while the $n-2$ lone vertices can be permuted in any way. 
    
    In the case where $F_2$ has a length 3 path, we get $$|\Aut(F_2)|=2\cdot (n-3)!,$$ since the path can be flipped or not, and the lone vertices can be permuted in any order. Setting these equal gives $(n-2)!=(n-3)!$, which can only happen if $n=3$, giving $F_1=K_1\sqcup K_2$ and $F_2=P_3.$ 
    
    The other case is that $F_2$ has the 2 disconnected edges, giving $$|\Aut(F_2)|=8\cdot(n-4)!,$$ since we can choose whether the edges map to themselves or each other and whether or not to flip each edge, and then there are $(n-4)!$ permutations of the lone vertices. Setting $2\cdot(n-2)!=8\cdot(n-4)!$ gives $(n-2)(n-3)=4,$ which does not have integer solutions. Thus, this case is impossible, so we must have either $\kappa(F_1)=\kappa(F_2)$, or $F_1=K_1\sqcup K_2$ and $F_2=P_3.$
\end{proof}

\section{Power sum expansions}\label{sec:p-expansions}

In this section, we focus on power sum expansions for $H$-CSFs in certain cases. For a partition $\lam=\lam_1\dots\lam_{\ell(\lam)},$ the \emph{\tb{\tcb{power sum symmetric function}}} $p_\lam$ is defined by $$p_\lam := \prod_{i=1}^{\ell(\lam)}(x_1^{\lam_i}+x_2^{\lam_i}+\dots).$$ The involution $\omega$ is the linear map on symmetric functions sending $p_\lam$ to $$\omega(p_\lam)=(-1)^{|\lam|-\ell(\lam)}p_\lam.$$ For any graph $G$, $\omega(X_G)$ is always \emph{\tb{\tcb{$p$-positive}}}, meaning the coefficients when $\omega(X_G)$ is expanded in the $p$-basis are all nonnegative. The authors of \cite{EFHKY22} point out that $\omega(X_G^H)$ is not always $p$-positive, or even \emph{\tb{\tcb{$p$-monotone}}} (meaning the nonzero $p$-coefficients are either all positive or all negative), citing $\omega(X_{P_4}^{C_7})$ as a counterexample. However, they conjecture that the $H$-CSF is always $p$-monotone in the case where $H=S_{n+1}$ is the $(n+1)$-vertex \emph{\tb{\tcb{star graph}}}, consisting of one central vertex and $n$ other vertices connected only to the central vertex. In \S\ref{subsec:p_monotonicity_stars}, we prove their conjecture as long as $n$ is sufficiently large compared to $G$. Then in the rest of \S\ref{sec:p-expansions}, we look at power sum expansions for other cases where $H$ is still a complete bipartite graph.

\subsection{Power sum monotonicity for \texorpdfstring{$H$}{H} a star}\label{subsec:p_monotonicity_stars}

In this section we prove the $p$-monotonicity of $\omega(X_G^H)$ for $H$ a star graph that is sufficiently large compared to $G$. As the authors of \cite{EFHKY22} note, the result is automatic when $G$ is not bipartite, since then there are no homomorphisms from $G$ to a star graph, so we may assume $G$ is bipartite.

\begin{prop}\label{prop:p_monotonicity_stars}
    As long as $n$ is at least as large as the sizes of both parts in any bipartition of $G$, the $p$-coefficients in $\omega(X_G^{S_{n+1}})$ all have the same sign.
\end{prop}

\begin{proof}
    We can write out explicitly what the $p$-expansion is in this case:

    \begin{lemma}\label{lem:star_p_expansion}
        If $G^1,\dots,G^{\ell}$ are the connected components of $G$, $k_1^i\ge k_2^i$ are the sizes of the two parts $G_1^i$ and $G_2^i$ in the bipartition of $G^i$, and $n\ge k_1^1+\dots + k_1^\ell$, then
        \begin{align*}
            X_G^{S_{n+1}} = \sum_{(b_1,\dots,b_\ell)\in\{1,2\}^\ell}n!\sum_{j=k_{b_1}^1+\dots+k_{b_\ell}^\ell}^{|V(G)|}(-1)^{j-(k^1_{b_1}+\dots+k^\ell_{b_\ell})}p_{j,1^{|V(G)|-j}}\sum_{\substack{j_1,\dots,j_\ell:\\ k_{b_i}^i\le j_i\le k^i_1 + k^i_2,\\ j_1+\dots+j_\ell=j}} \prod_{i=1}^\ell \binom{k_{3-b_i}^i}{j_i-k_{b_i}^i}.
        \end{align*}
    \end{lemma}

    \begin{proof}
        Recall that we defined the $H$-CSF by first defining it in terms of the $|V(H)|$ variables $x_1,x_2,\dots,x_{|V(H)|},$ and then using the $m$-expansion to extend it to a countably infinite variable set. Since $|V(H)|=n+1$ for $H=S_{n+1}$, this means we should start with the variable set $x_1,x_2,\dots,x_{n+1}$, and then extend to a countably infinite variable set by keeping the $m$-expansion the same. All our monomials will thus have length at most $n+1$. The span of the $m_\lam$'s with $\ell(\lam)\le n+1$ is the same as the span of the $p_\lam$'s with $\ell(\lam)\le n+1$, since we can write each $m_\lam$ as a linear combination of power sums of the same length or shorter, and similarly we can write each $p_\lam$ as a sum of monomials of the same length or shorter. Thus, there will be a unique way to write our $H$-CSF $X_G^{S_{n+1}}(x_1,x_2,\dots,x_{n+1})$ as a linear combination of $p_\lam$'s with $\ell(\lam)\le n+1$, and the $p$-expansion will then stay the same when we extend to infinitely many variables. Thus, the $p$-expansion for $X_G^{S_{n+1}}(x_1,x_2,\dots)$ will be the same as the unique way to write $X_G^{S_{n+1}}(x_1,x_2,\dots,x_{n+1})$ as a sum of power sum terms that all have length at most $n+1$, so we will consider the $p$-expansion for $X_G^{S_{n+1}}(x_1,x_2,\dots,x_{n+1})$ using monomials of length at most $n+1$. (We need to say this because if we allow power sum terms of length more than $n+1$, then there will actually be multiple possible $p$-expansions when we restrict to $n+1$ variables, but only the one where all $p$-terms have length at most $n+1$ will match the $p$-expansion when we extend to an infinite variable set.)

        Note that in any homomorphism from $G$ to $S_{n+1}$, for every $i$, one of $G_1^i$ or $G_2^i$ must map entirely onto the central vertex $v$ of the star, while the other part of $G^i$ must map onto some of the remaining vertices $S_{n+1}-v$. The tuple $(b_1,\dots,b_\ell)\in\{1,2\}^n$ corresponds to the case where $G_{b_i}^i$ maps to $v$ and the other part $G_{3-b_i}^i$ maps to $S_{n+1}-v$, so we will sum over those cases. (We can write $G_{3-b_i}^i$ to denote the other part of the bipartition of $G^i$ that is not $G_{b_i}^i$, since $3-1=2$ and $3-2=1$.)
        
        We then use inclusion-exclusion on the events that some of the vertices from the wrong part $G_{3-b_i}^i$ of one or more of the $G^i$'s also get mapped onto $v$. We start by adding the term $p_{j,1^{|V(G)|-j}}$ with $j=k_{b_1}^1+\dots+k_{b_\ell}^\ell$, because this corresponds to mapping all of the desired $j$ vertices onto $v$ (since $p_j$ represents mapping $j$ vertices of $G$ to the same label, or equivalently to the same vertex of $S_{n+1}$), and then assigning the remaining vertices of $G$ arbitrarily to any vertices of $S_{n+1}$ (since $p_{1^{|V(G)|-j}}$ represents independently assigning a label to each of the $|V(G)|-j$ vertices of $G$).  Note that these $p$-terms so far all satisfy $\ell(j,1^{|V(G)|-j})\le n+1$ by our assumption that $n\ge k_1^1+\dots+k_1^\ell$, since the  length is $$\ell(j,1^{|V(G)|-j})=1+|V(G)|-j=1+|V(G)|-(k_{b_1}^1+\dots+k_{b_\ell}^\ell)=1+k_{3-b_1}^1+\dots+k_{3-b_\ell}^\ell\le 1 + n,$$ using the fact that $$k_{b_1}^1+\dots+k_{b_\ell}^\ell +k_{3-b_1}^1 + \dots + k_{3-b_\ell}^\ell = |V(G)|,$$ since both sides count each vertex of $G$ exactly once.
        
        This enforces the condition that the desired vertices map to $v$, but it includes terms where additional vertices also map to $v$, so we need to alternately add or subtract the terms where a particular set of extra vertices also map to $v$. The terms we add or subtract are all of the form $p_{j,1^{|V(G)|-j}}$, since that is the generating series for maps sending at least $j$ total vertices to $v$. In that case, there are $j-(k_{b_1}^1+\dots+k_{b_\ell}^\ell)$ extra vertices getting sent to $v$, so this case falls into $j-(k_{b_1}^1+\dots+k_{b_\ell}^\ell)$ of the sets that each correspond to the maps sending one particular extra vertex to $v$. Thus, the sign of this term when we use inclusion-exclusion is $(-1)^{j-(k_{b_1}^1+\dots+k_{b_\ell}^\ell)}$. These new $p$-terms also all have length at most $n+1$ because of the assumption $n\ge k_1^1+\dots+k_1^\ell$.
        
        We also need to distribute the $j$ vertices mapping to $v$ among the $G^i$'s, so if we want $j_i$ of the vertices to come from $G^i$ for each $i$, we need $j_1+\dots+j_\ell=j.$ We also need $k_{b_i}^i\le j_i\le k_1^i+k_2^i$, because we require the $k_{b_i}$ vertices in $G_{b_i}^i$ to all map to $v$, and then the remaining $j_i-k_{b_i}^i$ vertices must be chosen from among the $k_{3-b_i}^i$ vertices in $G_{3-b_i}^i$. Thus, there are $\binom{k_{3-b_i}^i}{j_i-k_{b_i}^i}$ ways to choose which extra vertices in $G^i$ map to $v$ for a particular $j_i$, and we then need to multiply over all $i$ and then sum over all possible tuples of $j_i$'s.
        
        Finally, to explain the $n!$, recall that we are currently using the variables $x_1,x_2,\dots,x_{n+1}$. A given $p_{j,1^{|V(G)|-j}}$ term corresponds to assigning to a particular $j$ vertices in $G$ the same \emph{label} between 1 and $n+1$, and then assigning arbitrary labels to the remaining $|V(G)|-j$ vertices of $G$. However, we also need to choose which labels correspond to which vertices of $S_{n+1}$. For each monomial $x_{i_1}^jx_{i_2}x_{i_3}\dots x_{i_{|V(G)|-j+1}}$ that we get from expanding $p_{j,1^{|V(G)|-j}}$, we need $x_{i_1}$ to correspond to the central vertex $v$ of $S_{n+1}$, so we need to assign label $i_1$ to vertex $v$, but the remaining $n$ labels can be assigned arbitrarily to the other $n$ vertices of $S_{n+1}$. Thus, each of our $p_{j,1^{|V(G)|-j}}$ terms actually corresponds to $n!$ different labelings of $S_{n+1}$, so we need to multiply the whole expression by $n!$.
    \end{proof}

    Now to prove Proposition \ref{prop:p_monotonicity_stars}, we need to apply $\omega$ to the terms from Lemma \ref{lem:star_p_expansion}. For $\lam = j,1^{|V(G)|-j}$, we have $|\lam|=|V(G)|$ and $\ell(\lam)=|V(G)|-j+1$, so $|\lam|-\ell(\lam)=j-1$. Thus, $\omega(p_{j,1^{|V(G)|-j}}) = (-1)^{j-1}p_{j,1^{|V(G)|-j}}.$ Now if we apply $\omega,$ factor out the $n!$, and then group the terms in the $p$-expansion from Lemma \ref{lem:star_p_expansion} first according to $j$ and then according to the tuple $(j_1,\dots,j_\ell)$, we get $$[p_{j,1^{|V(G)|-j}}]\omega(X_G^{S_{n+1}})=n!\sum_{\substack{j_1,\dots,j_\ell:\\ k_{2}^i\le j_i\le k^i_1 + k^i_2,\\ j_1+\dots+j_\ell=j}}\sum_{(b_1,\dots,b_\ell)\in\{1,2\}^\ell}(-1)^{k^1_{b_1}+\dots+k^\ell_{b_\ell}-1}\prod_{i=1}^\ell \binom{k_{3-b_i}^i}{j_i-k_{b_i}^i}.$$ We can replace the $k_{b_i}^i\le j_i$ condition with $k_2^i\le j_i$, since we assumed $k_2^i\le k_1^i$ and hence $k_2^i\le k_{b_i}^i$ for either choice of $b_i\in\{1,2\}.$ This does not introduce any extra terms, because if $j_i<k_i^1$ and $b_i=1$, then $j_i-k_{b_i}^i=j_i-k_1^i$ will be negative, so the binomial coefficient $\binom{k_{3-b_i}^i}{j_i-k_{b_i}^i}$ simply becomes 0 and the extra term we might have introduced thus goes away.

    Now for a fixed choice of $j_1,\dots,j_\ell$, each $b_i$ can be chosen independently from all the other $b_i$'s, so we can actually factor the resulting sum of terms over $i=1,2,\dots,\ell$, where the two terms in the $i^{\text{th}}$ factor correspond to choosing either $b_i=1$ or $b_i=2$. This gives $$[p_{j,1^{|V(G)|-j}}]\omega(X_G^{S_{n+1}})=-n!\sum_{\substack{j_1,\dots,j_\ell:\\ k_{2}^i\le j_i\le k^i_1 + k^i_2,\\ j_1+\dots+j_\ell=j}}\prod_{i=1}^\ell\left((-1)^{k_1^i}\binom{k_2^i}{j_i-k_1^i}+(-1)^{k_2^i}\binom{k_1^i}{j_i-k_2^i}\right).$$ Rewriting the combinations using the identity $\binom ab = \binom a{a-b}$ gives $$[p_{j,1^{|V(G)|-j}}]\omega(X_G^{S_{n+1}})=-n!\sum_{\substack{j_1,\dots,j_\ell:\\ k_{2}^i\le j_i\le k^i_1 + k^i_2,\\ j_1+\dots+j_\ell=j}}\prod_{i=1}^\ell\left((-1)^{k_1^i}\binom{k_2^i}{k_1^i+k_2^i-j}+(-1)^{k_2^i}\binom{k_1^i}{k_1^i+k_2^i-j}\right).$$ If $k_i^1$ and $k_2^i$ have the same parity, then both terms have the same sign. If they have opposite parities, then the term with sign $(-1)^{k_2^i}$ dominates, since $k_1^i\ge k_2^i$ and $\binom ac\ge \binom bc$ whenever $a\ge b$. Thus, the overall sign is $-(-1)^{k_2^1}\dots(-1)^{k_2^\ell}=(-1)^{k_2^1+\dots+k_2^\ell-1}.$ This sign is independent of $j$, so all $p$-terms of $\omega(X_G^{S_{n+1}})$ have the same sign, and thus $\omega(X_G^{S_{n+1}})$ is $p$-monotone, as claimed.
\end{proof}

The difference in cases where $n<k_1^1+\dots+k_1^\ell$ is that we can still write $X_G^{S_{n+1}}(x_1,x_2,\dots,x_{n+1})$ as a sum of the same $p_\lam$'s as above, but the resulting partitions $\lam$ will not all have length $n+1$ or less, so that $p$-expansion will not be the same as the $p$-expansion for the $X_G^{S_{n+1}}(x_1,x_2,\dots)$ when we extend to infinitely many variables. 

As an example to illustrate this issue, note that if we restrict to just the single variable $x_1$, then $p_{1^2}(x_1)=p_2(x_1)=x_1^2$ does not have a unique $p$-expansion, but it does have a unique $m$-expansion $x_1^2=m_2(x_1)$. The corresponding symmetric function in infinitely many variables would then be $m_2(x_1,x_2,\dots)=x_1^2+x_2^2+\cdots.$ The correct $p$-expansion is then $p_2$, not $p_{1^2}$, since $p_2(x_1,x_2,\dots)=x_1^2+x_2^2+\cdots$ has exactly the terms we want, while $p_{1^2}(x_1,x_2,\dots)=(x_1+x_2+\dots)^2=x_1^2+x_2^2+\dots+2(x_1x_2+\cdots)$ has extra terms that we do not want. The correct $p$-expansion to use is thus $p_2$ rather than $p_{1^2}$, which corresponds to 2 being a partition of length 1 (which equals our starting number of variables), while $1^2$ is a partition of length 2 and so should not be allowed (since its length exceeds our starting number of variables).

If $n<k_1^1+\dots+k_1^\ell$, one way we could instead find the $p$-expansion of $X_G^{S_{n+1}}$ would be to start with the $p$-expansion from Lemma \ref{lem:star_p_expansion}, and then convert each $p$-term into a sum of $p$-terms of length at most $n+1$, although the resulting expression will be more complicated.

\subsection{Power sum expansion for \texorpdfstring{$X_G^H$}{XGH} with \texorpdfstring{$H=K_{2,n}$}{H=K(2,n)}}\label{subsec:p_K_2,n}

We can use similar ideas to the star case to write down a $p$-expansion for $X_G^{K_{2,n}}$, for $n$ sufficiently large. Again, we have $X_G^{K_{2,n}}=0$ unless $G$ is bipartite, so we may assume $G$ is bipartite.

\begin{prop}\label{prop:K_2,n_p-expansion}
    With the notation from Lemma \ref{lem:m21n-2_XTT} and again assuming $n\ge k_1^1+\dots+k_1^\ell$, for each pair $j\ge j'\ge 2$ with $j+j'\ge k_2^1+\dots+k_2^\ell,$ we have $$[p_{j,j',1^{|V(G)|-(j+j')}}]X_G^{K_{2,n}} = 2\cdot n!\binom{j+j'}{j}\sum_{\substack{(b_1,\dots,b_\ell)\in\{1,2\}^\ell,\\
    j_1+\dots+j_\ell=j+j',\\k_{b_i}^i\le j_i\le k_1^i+k_2^i}}(-1)^{j+j'-(k_{b_1}^1+\dots+k_{b_\ell}^\ell)}\prod_{i=1}^\ell \binom{k_{3-b_i}^i}{j_i-k_{b_i}^i},$$ while for each $j\ge k_2^1+\dots+k_2^\ell,$ we have
    \begin{align*}
        [p_{j,1^{|V(G)|-j}}]X_G^{K_{2,n}} = &\ 2\cdot (n+1)!\sum_{\substack{(b_1,\dots,b_\ell)\in\{1,2\}^\ell,\\ j_1+\dots+j_\ell=j,\\ k_{b_i}^i\le j_i\le k_1^i+k_2^i}}(-1)^{j-(k_{b_1}^1+\dots+k_{b_\ell}^\ell)}\prod_{i=1}^\ell \binom{k_{3-b_i}^i}{j_i-k_{b_i}^i} 
        \\
        &+ 2\cdot n! \cdot (j+1)\sum_{\substack{(b_1,\dots,b_\ell)\in\{1,2\}^\ell, \\ j_1+\dots+j_\ell=j+1,\\ k_{b_i}^i\le j_i\le k_1^i+k_2^i}}(-1)^{j+1-(k_{b_1}^1+\dots+k_{b_\ell}^\ell)}\prod_{i=1}^\ell\binom{k_{3-b_i}^i}{j_i-k_{b_i}^i}\\ 
        &-\sum_{j'=2}^{j-2}[p_{j',j-j',1^{|V(G)|-(j+j')}}]X_G^{K_{2,n}},
    \end{align*} 
    and all other $p$-coefficients are 0.
\end{prop}

\begin{proof}
    As noted above, the $p$-expansion for $X_G^{K_{2,n}}$ is the same as the $p$-expansion if we restrict to $|V(H)|=n+2$ variables and only allow $p$-terms of length at most $n+2$, and all the $p$-terms above do have the required length, because for all our choices of $j$ and $j'$ we have $$n\ge k_1^1+\dots+k_1^\ell=|V(G)|-(k_2^1+\dots+k_2^\ell)\ge |V(G)|-(j+j'),$$ which implies that our partitions $j,j',1^{|V(G)|-(j+j')}$ and $j,1^{|V(G)|-j}$ all have length at most $n+2$, since $$\ell(j,j',1^{|V(G)|-(j+j')}) = |V(G)|-(j+j')+2\le n+2.$$ Then, like for $H=S_{n+1}$, we will sum over tuples  $(b_1,\dots,b_\ell)\in\{1,2\}^\ell$, corresponding to the case where $G_{b_i}^i$ maps to the size 2 part $\{v,w\}$ of $H$, while $G_{3-b_i}^i$ maps onto the other part of $H$. In our initial $p$-terms, we will enforce the condition that all vertices of $G_{b_i}^i$ map onto $\{v,w\}$ while allowing the vertices of $G_{3-b_i}^i$ to be mapped arbitrarily, and then we will use inclusion-exclusion to eliminate the unwanted terms where some vertices of $G_{3-b_i}^i$ also map to $v$ or $w$.

    Our inclusion-exclusion is again based on the overlapping events where each particular vertex from one of the $G_{3-b_i}^i$'s is required to map to $v$ or $w$. The $p_{j,j',1^{|V(G)|-(j+j)'}}$ terms will correspond to cases where $j$ total vertices map to $v$ and $j'$ vertices map to $w$ or vice versa, meaning $j+j'-(k_{b_1}^1+\dots+k_{b_\ell}^\ell)$ extra vertices map to $\{v,w\}$ in addition to the $k_{b_1}^1+\dots+k_{b_\ell}^\ell$ that are supposed to. Thus, each such term will have sign $(-1)^{j+j'-(k_{b_1}^1+\dots+k_{b_\ell}^\ell)}$. We then need to choose how to distribute those $j+j'$ vertices among the $G^i$'s, so we can sum over cases such that $j_i$ of them come from $G^i$. In such a case, we need $j_1+\dots+j_\ell=j+j'$ since there are $j+j'$ total vertices mapping to $\{v,w\}$. 
    
    If we want $j_i$ of the vertices to come from $G^i$, then we must have $j_i\ge k_{b_i}^i$, since we are already mapping the $k_{b_i}^i$ vertices in $G_{b_i}^i$ to $\{v,w\}$, and we must then choose an additional $j_i-k_{b_i}^i$ vertices from the $k_{3-b_i}^i$ vertices in $G_{3-b_i}^i$ to also map to $\{v,w\}$, which can be done in $\binom{k_{3-b_i}^i}{j_i-k_{b_i}^i}$ ways. Then we need to distribute those $j+j'$ vertices such that $j$ of them map to $v$ and $j'$ to $w$ or vice versa, which can be done in $2\cdot\binom{j+j'}j$ ways, where the extra 2 is to choose which vertex set maps to $v$ and which vertex set to $w$. For each such $p_{j,j',1{|V(G)}-(j+j')}$ term, we need to choose which of the remaining $n$ vertices corresponds to which label, which can be done in $n!$ ways. The labels of $v$ and $w$ are already determined for each monomial we get when expanding $p_{j,j',1{|V(G)}-(j+j')}=p_jp_{j'}p_1^{|V(G)|-(j+j')}$, because if we choose the $x_i^j$ term from the $p_j$ factor, then $v$ must have label $i$, and if we choose $x_{i'}^{j'}$ from the $p_{j'}$ factor, then $w$ must have label $i'$. In the case $j=j'$, we can still say that there are initially 2 ways to choose which vertex set maps onto $v$ and which one onto $w$, and then we can require that in the product $p_jp_jp_1^{|V(G)|-2j},$ the $x_{i}^j$ chosen from first $p_j$ factor corresponds to the label $i$ assigned to $v$ while the $x_{i'}^j$ chosen from the second $p_j$ factor corresponds to the label $i'$ assigned to $w$, so the same formula holds when $j=j'$.
    
    However, for each such term, we also want to enforce that the variable $x_i$ chosen from the $p_j$ factor is not the same as the variable $x_{i'}$ chosen from the $j'$ factor, because then we would have $j+j'$ vertices mapping onto one of $v$ or $w$, instead of $j$ mapping onto $v$ and $j'$ onto $w$. Thus, we need to subtract a $p_{j+j',1^{|V(G)|-(j+j')}}$ term for each of the above pairs $j,j'$, which explains the subtracted sum on the final line of our expression for the $p_{j,1^{|V(G)|-j}}$ coefficient.
    
    On the other hand, we do also want to get $p_{j,1^{|V(G)|-j}}$ terms from cases where $j$ vertices map onto one of $v$ or $w$ and none onto the other, or where $j$ vertices map onto one of $v$ or $w$ and 1 vertex maps onto the other. In the latter case, we get the term on the second line of our expression for the $p_{j,1^{|V(G)|-j}}$ coefficient, by the same argument as for the $p_{j,j',1^{|V(G)|-(j+j')}}$ case, but with $j'=1$.

    In the other case, we need to map $j$ of the vertices to just one of $v$ or $w$, and the other vertices can be mapped arbitrarily. The sign is then $(-1)^{j-(k_{b_1}^1+\dots+k_{b_\ell}^\ell)}$, since $j-(k_{b_1}^1+\dots+k_{b_\ell}^\ell)$ extra vertices are mapping onto $v$ or $w$ in addition to the ones that are supposed to. There are 2 ways to choose which of $v$ or $w$ those $j$ vertices map onto, and then $(n+1)!$ ways to assign labels to the remaining $n+1$ variables. We must then choose how to distribute the $j$ vertices among the $G^i$'s, so we can assume $j_i$ of them come from $G^i$ for each $i=1,\dots,\ell$, in which case we need $j_1+\dots+j_\ell=j$. Then we need to choose for each $i$ which $j_i-k_{b_i}^i$ vertices out of the $k_{3-b_i}^i$ vertices in $G_{3-b_i}^i$ map onto $v$ or $w$ in addition to the $k_{b_i}^i$ vertices from $G_{b_i}^i$, which can be done in $\binom{k_{3-b_i}^i}{j_i-k_{b_i}^i}$ ways for each $i$. Putting this together gives the first line in our expression for the coefficient $[p_{j,1^{|V(G)|-j}}]X_G^{K_{2,n}}$. No other terms arise from our inclusion-exclusion, so these are the only terms we need.
\end{proof}

While it is not so easy to tell the signs of the $p_{j,1^{|V(G)|-j}}$ terms from this formula since they are a combination of several different sums, we can deduce that the $p_{j,j',1^{|V(G)|-(j+j')}}$ terms all have the same sign in $\omega(X_G^{K_{2,n}})$ by the same argument as in the $H=S_{n+1}$ case:

\begin{cor}\label{cor:p_monotonicity_K(2,n)}
    If $n\ge k_1^1+\dots+k_1^\ell$, then all power sum terms in $X_G^{K_{2,n}}$ have at most two parts not equal to 1, and all the ones with two parts not equal to 1 have the same sign in $\omega(X_G^{K_{2,n}})$.
\end{cor}

\begin{proof}
    For each $j,j'\ge 2,$ since the partition $j,j',1^{|V(G)|-(j+j')}$ has size $|V(G)|$ and length $|V(G)|-(j+j')+2,$  applying $\omega$ multiplies $p_{j,j',1^{|V(G)|-(j+j')}}$ by $$(-1)^{|j,j',1^{|V(G)|-(j+j')}|-\ell(j,j',1^{|V(G)|-(j+j')})}=(-1)^{|V(G)|-(|V(G)|-(j+j')+2)}=(-1)^{j+j'-2}=(-1)^{j+j'}.$$ Thus, applying $\omega$ just removes the $(-1)^{j+j'}$ from the $p_{j,j',1^{|V(G)|-(j+j')}}$ coefficient. Then we can factor $[p_{j,j',1^{|V(G)|}}]\omega(X_G^{K_{2,n}})$ just like in the $H=S_{n+1}$ case as $$[p_{j,j',1^{|V(G)|}}]\omega(X_G^{K_{2,n}}) = 2\cdot n!\binom{j+j'}{j}\sum_{\substack{j_1+\dots+j_\ell=j+j',\\ k_2^i \le j_i\le k_1^i+k_2^i}} \prod_{i=1}^\ell \left((-1)^{k_1^i}\binom{k_2^i}{j_i-k_1^i}+(-1)^{k_2^i}\binom{k_1^i}{j_i-k_2^i}\right).$$ No matter what $j_1,\dots,j_\ell$ are, the $(-1)^{k_2^i}$ term always dominates, so the overall sign is $(-1)^{k_2^1+\dots+k_2^\ell}$ regardless of what $j$ and $j'$ are.
\end{proof}

\subsection{Some properties of the power sum expansion for \texorpdfstring{$X_G^H$}{XGH} with \texorpdfstring{$H=K_{m,n}$}{H=K(m,n)}}\label{subsec:p_K_m,n}

Now we can extend some of the above ideas to cases where $H=K_{m,n}$ is a general complete bipartite graph. Again, we assume $G$ is bipartite, we let $G=G^1\sqcup \dots \sqcup G^\ell$ be the partition of $G$ into connected components, and we let $k_1^i\ge k_2^i$ be the sizes of the two parts $G_1^i$ and $G_2^i$ in the bipartition of $G^i$.

\begin{prop}\label{prop:K_m,n_p-expansion}
    If $H=K_{m,n}$ with $n\ge k_1^1+\dots+k_\ell^1$ and $n\ge m$, then:
    \begin{enumerate}[label=(\arabic*)]
        \item For each $p_\lam$ term in $X_G^{K_{m,n}}$, at most $m$ parts of $\lam$ are greater than 1.
        \item For each $p_\lam$ term in $X_G^{K_{m,n}}$, the sum of the largest $m$ parts of $\lam$ must be at least $k_2^1+\dots+k_\ell^1$.
        \item All $p_\lam$ terms with exactly $m$ parts greater than 1 have the same sign $(-1)^{m+k_2^1+\dots+k_2^\ell}$ in $\omega(X_G^{K_{m,n}}).$
    \end{enumerate}
\end{prop}

\begin{proof}
    Let $H_m\sqcup H_n$ be the bipartition of $H=K_{m,n}$, where $H_m$ is the size $m$ part and $H_n$ the size $n$ part. The idea is that we can compute the $p$-expansion for $X_G^{K_{m,n}}$ using an inclusion-exclusion argument similar to the ones in Propositions \ref{prop:p_monotonicity_stars} and \ref{prop:K_2,n_p-expansion}. Like before, we can start by computing the $p$-expansion when we restrict the $H$-CSF to $|V(H)|=m+n$ variables, and then as long as all $p$-terms showing up have length at most $m+n$, our resulting $p$-expansion will be the same as the actual $p$-expansion for $X_G^H$ when we extend to an infinite variable set.
    
    To apply inclusion-exclusion, we can start with the maps where for each connected component $G^i$, we force one of $G_1^i$ or $G_2^i$ (which we can call $G_{b_i}^i$) to map entirely onto $H_m$, and then we map the remaining vertices arbitrarily. We can split these colorings into cases based on how many vertices of $G$ map to each of the $m$ vertices of $H_m$. Each such case can be represented by a power sum term $$p_{j_1',\dots,j_m',1^{|V(G)|-(j_1'+\dots+j_m')}},$$ where the first $m$ factors $p_{j_1'},\dots,p_{j_m'}$ each represent mapping $j_i'$ vertices of $G$ onto a specific vertex of $H_m$, and the remaining factors represent mapping the remaining vertices of $G$ arbitrarily. Then since the $j_1'+\dots+j_m'$ vertices mapping to $H_m$ are precisely the vertices in $G_{b_1}^1\sqcup \dots \sqcup G_{b_\ell}^\ell$, we have $$j_1'+\dots+j_m'=k_{b_1}^1+\dots+k_{b_\ell}^\ell=|V(G)|-(k_{3-b_1}^1+\dots+k_{3-b_\ell}^\ell)\ge |V(G)|-n,$$ hence $$\ell(j_1',\dots,j_m',1^{|V(G)|-(j_1'+\dots+j_m')})\le m+n=|V(H)|.$$ Thus, all our power sum terms so far correspond to partitions of length $m+n,$ as needed. We should actually also include terms here where $j_1'+\dots+j_m'=k_{b_1}^1+\dots+k_{b_\ell}^\ell$ but some of the $j_i'$ values are 0, because we do not require all vertices in $H_m$ to actually get used. In that case, we have fewer parts in our partition, so fewer than $m+n$ total parts, and we also certainly have fewer than $m$ parts that are greater than 1. Thus, (1) and (2) hold for all our terms so far.

    The first issue is that these current power sum terms do not enforce that the variables chosen from the first $m$ factors are all different from each other. To fix this, we can use inclusion-exclusion to alternately add or subtract cases where certain subsets of the first $m$ parts of the partition get merged with each other. Each added or subtracted case corresponds to a power sum with some of the first $m$ parts merged, such as $$p_{j_1'+j_2',j_3'+j_4',\dots,j_m',1^{|V(G)|-(j_1'+\dots+j_m')}}$$ if the first two parts get merged and the next two parts also get merged, or eventually $$p_{j_1'+\dots+j_m',1^{|V(G)|-(j_1'+\dots+j_m')}}$$ if all $m$ parts get merged. Any such term has strictly fewer than $m+n$ total parts and strictly fewer than $m$ parts that are greater than 1. If we choose the correct coefficients for these added and subtracted $p$-terms, what remains will be the generating series for all ways to choose a tuple $(b_1,\dots,b_\ell)$ together with a map from $G$ to $H$ and a labeling of $V(H)$ such that all vertices of $G_{b_i}^i$ map onto $H_m$ for every $i$, while the remaining vertices of $G$ may be mapped arbitrarily. Since all we are doing to our partitions here is merging some of the $m$ largest parts together, (1) and (2) will still hold for all $p$-terms introduced so far.

    The second issue is that we are not enforcing that vertices from the $G_{3-b_i}^i$'s all map to $H_n$ instead of to $H_m,$ so we can again use inclusion-exclusion to subtract the unwanted terms. That is, we will alternately add and subtract terms where the sum of the first $m$ parts $j_1'+\dots+j_m'$ is greater than our desired value $k_{b_1}^1+\dots+k_{b_\ell}^\ell$, with the sign being given by the excess $(-1)^{(j_1'+\dots+j_m')-(k_{b_1}^1+\dots+k_{b_\ell}^\ell)},$ since that excess represents the number of extra vertices from the $G_{3-b_i}^i$'s that map onto $H_m$. In these new terms, the sum of the parts greater than 1 is larger than in our initial, so the total number of parts is less than for our initial terms, and hence is still always at most $m+n$. We will also again need to introduce additional terms to ensure that we are not double-counting cases where some of the $m$ parts corresponding to the vertices is $H_m$ get merged. However, again, all the additional terms introduced will just be formed by either increasing the sizes of some of the $m$ largest parts, or merging some of the $m$ largest parts together. Thus, all partitions introduced will be valid since they will still have length less than $m+n$, and also (1) and (2) will hold for all of them.

    It remains to check (3). The key here is that a partition $j_1',\dots,j_m',1^{|V(G)|-(j_1'+\dots+j_m')}$ with exactly $m$ parts greater than 1 can only represent a case where at least $j_1'+\dots+j_m'$ vertices of $G$ get assigned to the $m$ vertices of $H_m$, and $j_i'$ of them get assigned to the $i^{\text{th}}$ vertex of $H_m$ for some ordering on the vertices in $H_m$. Such a term cannot get introduced in any of the additional added or subtracted terms corresponding to merging some of the parts together by actually assigning multiple groups of vertices to the same vertex of $H_m$ instead of to different vertices of $H_m$. Thus, in the $(b_1,\dots,b_\ell)$ case, this term must have sign $$(-1)^{(j_1'+\dots+j_m')-(k_{b_1}^1+\dots+k_{b_\ell}^\ell)},$$ since that sign represents the number of extra vertices that get mapped to $H_m$ in addition to the desired ones. We can also compute the coefficient of such a term explicitly in the same way as in Lemma \ref{lem:star_p_expansion} and Proposition \ref{prop:K_2,n_p-expansion}. We can sum over tuples $(j_1,\dots,j_\ell)$ with $j_1+\dots+j_\ell=j_1'+\dots+j_m'$ such that $j_i$ represents the number of vertices mapped onto $H_m$ that come from $G^i$. For such a case, there are $\binom{k_{3-b_i}^i}{j_i-k_i}$ ways to choose which extra vertices in $G^i$ get mapped to $H_m$. We also need to multiply by $m!\cdot n!$ to account for the ways to order the vertices within each part of $H$, and by $\binom{j_1'+\dots+j_m'}{j_1',\dots,j_m'}$ to choose which vertices of $G$ map to which vertex of $H_m$. This gives a $[p_{j_1',\dots,j_m',1^{|V(G)|-(j_1'+\dots+j_m')}}]X_G^H$ coefficient of $$m!\cdot n!\cdot\binom{j_1'+\dots+j_m'}{j_1',\dots,j_m'}\sum_{(b_1,\dots,b_\ell)\in\{1,2\}^\ell} (-1)^{(j_1'+\dots+j_m')-(k_{b_1}^1+\dots+k_{b_\ell}^\ell)} \sum_{j_1+\dots+j_\ell=j_1'+\dots+j_m'} \ \prod_{i=1}^\ell\binom{k_{3-b_i}^i}{j_i-k_{b_i}^i}.$$ Our partition $j_1',\dots,j_m',1^{|V(G)|-(j_1'+\dots+j_m')}$ has size $|V(G)|$ and length $m+|V(G)|-(j_1'+\dots+j_m'),$ so the difference between its size and its length is $j_1'+\dots+j_m'-m.$ Thus, applying $\omega$ multiplies the power sum coefficient by $(-1)^{j_1'+\dots+j_m'-m}$. If we then reorganize the terms so they are first grouped by $(j_1,\dots,j_\ell)$ and then by $(b_1,\dots,b_\ell)$, we get that the coefficient of our $p$-term in $\omega(X_G^{K_{m,n}})$ is $$m!\cdot n!\cdot\binom{j_1'+\dots+j_m'}{j_1',\dots,j_m'}\sum_{j_1+\dots+j_\ell=j_1'+\dots+j_m'}(-1)^m\sum_{(b_1,\dots,b_\ell)\in\{1,2\}^\ell} (-1)^{k_{b_1}^1+\dots+k_{b_\ell}^\ell} \prod_{i=1}^\ell\binom{k_{3-b_i}^i}{j_i-k_{b_i}^i}.$$ We can factor the portion corresponding to each tuple $(j_1,\dots,j_\ell)$ to get $$m!\cdot n!\cdot\binom{j_1'+\dots+j_m'}{j_1',\dots,j_m'}\sum_{j_1+\dots+j_\ell = j_1'+\dots+j_m'}(-1)^m \prod_{i=1}^\ell \left((-1)^{k_{1}^i}\binom{k_2^i}{k_1^i+k_2^i-j_i}+(-1)^{k_2^i}\binom{k_1^i}{k_1^i+k_2^i-j_i}\right).$$ Since $k_1^i\ge k_2^i,$ the $(-1)^{k_2^i}$ term dominates for every $i$ regardless of what $j_i$ is, so the overall sign is always $$(-1)^{m+k_2^1+\dots+k_2^\ell},$$ proving (3).
\end{proof}

\section{Bases for \texorpdfstring{$\Lam$}{Lambda} using \texorpdfstring{$H$}{H}-CSFs}\label{sec:bases}

The authors of \cite{EFHKY22} show that it is not possible to construct a basis for the vector space $\Lam^n$ of degree $n$ symmetric functions using $H$-CSFs $X_G^{H_i}$ with $G$ fixed and $H$ varying. On the other hand, they point out that the chromatic bases of Cho and van Willigenburg \cite{cho2015chromatic} form bases of $\Lam^n$ using normal CSFs, and hence also using $H$-CSFs for $H$ a complete graph, since the normal CSF is just a scalar multiple of the $H$-CSF for $H=K_n$. However, they ask what other bases can be formed using $H$-CSFs, and in particular whether there are bases with $H$ fixed but not a complete graph. We investigate those questions in this section.

\subsection{A basis of self-CSFs of complete multipartite graphs}\label{subsec:self-CSF_basis}

Inspired by the $H$-CSF bases discussed in \cite{EFHKY22} and the $r_\lam$ basis studied in \cite{penaguiao2020kernel} and \cite{crew2021complete} of CSFs of complete multipartite graphs, we observe that the self-CSFs of complete multipartite graphs also form a basis for the ring $\Lam$ of symmetric functions:

\begin{prop}\label{prop:multipartite_basis}
    The self-CSFs $X_{K_\lam}^{K_\lam}$ form a basis for $\Lam,$ where $\lam$ ranges over all partitions and $K_\lam$ is the complete multipartite graph whose part sizes are the parts of $\lam$.
\end{prop}

\begin{proof}
    A homomorphism from $X_{K_\lam}^{K_\lam}$ to itself must send vertices from different parts of $K_\lam$ to different vertices, so the shortest length of any monomial in $X_{K_\lam}^{K_\lam}$ is $\ell(\lam)$, and the only monomial of that length that shows up is $m_\lam^n$, because the only way to map $K_\lam$ onto $\ell(\lam)$ vertices is to send all the vertices in each part to the same image vertex. Thus, if we order the set of $X_{K_\lam}^{K_\lam}$'s in some order such that $\ell(\lam)$ is nondecreasing, and we order the $m_\lam$'s in the same order, we get a triangular transition matrix with nonzero diagonal entries from the $m_\lam$ basis to the $X_{K_\lam}^{K_\lam}$'s. Such a triangular matrix must be invertible, so since the $m_\lam$'s form a basis for $\Lam$, the $X_{K_\lam}^{K_\lam}$'s do as well.
\end{proof}

\subsection{Bases \texorpdfstring{$X_{G_\lam}^H$}{XGH} for \texorpdfstring{$\Lam^n$}{Lambda n} with fixed \texorpdfstring{$H$}{H}}\label{subsec:clique_containing_basis}

One question asked in \cite{EFHKY22} is whether there exist non-complete graphs $H$ for which one can build a basis for $\Lam^n$ of self-CSFs $X_{G_\lam}^H$ with $H$ fixed but $G_\lam$ varying. We show that the answer to this question is yes. Here is one construction where $H$ has more than $n$ vertices:

\begin{prop}\label{prop:clique_containing_basis}
    If $H$ is any graph containing an $n$-vertex clique, the self-CSFs $X_{K_\lam}^H$ form a basis for $\Lam^n$, where $\lam$ ranges over all partitions of $n$.
\end{prop}

\begin{proof}
    The argument is essentially the same as for Proposition \ref{prop:multipartite_basis}. Since $H$ contains an $n$-vertex clique, there is a map from $K_\lam$ to $H$ of type $\lam$ for every $\lam\vdash n$. (The requirement that $H$ contain an $n$-vertex clique is needed to ensure this works for $\lam = 1^n$.) However, for any $H$, $X_{K_\lam}^H$ can never contain a monomial of length less than $\ell(\lam)$, and the only monomial of length exactly $\ell(\lam)$ it can contain is $m_\lam^{|V(H)|}$, since only vertices from the same part of $K_\lam$ can map to the same vertex of $H$. Thus, each $X_{K_\lam}^H$ has a different unique monomial of minimal length, so the $X_{K_\lam}^H$'s form a basis for $\Lam^n$ by the same triangularity argument as in Proposition \ref{prop:multipartite_basis}.
\end{proof}

If $H$ has fewer than $n$ vertices, then it is not possible to form a basis for $\Lam^n$ using $H$-CSFs, because there is no way to get the basis vector $m_{1^n}^n.$ The remaining question is thus whether it is possible when $H$ has exactly $n$ vertices but $H\ne K_n.$ It turns out the answer here is also yes. In particular, computations in Sage \cite{sage} show that for $n=3$, it is possible when $H=K_1\sqcup K_2$ but not when $H=P_3$, and for $n=4$, it is possible when $H$ is 7 out of the 11 order 4 graphs: all of them except the edgeless graph, the claw, the 4-cycle, and the complete graph minus an edge. 

We can also give a construction that works for any $n\ge 3$ of a family bases for $\Lam^n$ using $X_{G_\lam}^H$'s with $H$ a fixed $n$-vertex graph that is not a complete graph:

\begin{prop}\label{prop:union_of_cliques_basis}
    Suppose $H$ is a disjoint union of some number of cliques at least one of which has size 3 or more, and at most $\lceil n/2\rceil $ of which have size 1. Then there exist $n$-vertex graphs $G_\lam$ such that the $X_{G_\lam}^H$'s form a basis for $\Lam^n$.
\end{prop}

\begin{proof}
    For $\ell(\lam)=1$, we define $G_\lam=G_n:=\ol{K_n}$ to be the edgeless graph on $n$ vertices. To construct $G_\lam$ for each $\lam$ with $\ell(\lam)\ge 2$, order the parts of $H$ in decreasing order by size, and call their sizes $h_1\ge h_2 \ge \dots \ge h_\ell$. Then we will let $G_\lam$ be a disjoint union of complete multipartite graphs such that the first one has $h_1$ parts whose sizes are the $h_1$ largest parts of $\lam$, the next has $h_2$ parts whose sizes are the next $h_2$ largest parts of $\lam$, and so on.  If we end up at a point where $h_1+\dots+h_{k-1}\le \ell(\lam)\le h_{k}$, we will simply let the last complete multipartite graph have $\ell(\lam)-(h_1+\dots+h_{k-1})$ parts instead of $h_k$ parts, whose sizes are the remaining parts of $\lam$. 
    
    The exception is that if $\ell(\lam)-(h_1+\dots+h_{k-1})=1$ and $\lam_{\ell(\lam)}>1$, we will want to modify the construction slightly so that this final complete multipartite graph is actually connected. To do that, we can remove one of the parts (say, the size $\lam_1$ part) from the first component of $G_\lam$, and instead add it to the last one, so the last one is now a complete bipartite graph with part sizes $\lam_1$ and $\lam_{\ell(\lam)}$. Then since we assumed $h_1\ge 3$, the first component of $G_\lam$ still has at least 2 parts, so it is still connected.

    The other thing to be careful of is that potentially multiple components of $G_\lam$ could end up being disconnected, if for some $k$ we have $h_1+\dots+h_{k-1}<\ell(\lam)$ and $\lam_{h_1+\dots+h_{k-1}+1}>1$ but $h_k=1$, since then the $k^{\text{th}}$ component of $G_\lam$ would consist of multiple isolated vertices and so would not be connected. However, we assumed $H$ has at most $\lceil n/2\rceil$ isolated vertices, meaning that if $h_k=1$, $h_1+\dots+h_{k-1}\ge \lfloor n/2\rfloor.$ If the $(h_1+\dots+h_{k-1}+1)^{\text{st}}$ part of $\lam$ is greater than 1, all parts before that must also be greater than 1, meaning their sum is more than $2\cdot(h_1+\dots+h_{k-1}+1)\ge 2\cdot(\lfloor n/2\rfloor +1)>n.$ This is a contradiction, since $\lam$ is supposed to be a partition of $n$. Thus, this issue cannot actually happen, so all components of $G_\lam$ will actually be connected, although some of them may be singleton vertices.

    To show that the $X_{G_\lam}^H$'s form a basis, first note that $\lam=n$ is the only partition for which $X_{G_\lam}^H$ contains an $m_n^n$ term, because that means mapping all vertices onto the same vertex of $H$, which can only happen if $G_\lam$ is edgeless. 

    We will then order the remaining partitions as follows. We will first list all the partitions with length at most $h_1$ in some order that is \emph{decreasing by length}, so we first list the length $h_1$ partitions, then the length $h_1-1$ partitions, and so on down to the length 2 partitions. Next we list the partitions with length between $h_1+1$ and $h_1+h_2$, again in an order that is decreasing by length, followed by the ones with length between $h_1+h_2+1$ and $h_1+h_2+h_3$ in an order decreasing by length, and so on.

    We claim that each $X_{G_\lam}^H$ in our list includes an $m_\lam^n$ term, and aside from that only includes $m_\mu^n$ terms with $\mu$ showing up \emph{earlier} in the list than $\lam$. To see this, suppose $\ell(\lam)=h_1+\dots+h_{k-1}+r$ where $1\le r \le h_{k}$. We can get an $m_\lam^n$ term by mapping the component of $G_\lam$ with $h_1$ parts onto the $K_{h_1}$ in $H$, then the component with $h_2$ parts onto the $K_{h_2}$, and so on until we map the component with $r$ parts onto $r$ vertices of the $K_{h_k}$ in $H$. For each part of some size $\lam_i$ in of one of these $K_{h_j}$'s, all $\lam_i$ vertices of that part will map onto the same vertex of $H$, while vertices from different parts will map onto different vertices of $H$. This gives a map from $G_\lam$ to $H$ of type $\lam$, so $[m_\lam^n]X_{G_\lam}^H\ne 0$. If we are in an exceptional case where $G_\lam$ instead has a component with $h_1-1$ parts and the last component has 2 parts, we essentially do the same thing except we map the first component onto $h_1-1$ of the vertices in the $K_{h_1}$ in $H$, and the last component onto 2 of the vertices in the $K_{h_k}$ in $H$.

    Now we need to show that any longer monomial $m_\mu^n$ showing up in $X_{G_\lam}^H$ has $\ell(\mu)\le h_1+h_2+\dots+h_k,$ and any shorter monomial has $\ell(\mu)\le h_1+h_2+\dots+h_{k-1}.$ For the longer monomials, note that to maximize the length of a monomial in $X_{G_\lam}^H$, the best we can do is map each component of $G_\lam$ onto a distinct clique in $H$. Since $G_\lam$ has $k$ connected components, we can use at most $k$ of the cliques in $H$, so to maximize the number of vertices we can use, we should choose the $k$ largest cliques, in which case we have at most $h_1+h_2+\dots+h_k$ available vertices. Thus, the longest monomial in $X_{G_\lam}^H$ has length at most $h_1+h_2+\dots+h_k$. 
    
    For the next longest monomial after $m_\lam^n$, note that to get a monomial shorter than $\lam$, we must map two components of $G_\lam$ onto the same component of $H$, because otherwise we would need a distinct vertex of $H$ for each part of each of the multipartite graphs making up $G$, since vertices from different parts of the same component of $G$ must map to distinct vertices of $H$. Then we can at best use the vertices within the $k-1$ largest components $K_{h_1},K_{h_2},\dots,K_{h_{k-1}}$, so at most $h_1+h_2+\dots+h_{k-1}$ vertices of $H$ can be used.

    Then if we add the partition $n$ of length 1 at the very \emph{end} of our ordering, it will still be the case that every $X_{G_\lam}^H$ only contains monomials corresponding to partitions earlier in the ordering, since no $X_{G_\lam}^H$ except for $X_{G_n}^H$ contains an $m_n^n$. It will also still be the case that every $X_{G_\lam}^H$ \emph{does} contain an $m_\lam^n$, since $X_{G_n}^H$ does contain an $m_n^n$ term. So, we have a triangular transition matrix from the $m_\lam^n$'s to the $X_{G_\lam}^H$'s, and all the diagonal entries are nonzero. Thus, we can invert this matrix to write the $m_\lam^n$'s in terms of the $X_{G_\lam}^H$'s, so since the $m_\lam^n$'s form a basis for $\Lam^n$, the $X_{G_\lam}^H$'s do as well.
\end{proof}

On the other hand, there are certain families of $n$-vertex graphs $H$ for which we can show that the $X_G^H$'s do not span $\Lam^n$ as $G$ ranges over all $n$-vertex graphs. The simplest case is when $H$ is edgeless, since then $X_G^H=0$ unless $G$ is also the edgeless graph. A somewhat more interesting example is that when $H=S_n$ is a star graph, Lemma \ref{lem:star_p_expansion} implies that for all $n$-vertex graphs $G$, the $p$-expansion of $X_G^{S_n}$ can only include $p_\lam$ terms where $\lam$ has at most one part larger than 1, so for $n\ge 3$ they will not span $\Lam^n$. There are some other cases we can similarly rule out using power sums. For instance:

\begin{prop}\label{prop:matching_non_basis}
    Let $H$ be the disjoint union of $k$ edges and $n-2k$ isolated vertices, where $k\le \lfloor n/4\rfloor -1$. Then the $X_G^H$'s do not span $\Lam^n$.
\end{prop}

\begin{proof}
    If $G$ is not bipartite, then $X_G^H=0$. For $G$ bipartite, we can imagine computing the $p$-expansion for $X_G^H$ using an inclusion-exclusion approach similar to our approach to the $p$-expansions in \S\ref{sec:p-expansions}. We need each connected component of $G$ of size 2 or more to independently map onto one of the edges of $H$. We can use casework based on which components of size 2 or more in $G$ map onto the same edge of $H$ as each other. For each such case, we get an initial power sum term with parts corresponding to how many of the non-singleton vertices of $G$ are supposed to map onto each endpoint of an edge of $H$, and then the rest of the parts will be $p_1$'s, corresponding to mapping any remaining isolated vertices of $G$ arbitrarily. 
    
    The issue is that among the factors in our $p_\lam$ term that are supposed to correspond to endpoints of edges in $H$, we need to ensure that the same variable does not get chosen multiple times, so that all the endpoints of edges in $H$ get distinct labels. This means we need to alternately subtract and add various power sum terms formed by merging some of the parts in the partitions for our starting power sum terms. However, these are the \emph{only} sorts of additional power sum terms that will arise in our inclusion-exclusion process, so in particular, for any power sum term showing up in our $p$-expansion, the number of parts of size greater than 1 will never exceed the number of parts of size greater than 1 in our starting terms. 
    
    Thus, for any $p$-term in our expansion, the number of parts of size greater than 1 is always at most twice the number of edges in $H$, or at most $2\lfloor n/4\rfloor -2<n/2-2.$ So, for every $n$-vertex graph $G$, $X_G^H$ will only include power sum terms with fewer than $n/2-2$ parts of size 2 or more. This means it is impossible to take a linear combination of these $H$-CSFs to get $p_{2^{n/2}}$ for $n$ even, or to get $p_{2^{(n-1)/2},1}$ for $n$ odd, so the $H$-CSFs do not span all of $\Lam^n$.
\end{proof}

Since it is possible to form a basis for $\Lam^n$ with $X_G^H$'s when $H=K_n$ is the complete graph on $n$-vertices, it is natural to ask whether it is possible to form a basis for $\Lam^n$ with $X_G^H$'s when $H$ is the graph formed by removing one edge from $K_n$, since that graph seems to be closest to the complete graph. However, somewhat surprisingly, this turns out not to be possible:

\begin{prop}\label{prop:complete_minus_edge_non_basis}
    If $H$ is the complete graph $K_n$ with one edge deleted, the $X_G^H$'s do not span $\Lam^n$, and in fact they span a subspace of $\Lam^n$ of codimension 1.
\end{prop}

\begin{proof}
    For each $n$-vertex graph $G$, we will relate the coefficients $[m_\lam^n]X_G^H$ to the coefficients $[m_\lam^{n-1}]X_G^{K_{n-1}}$, and then use that to show that there is a fixed linear relation between the coefficients of $X_G^H$ that is independent of $G$, hence implying that all $X_G^H$'s lie in the proper subspace of $\Lam^n$ consisting of linear combinations of $m_\lam^n$'s whose coefficients satisfy that linear relation.

    Let $u$ and $v$ be the two nonadjacent vertices of $H$. We know that $[m_\lam^n]X_G^H$ counts the number of maps from $G$ to $H$ of type $\lam$. We can split those maps into cases based on which of $u$ or $v$ or both have vertices of $G$ mapping onto them. 
    \begin{itemize}
        \item If $u$ is not used, we can choose any map of type $\lam$ from $G$ to $H-u\cong K_{n-1}$, which can be done in $[m_\lam^{n-1}]X_G^{K_{n-1}}$ ways.
        \item If $u$ is used but $v$ is not, we need to choose a map of type $\lam$ from $G$ to $H-v\cong K_{n-1},$ of which there are $[m_\lam^{n-1}]X_G^{K_{n-1}}$ total. Of those maps to $K_{n-1}$ of type $\lam$, $\ell(\lam)/(n-1)$ of them use $u$, since each map uses $\ell(\lam)/(n-1)$ of the vertices and all vertices are equally likely. Thus, there are $\ell(\lam)/(n-1)\cdot [m_\lam^{n-1}]X_G^{K_{n-1}}$ maps using $u$ but not $v$.
        \item If both $u$ and $v$ are used, we can imagine contracting them into one vertex to get a map from $G$ to $H/uv\cong K_{n-1}$ of type $\mu$, where $\mu$ is obtained from $\lambda$ by combining the two parts $\lam_i$ and $\lam_j$ of $\lam$ corresponding to the numbers of vertices of $G$ mapping onto $u$ and $v.$ The number of such maps is $[m_\mu^{n-1}]K_{n-1}$. The fraction of such maps that actually use the merged vertex $uv$ as a vertex of the desired weight $\lam_i+\lam_j$ is $r_{\lam_i+\lam_j}(\mu)/(n-1)$, since each map uses $r_{\lam_i+\lam_j}(\mu)$ vertices out of the $n-1$ vertices as vertices with $\lam_i+\lam_j$ vertices of $G$ mapping to them, and all vertices of $H$ are equally likely to be used for that. Then if $\lam_i$ and $\lam_j$ were the two merged parts, there are $\binom{\lam_i+\lam_j}{\lam_i}$ ways to choose which of the vertices of $G$ mapping onto $uv$ map onto $u$ and which ones map onto $v$. Putting this together, we get $r_{\lam_i+\lam_j}(\mu)/(n-1)\cdot \binom{\lam_i+\lam_j}{\lam_i}\cdot [m_\mu^{n-1}]K_{n-1}$ maps for each choice of $\mu$.
    \end{itemize}
    Combining these cases gives $$[m_\lam^n]X_G^H = \left(1+\frac{\ell(\lam)}{n-1}\right)\cdot[m_\lam^{n-1}]X_G^{K_{n-1}}+\sum_{\substack{\mu\tn{ from merging }\\ \lam_i\tn{ and }\lam_j}}\frac{r_{\lam_i+\lam_j}(\mu)}{n-1}\cdot\binom{\lam_i+\lam_j}{\lam_i}\cdot[m_\mu^{n-1}]X_G^{K_{n-1}}.$$ The key point is that for every $n$-vertex graph $G$, we can write $[m_\lam^n]X_G^H$ as a linear combination of the form 
    \begin{equation}\label{eqn:transition_formula}
        [m_\lam^n]X_G^H = a_\lam \cdot [m_\lam^{n-1}]X_G^{K_{n-1}}+\sum_{\ell(\mu)<\ell(\lam)}b_{\lam\mu}\cdot [m_\mu^{n-1}]X_G^{K_{n-1}}
    \end{equation} where $a_\lam\ne 0$, $b_{\lam\mu}= 0$ unless $\ell(\mu)<\ell(\lam)$, and the constants $a_\lam$ and $b_{\lam\mu}$ do not depend on the graph $G$. If we let $\lam$ range over all partitions of $n$ except $1^n$ and order the partitions by length, this gives a triangular transition matrix with nonzero diagonal entries from the coefficients $[m_\lam^{n-1}]X_G^{K_{n-1}}$ to the coefficients $[m_\lam^n]X_G^H$. This transition matrix must then be invertible, so we can multiply by the inverse matrix to write every $[m_\lam^{n-1}]X_G^{K_{n-1}}$ with $\lam\ne 1^n$ as a linear combination of the coefficients $[m_\lam^n]X_G^H$ with $\lam\ne 1^n.$ 
    
    Then if we plug in $\lam=1^n$ in (\ref{eqn:transition_formula}), the $[m_{1^n}^{n-1}]X_G^{K_{n-1}}$ term is 0, since there can be no map from $G$ to $K_{n-1}$ sending all $n$ vertices of $G$ to different vertices of $K_{n-1}$, as $K_{n-1}$ only has $n-1$ vertices. The remaining terms can then all be written as linear combinations of the $[m_\lam^n]X_G^H$'s with $\lam \ne 1^n$, as we just noted. Thus, we can plug in those linear combinations to write $[m_{1^n}^n]X_G^H$ as a linear combination of other coefficients $[m_\lam^n]X_G^H$ with $\lam\ne 1^n$, giving a nontrivial linear relation among the $[m_\lam^n]X_G^H$ coefficients that holds for every $n$-vertex graph $G$, since none of the coefficients in the linear relation depend on $G$. This show that the $X_G^H$'s all live in the subspace of $\Lam^n$ consisting of elements whose coefficients satisfy that linear relation, so the $X_G^H$'s do not span $\Lam^n$.

    To see that the $X_G^H$'s span a subspace of $\Lam^n$ of codimension 1, we note that since (\ref{eqn:transition_formula}) describes an invertible linear map between the coefficients of $X_G^{K_{n-1}}$ and the coefficients of $X_G^H$ excluding the $m_{1^n}^n$ coefficient, this also gives us a fixed linear map from $X_G^H$ to $X_G^{K_{n-1}}$, where we can just ignore the $m_{1^n}^n$ coefficient of $X_G^H$, and this map is the same for all $G$. But the coefficients of $X_G^{K_{n-1}}$ are related by another invertible linear map to the coefficients of the ordinary CSF $X_G$, excluding the $m_{1^n}$ coefficient of $X_G$, since for $\ell(\lam)<n$, $[m_\lam]X_G^{K_{n-1}}$ is a scalar multiple of $[m_\lam]X_G$, with the scale factor depending only on $\lam$ and $n$ and not on $G$. It is know from \cite{cho2015chromatic} that the ordinary CSFs $X_G$ span $\Lam^n$. The image of the map from $\Lam^n\to \Lam^n$ sending $m_{1^n}$ to 0 and fixing all other monomials has codimension 1, so the images of the $X_G$'s under that map span that subspace of codimension 1. Then we can apply an invertible linear map to this subspace that scales all the remaining $m_\lam$'s such that $X_G$ maps to $X_G^{K_{n-1}}$, showing that the $X_G^{K_{n-1}}$'s span the same subspace of codimension 1. Then the map taking the $X_G^{K_{n-1}}$'s to the $X_G^H$'s is also invertible, as noted above, so the $X_G^H$'s also span a subspace of $\Lam^n$ of codimension 1.
\end{proof}

We can also generalize this argument a bit further:

\begin{prop}\label{prop:clone_non_basis}
    Suppose $H$ is an $n$-vertex graph containing two nonadjacent vertices $u$ and $v$ both with the same set of neighbors in $H$, such that the graph $H'=H-u\cong H-v\cong H/uv$ is vertex-transitive (meaning for any two of its vertices $w$ and $x$, there is an automorphism taking $w$ to $x$).Then the $H$-CSFs $X_G^H$ do not span $\Lam^n$, and in fact their span is a subspace of the codimension 1 subspace from Proposition \ref{prop:complete_minus_edge_non_basis} with dimension equal to the span of the $H$-CSFs $X_G^{H'}$.
\end{prop}

\begin{proof}
    The argument is essentially the same as for Proposition \ref{prop:complete_minus_edge_non_basis}. For each $n$-vertex graph $G$ and each partition $\lam$ of $n$, we can compute $[m_\lam^n]X_G^H$ by splitting the maps from $G$ to $H$ into cases based on whether or not each of $u$ or $v$ is used. The number of maps with $u$ not used is $[m_\lam^{n-1}]X_G^{H-u}$. The number of maps with $v$ not used but $u$ used is $\ell(\lam)/(n-1)\cdot[m_\lam^{n-1}]X_G^{H-v}$, since the vertex-transitive property guarantees that $u$ has the same chance of being used as all other vertices in $H-v$, so it is used in $\ell(\lam)/(n-1)$ of the possible maps of type $\lam$, since $\ell(\lam)$ of the $n-1$ vertices are used in each such map. Finally, the number of maps with $u$ and $v$ both used can be split into cases as before, with each case corresponding to a partition $\mu$ formed by merging two parts of $\lam$. Since $H-u\cong H-v\cong H/uv \cong H'$, we get $$[m_\lam^n]X_G^H = \left(1+\frac{\ell(\lam)}{n-1}\right)\cdot[m_\lam^{n-1}]X_G^{H'}+\sum_{\substack{\mu\tn{ from merging }\\ \lam_i\tn{ and }\lam_j}}\frac{r_{\lam_i+\lam_j}(\mu)}{n-1}\cdot\binom{\lam_i+\lam_j}{\lam_i}\cdot[m_\mu^{n-1}]X_G^{H'}.$$ This gives us the same triangular matrix as in Proposition \ref{prop:complete_minus_edge_non_basis}, now sending the coefficients $[m_\lam^{n-1}]X_G^{H'}$ to the coefficients $[m_\lam^n]X_G^H$ for all $\lam$ with $\ell(\lam)<n$. Thus, we can invert the matrix to solve for the coefficients $[m_\lam^{n-1}]X_G^{H'}$ in terms of the coefficients $[m_\lam^n]X_G^H$, and then we can plug that in to get exactly the same expression as before for $[m_{1^n}^n]X_G^H$ in terms of the remaining monomial coefficients of $X_G^H$. Thus, we actually get the same fixed linear relation among the coefficients of $X_G^H$ as in the case where $H$ is the complete graph minus an edge, so the $X_G^H$'s all lie in the subspace of $\Lam^n$ defined by that linear relation.

    Since we have an invertible linear relation between the coefficients of the $H$-CSFs $X_G^H$ excluding their $m_{1^n}$ coefficient and the full set of coefficients of the $H'$-CSFs $X_G^{H'}$ (which do not have an $m_{1^n}$ term as $H'$ has only $n-1$ vertices), and also the $m_{1^n}$ coefficient of $X_G^H$ can be expressed in terms of the remaining coefficients, we get an invertible linear map between the $X_G^H$'s and the $X_G^{H'}$'s. Thus, the span of the $X_G^H$'s must have the same dimension as the span of the $X_G^{H'}$'s.
\end{proof}

We can also generalize Proposition \ref{prop:complete_minus_edge_non_basis} in a different direction:

\begin{prop}\label{prop:k_clones_non_basis}
    Suppose $H$ is formed by removing $k$ non-overlapping edges from $K_n$ for some $k\le \lfloor n/2\rfloor.$ Then the $H$-CSFs $X_G^H$ span a subspace of $\Lam^n$ of dimension equal to the number of partitions of $\lam$ of length at most $n-k$.
\end{prop}

\begin{proof}
    The idea will be to relate the coefficients of $X_G^H$ to the coefficients of $X_G^{K_{n-k}}$. Let $u_1v_1,\dots,u_kv_k$ be the $k$ removed edges. To compute $[m_\lam^n]X_G^H$, we can consider cases based on which of which of the $u_i$'s and $v_i$'s are used in the map from $G$ to $H$. For each map of type $\lam$ from $G$ to $H$, we can write $\{1,2,\dots,k\}=S_u\sqcup S_v\sqcup S_{uv}$, where $S_u$ is the set of indices $i$ such that $u_i$ is not used, $S_v$ is the set of indices $i$ such that $u_i$ is used but $v_i$ is not, and $S_{uv}$ is the set of indices such that both $u_i$ and $v_i$ are used. 

    If we remove each $u_i$ for $i\in S_u$, remove each $v_i$ for $i\in S_v$, and then contract each edge $u_iv_i$ with $i\in S_{uv}$, our map from $G$ to $H$ of type $\lam$ turns into a map from $G$ to $K_{n-k}$ of type $\mu$, where $\mu$ is obtained from $\lam$ by merging $|S_{uv}|$ pairs of parts of $\lam$ into $|S_{uv}|$ single parts of $\mu$. Now we need to count how many maps of type $\lam$ from $G$ to $H$ falling under this particular case for $S_u,S_v,$ and $S_{u\sqcup v}$ correspond to each map of type $\mu$ from $G$ to $K_{n-k}$. 
    
    First, we must ensure that the vertex $u_i$ actually gets used for every $i\in S_v$, and that for each $\ell\in S_{uv}$, the merged vertex $u_\ell v_\ell$ gets used and corresponds to a part of the desired size $\lam_{i_\ell}+\lam_{j_\ell}$. Thus, we have $|S_v|+|S_{uv}|$ particular vertices that need to be in the image. All $\binom{n-k}{\ell(\mu)}$ sets of $\ell(\mu)$ vertices are equally likely by the symmetry of $K_{n-k}$, and the number of such sets that contain all $|S_v|+|S_{uv}|$ of the desired vertices is $\binom{n-k-|S_v|-|S_{uv}|}{\ell(\mu)-|S_v|-|S_{uv}|}$, so the chance our vertices in $S_v$ and $S_{uv}$ all get used is $\binom{n-k-|S_v|-|S_{uv}|}{\ell(\mu)-|S_v|-|S_{uv}|}/\binom{n-k}{\ell(\mu)}$. Then the chance that the vertices in $|S_{uv}|$ all have the weights we want them to is a probability $P_{\mu,S_{uv}}$ independent of $G$, that can be explicitly computed by counting the number of ways to assign the remaining weights in $\mu$ to the remaining vertices after assigning all the desired weights to the desired vertices, and then dividing by the total number of ways to assign weights to vertices.
    
    Next, for the merged vertices $u_iv_i$ in the $K_{n-k}$ for each $i\in S_{uv}$, we need to choose how to partition the corresponding vertices of $G$ among the two vertices $u_i$ and $v_i$ of $H$. If the merged parts have sizes $\lam_i$ and $\lam_j,$ there are $\binom{\lam_i+\lam_j}{\lam_i}$ ways to do this, and then we need to multiply over all such pairs of merged parts.

    Putting this together and summing over all cases, we get $$[m_\lam^n]X_G^H = \sum_{S_u\sqcup S_v\sqcup S_{uv}=\{1,2,\dots,k\}} \ \sum_{\substack{\mu\tn{ from merging} \\ \lam_{i_\ell}\tn{ and }\lam_{j_\ell} \\ \tn{for }1\le \ell\le |S_{uv}|}}\left([m_\mu^{n-k}]X_G^{K_{n-k}}\cdot \frac{\binom{n-k-|S_v|-|S_{uv}|}{\ell(\mu)-|S_v|-|S_{uv}|}}{\binom{n-k}{\ell(\mu)}}\cdot P_{\mu,S_{uv}}\cdot\prod_{\ell=1}^{|S_{uv}|} \binom{\lam_{i_\ell}+\lam_{j_\ell}}{\lam_{i_\ell}}\right).$$ Like before, the key point is that we can write $$[m_\lam^n]X_G^H = a_\lam \cdot [m_\lam^{n-k}]X_G^{K_{n-k}}+\sum_\mu b_{\lam\mu}\cdot [m_\lam^{n-k}]X_G^{K_{n-k}}$$ where $a_\lam\ne 0$ (since we get nonzero terms when $S_{uv}=\emptyset$), $b_{\lam\mu}=0$ unless $\ell(\mu)<\ell(\lam)$, and the $a_\lam$'s and $b_{\lam\mu}$'s do not depend on our specific choice of $G$. Thus, if we order our partitions by length, we have a triangular transition matrix with nonzero diagonal entries between the coefficients of $X_G^{K_{n-k}}$ and the coefficients of the $X_G^H$, and this matrix is the same for all $G$. That is, we have an invertible linear map on $\Lam^n$ that takes every $X_G^{K_{n-k}}$ to the corresponding $X_G^H$, which means the dimension of the span of the $X_G^H$'s must be the same as the dimension of the span of the $X_G^{K_{n-k}}$'s. But in the $X_G^{K_{n-k}}$'s, there are no monomial terms of the length more than $n-k$, since $K_{n-k}$ has only $n-k$ vertices. Also, the remaining monomial coefficients $[m_\lam^{n-k}]X_G^{K_{n-k}}$ are scalar multiples of the corresponding coefficients $[m_\lam]X_G$, with nonzero scale factors that depend on $\lam,n,$ and $k$ but not on $G$. Thus, there is a linear map from $\Lam^n$ to itself taking all the ordinary CSFs $X_G$ to the corresponding $H$-CSFs $X_G^{K_{n-k}}$, given by sending all partitions of length more than $n-k$ to 0, and scaling each partition of shorter length by some nonzero scale factor. The classic CSFs $X_G$ are known to span $\Lam^n$ by \cite{cho2015chromatic}, so their images under this map span the full image of the map, which has dimension equal to the number of partitions of $n$ of length at most $n-k$. Since we also have an invertible linear map taking the $X_G^{K_{n-k}}$'s to the $X_G^H$'s, the span of the $X_G^H$'s has the same dimension.
\end{proof}

Another thing the authors of \cite{EFHKY22} ask about is constructing bases for $\Lam^n$ where the graphs $G_\lam$ and $H$ may have \emph{\tb{\tcb{loops}}}, i.e. edges connected a vertex to itself. We can indeed generate some new bases by allowing for the graphs to have loops. For instance, while an edgeless simple graph $H$ generally does not allow for $X_{G_\lam}^H$'s that form a basis since $X_G^H=0$ unless $G$ has a loop, if $H$ is an edgeless graph with loops, we can form such a basis as long as the components of $G_\lam$ have sizes equal to the parts of $\lam$, as in the chromatic bases from \cite{cho2015chromatic}:

\begin{prop}\label{prop:edgeless_with_loops_basis}
    Let $H$ be a graph consisting of $n$ or more isolated vertices, with loops attached to at least $n$ of those vertices, but no other edges. For each $\lam\vdash n$, let $G_\lam$ be any graph such that the sizes of the connected components are the parts of $\lam$. Then the $X_{G_\lam}^H$'s form a basis for $\Lam^n$.
\end{prop}

\begin{proof}
    In $X_{G_\lam}^H$, the longest monomial is $m_\lam$, because each connected component of $G_\lam$ must map to a single vertex of $H$, and since $H$ has at least $n$ different vertices with loops, mapping the components of $G_\lam$ to $\ell(\lam)$ distinct vertices of $H$ is guaranteed to be possible. In that case, we have $\lam_i$ vertices mapping onto whichever vertex in $H$ the component of $G$ of size $\lam_i$ maps onto, so the map has type $\lam$. We have a triangular transition matrix between the $G_\lam$'s and the $m_\lam$'s with nonzero diagonal entries, so the $G_\lam$'s form a basis for $\Lam^n$.
\end{proof}

As another example, we can construct bases where $H$ is a path with a loop on one of its ends, and $G_\lam$ is a disjoint union of such paths:

\begin{prop}\label{prop:path_loop_basis}
    Let $H$ be the $n$-vertex path graph $P_n$ with a loop attached to one end of the path, and for each $\lam\vdash n$, let $G_\lam$ be the disjoint union of $\ell(\lam)$ such paths with lengths equal to the parts of $\lam$, each with a loop attached to one of its endpoints. Then the $X_{G_\lam}^H$'s form a basis for $\Lam^n.$
\end{prop}

\begin{proof}
    Write $\lam^t$ for the transpose of $\lam$, whose $i^{\text{th}}$ part is the number of parts of $\lam$ of size $i$ or more. Arrange the partitions of $n$ in increasing lexicographic order (so $\lam=1^n$ comes first and $\lam=n$ comes last). We claim that under this ordering, $X_{G_\lam}^H$ is the first $H$-CSF to contain an $m_{\lam^t}$ term.
    
    To see this, note first that when mapping $G_\mu$ onto $H$, all $\ell(\mu)$ loop vertices in $G_\mu$ need to map onto the single loop vertex in $H$. Then to get a map of type ${\lam^t}$ from ${G_\mu}$ to $H$, we need to use at least $\ell(\lam^t)=\lam_1$ distinct vertices of $H$, since $\lam^t$ has $\lam_1$ parts. Let the vertices of $H$ be $v_1,v_2,\dots,v_n$ in order, with $v_1$ as the loop vertex. In $G_\mu$, the component of size $\mu_i$ can at most map onto the first $\mu_i$ vertices $v_1,\dots,v_{\mu_i}$ of $H$, since its loop vertex must map onto $v_1$. Thus, overall at most the first $\mu_1$ vertices $v_1,\dots,v_{\mu_1}$ can be used in any map from $G_\mu$ to $H$. Thus, to get an $m_{\lam^t}$ term in $X_{G_\mu}^H$, we first need $\mu_1\ge \lam_1$. 
    
    We will now show that for each $k$, if $\mu_i=\lam_i$ for $i=1,2,\dots,k-1$, then the only way $G_\mu$ can have an $m_{\lam^t}$ term is if $\mu_k\ge \lam_k$. To see this, note that for each $i=1,2,\dots,k$, $\lam^t$ has $\lam_i$ parts of size $i$ or more, which means we need a map from $G_\mu$ to $H$ such that at least $\lam_i$ vertices have $i$ or more vertices of $G_\mu$ mapping onto them. For $i=1,$ at most the vertices $v_1,\dots,v_{\mu_1}$ can have a vertex of $G_\mu$ mapping onto them, so if all of them do, the vertices of a size $\mu_1$ component of $G_\mu$ need to map exactly onto those vertices in a 1-to-1 mapping. Then, at most the vertices $v_1,\dots,v_{\mu_2}$ can have 2 or more vertices of $G_\mu$ mapping onto them, with equality if and only if there is a 1-to-1 mapping from the vertices of a size $\mu_2$ component of $G_\mu$ onto $v_1,\dots,v_{\mu_2}$. Continuing in this manner, we can show by induction that in order to get $\lam_i$ vertices of $H$ with $i$ or more vertices of $G$ mapping onto them, there must be a 1-to-1 mapping from the vertices in a size $\mu_i$ component of $G_\mu$ onto the vertices $v_1,\dots,v_{\mu_i}.$
    
    Then to get at least $\lam_k$ vertices of $H$ with $k$ or more vertices of $G_\mu$ mapping onto them, we can at most use the vertices $v_1,\dots,v_{\mu_k}$, since after mapping the size $\mu_1,\dots,\mu_{k-1}$ components of $G_\mu$ onto $H$, all vertices of $H$ have at most $k-1$ vertices mapping onto them (since they have at most 1 vertex from each component), and $v_{\mu_k}$ is the last vertex of $H$ that can get additional vertices of $G_\mu$ mapping onto it coming from smaller parts of $\mu$. Thus, for there to be $\lam_k$ vertices of $H$ with $k$ or more vertices of $G_\mu$ mapping onto them, we need $\mu_k\ge \lam_k$, as claimed.

    If $\mu_i=\lam_i$ for all $i$, then $\mu=\lam$, and in that case we do get an $m_{\lam^t}$ term in $X_{G_\mu}^H=X_{G_\lam}^H$, because if we map the size $\lam_i$ component onto $v_1,\dots,v_{\lam_i}$, then the number of vertices mapping onto $v_j$ is equal to the number of parts of $G_\lam$ of size $j$ or more, which is precisely the $j^{\text{th}}$ part of $\lam^t$, so the type is $\lam^t$. Otherwise, we can only get a $\lam^t$ term in $X_{G_\mu}^H$ if for some $k$, $\lam_i=\mu_i$ for all $i\le k-1$ but $\mu_i>\lam_i$, which means that $\mu$ is larger than $\lam$ lexicographically.

    Thus, we have a triangular transition matrix with nonzero diagonal entries between the $X_{G_\lam}^H$'s and the monomials $m_{\lam^t}$, so the $X_{G_\lam}^H$'s form a basis for $\Lam^n$.
\end{proof}

Although the results above are substantial on their own, we are nevertheless far from achieving a complete classification of all finite graphs $H$ according to whether or not the $X_{G_\lam}^H$'s span $\Lam^{|V(H)|}$. As such, this leaves open the following question:

\begin{question}
    Besides the cases highlighted above, what are some other necessary and/or sufficient conditions for $H$ such that there exist graphs $G_\lam$ with the same number of vertices where the $X_{G_\lam}^H$'s span $\Lam^{|V(H)|}$? Is it possible to find a single set of conditions that are uniformly necessary and sufficient, in the sense that for any value $n$ of $|V(H)|$, the same conditions, which may be stated as functions of $n$, are necessary and sufficient for every $n$?
\end{question}

The question above is very difficult in its full generality, but if we forget about the structures of the (in)valid choices of $H$ and focus instead on just the number of such graphs for each $n$, we get another interesting question:

\begin{question}
    Let $p_H(n)$ denote the proportion of isomorphism classes of $n$-vertex graphs $H$ such that there exist $n$-vertex graphs $G_\lam$ where the $X_{G_\lam}^H$'s span $\Lam^n$. What is the asymptotic behavior of $p_H(n)$ as $n\rightarrow\infty$?
\end{question}

To get some insight into this question, we have gathered some numerical evidence using our Sage code:

\begin{center} \begin{tabular}{|c|c|c|c|c|c|} \hline
    $n$ & 2 & 3 & 4 & 5 & 6 \\ \hline
    $p_H(n)$ & $1/2=0.5$ & $2/4=0.5$ & $7/11\doteq0.636$ & $27/34\doteq0.794$ & $138/156\doteq0.885$ \\ \hline
\end{tabular} \end{center}
This leads us to believe that $p_H(n)\overset{n\rightarrow\infty}{\longrightarrow}1$, but we are unsure of the rate of convergence. Obtaining further numerical evidence would be tedious, since each calculation of $p_H(n)$ involves computing $X_G^H$ for all $(\#\{n\tn{-vertex graphs}\})^2$ choices of the ordered pair $(G,H)$, and the number of $n$-vertex graphs grows superexponentially with $n$ (for instance, there are 1044 7-vertex graphs, compared to 156 6-vertex ones), not to mention the exponentially growing time-complexity of individual $H$-CSF computations. 

% \begin{prop}\label{prop:complete_plus_loops_basis}
%     Suppose $H$ is the complete graph $K_n$ with loops attached to $k$ of the vertices for some $k\le n$, and take any set of graphs $G_\lam$ such that the ordinary CSFs $X_{G_\lam}$ form a basis for $\Lam^n$. Then the $H$-CSFs $X_{G_\lam}^{H}$ also form a basis for $\Lam^n$.
% \end{prop}

% \begin{proof}
%     We use a similar approach to Propositions \ref{prop:complete_minus_edge_non_basis}, \ref{prop:clone_non_basis}, and \ref{prop:k_clones_non_basis}, except that instead of contracting vertices, we will ``uncontract" the $k$ vertices $v_1,\dots,v_k$ with loops to turn each $v_i$ into two connected vertices $u_i$ and $w_i$, forming a complete graph $K_{n+k}.$ To compute $[m_\lam^n]X_G^H$, we can consider cases based on how many of the vertices $v_1,\dots,v_k$ are used in our map of type $\lam$ from $G$ to $H$. If $i$ of them are used, our map could correspond to 
% \end{proof}

\subsection{Nonexistence of bases \texorpdfstring{$X_G^{H_\lam}$}{XGH} for \texorpdfstring{$\Lam^n$}{Lambda} with fixed \texorpdfstring{$G$}{G}, even with loops}\label{subsec:G_fixed_no_bases}

The authors of \cite{EFHKY22} show that for $n\ge 4$, it is impossible to build a basis for $\Lam^n$ using $H$-chromatic symmetric functions $X_G^{H_\lam}$ with $G$ fixed. Their argument is based on the fact that to get an $m_n$ term, $G$ would need to be edgeless. (For $n\le 3$, if $G$ is the edgeless graph, there does exist a basis of $X_G^{H_\lam}$'s.) However, they ask whether it is possible to form a basis for $\Lam^n$ with $H$-CSFs $X_G^{H_\lam}$ with $G$ fixed if the graphs $G$ and $H_\lam$ are allowed to have loops. We show that this still generally not possible for $n$ sufficiently large (although it is possible for a few more values of $n$):

\begin{prop}\label{prop:no_basis_H_fixed}
    For $n\ge 12$, it is not possible to find a graph $G$ and a set of graphs $H_\lam$ for which the $H$-CSFs $X_G^{H_\lam}$ span $\Lam^n$, even if the graphs $G$ and $H_\lam$ are allowed to have loops.
\end{prop}

\begin{proof}
    If the $X_G^{H_\lam}$'s span $\Lam^n$, there must be some linear combination of them giving every length 2 monomial $m_{\lam_1\lam_2}$. The coefficient $[m_{\lam_1\lam_2}^{|V(H)|}]X_G^{H}$ counts the number of maps from $G$ to $H$ of type $\lam_1\lam_2$, each of which corresponds to a map of type $\lam_1\lam_2$ from $G$ to 2-vertex induced subgraph of $H$. If we allow loops, there are 6 different isomorphism classes of 2-vertex graphs, since there can be an edge between the vertices of not (2 options), and there can be a loop attached to both, one, or neither of the vertices (3 options). Call these graphs $H_1,\dots,H_6$, and for $i=1,\dots,6$, let $n_i(H)$ be the number of induced subgraphs of $H$ isomorphic to $H_i$. Then we get \begin{equation}\label{eqn:length2_coeffs1}
        [m_{\lam_1\lam_2}^{|V(H)|}]X_G^H = \sum_{i=1}^6 n_i(H)\cdot [m_{\lam_1\lam_2}^2]X_G^{H_i},
    \end{equation} 
    since we can split the maps from $G$ to $H$ of type $\lam_1\lam_2$ into 6 cases depending on the isomorphism class of the induced subgraph of $G$ defined by the 2 vertices used, and for each $1\le i\le 6$, there are $n_i(H)$ ways to choose the 2 vertices such that the induced subgraph is isomorphic to $H_i$, and then $[m_{\lam_1\lam_2}^2]X_G^{H_i}$ ways to choose a map of type $\lam_1\lam_2$ from $G$ to that pair of vertices in $H$. We have $m_{\lam_1\lam_2}^{|V(H)|}=(|V(H)|-2)!\cdot m_{\lam_1\lam_2}^2$, and both are scalar multiples of $m_{\lam_1\lam_2}$, so we can rewrite (\ref{eqn:length2_coeffs1}) as \begin{equation}\label{eqn:length2_coeffs2}
        [m_{\lam_1\lam_2}]X_G^H = (|V(H)|-2)!\cdot \sum_{i=1}^6 n_i(H)\cdot [m_{\lam_1\lam_2}]X_G^{H_i}.
    \end{equation} 
    Let $\pi_2$ be the projection of $\Lam^n$ onto the subspace spanned by the length 2 monomials, which we will call $\Lam^n_2$. That is, $$\pi_2(m_\lam):= \begin{cases}
        m_\lam & \tn{ if }\ell(\lam)=2, \\
        0 & \tn{else}.
    \end{cases}$$
    Then (\ref{eqn:length2_coeffs2}) implies that for every $H$, $$\pi_2(X_G^H) = (|V(H)|-2)!\cdot \sum_{i=1}^6 n_i(H)\cdot \pi_2(X_G^{H_i}).$$ That is, for every $H$, $\pi_2(X_G^H)$ lies in the subspace of $\Lam^n_2$ spanned by the 6 elements $\pi_2(X_G^{H_i})$, which has dimension at most 6. If the $X_G^{H_\lam}$'s are to span $\Lam^n$, then their projections $\pi_2(X_G^{H_\lam})$ need to span $\Lam^n_2$, which means $\Lam^n_2$ must have dimension at most 6, so $n\le 13$.

    To show that in fact we need $n\le 11$, note that $n=12,13$, $\Lam^n_2$ has dimension exactly 6, so for the $X_G^{H_\lam}$'s to span $\Lam^n$, the 6 $\pi_2(X_G^{H_i})$'s all need to be linearly independent, and in particular nonzero. But one of these $H_i$'s is the edgeless graph with no loops, in which case $X_G^{H_i}$ can only be nonzero if $G$ is edgeless. But if $G$ is edgeless, $X_G^H$ depends only on the number of vertices of $H$, since when mapping an edgeless graph to $H$ it does not matter how the vertices of $H$ are connected to each other. Thus, we do not gain any new bases with $G$ edgeless from allowing the $H_\lam$'s to have loops as opposed to requiring them to be simple, so since it was shown in \cite{EFHKY22} that the $X_G^{H_\lam}$'s cannot span $\Lam^n$ for $G$ edgeless if the $H_\lam$'s are simple, they still do not span $\Lam^n$ if the $H_i$'s may have loops. Then for any $G$ that is not edgeless, the $\pi_2(X_G^{H_i})$'s actually only span a subspace of $\Lam_2^n$ of dimension 5, so they do not span all of $\Lam_2^n$, and thus the $X_G^{H_\lam}$'s cannot span $\Lam^n$ for $n=12,13$.
\end{proof}

On the other hand, we can check that for some small values of $n$, it is possible to get a basis for $\Lam^n$ with $G$ fixed if one allows the $H_\lam$'s to have loops but not if the $H_\lam$'s need to be simple, as the following examples illustrate:

\begin{example}
    For $n=4$, we can use the following graphs:
    \begin{center}
        $G=\vcenter
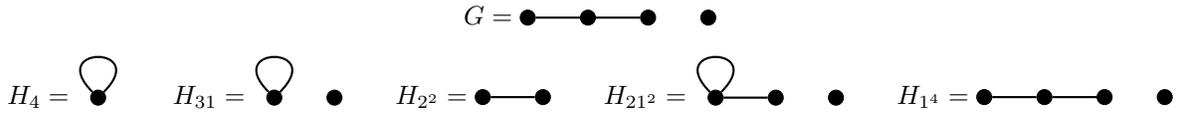
{\hbox{\begin{tikzpicture}
        \graph[nodes={draw, circle, fill=black, inner sep=2pt}, empty nodes, no placement, edges={thick}]{a[x=0,y=0]--b[x=0.8,y=0]--c[x=1.6,y=0];d[x=2.4,y=0]};
        \end{tikzpicture}}}$ \\[6pt]
        $H_4=\hspace{-12pt}\vcenter{\hbox{\begin{tikzpicture}
        \graph[nodes={draw, circle, fill=black, inner sep=2pt}, empty nodes, no placement, edges={thick}]{a[x=0,y=0]--[loop, in=45, out=135, looseness=15]a--[loop, in=315, out=225, looseness=15, draw=none]a};
        \end{tikzpicture}}}\hspace{8pt}
        H_{31}=\hspace{-12pt}\vcenter{\hbox{\begin{tikzpicture}
        \graph[nodes={draw, circle, fill=black, inner sep=2pt}, empty nodes, no placement, edges={thick}]{a[x=0,y=0]--[loop, in=45, out=135, looseness=15]a--[loop, in=315, out=225, looseness=15, draw=none]a;b[x=0.8,y=0]};
        \end{tikzpicture}}}\hspace{20pt}
        H_{2^2}=\vcenter{\hbox{\begin{tikzpicture}
        \graph[nodes={draw, circle, fill=black, inner sep=2pt}, empty nodes, no placement, edges={thick}]{a[x=0,y=0]--b[x=0.8,y=0]};
        \end{tikzpicture}}}\hspace{20pt}
        H_{21^2}=\hspace{-12pt}\vcenter{\hbox{\begin{tikzpicture}
        \graph[nodes={draw, circle, fill=black, inner sep=2pt}, empty nodes, no placement, edges={thick}]{a[x=0,y=0]--[loop, in=45, out=135, looseness=15]a--[loop, in=315, out=225, looseness=15, draw=none]a--b[x=0.8,y=0];c[x=1.6,y=0]};
        \end{tikzpicture}}}\hspace{20pt}
        H_{1^4}=\vcenter{\hbox{\begin{tikzpicture}
        \graph[nodes={draw, circle, fill=black, inner sep=2pt}, empty nodes, no placement, edges={thick}]{a[x=0,y=0]--b[x=0.8,y=0]--c[x=1.6,y=0];d[x=2.4,y=0]};
        \end{tikzpicture}}}$\vspace{-12pt}
    \end{center}
    If we list the $H$-CSFs $X_G^{H_\lam}$ with the $H_\lam$'s in the order shown (reverse lexicographic order), then for each $\lam\vdash 4$, $X_G^{H_\lam}$ is the first $H$-CSF on our list containing an $m_\lam$ term. Thus, we have a triangular transition matrix with nonzero diagonal entries between the $m_\lam$'s and the $X_G^{H_\lam}$'s, the the $X_G^{H_\lam}$'s form a basis for $\Lam^4$.
\end{example}

\begin{example}
    Similarly, for $n=5,$ we can use the following graphs $G$ and $H_\lam$:
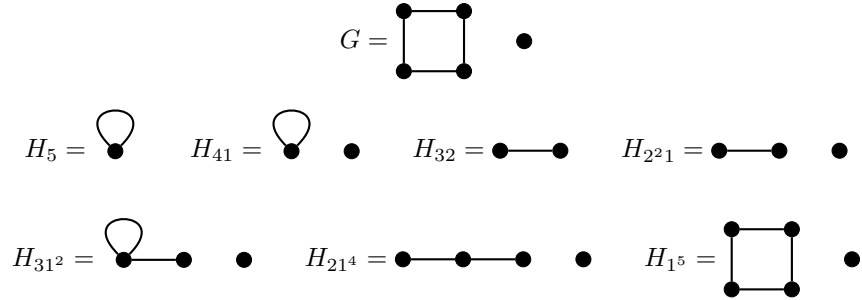
    \begin{center}
        $G=\vcenter{\hbox{\begin{tikzpicture}
        \graph[nodes={draw, circle, fill=black, inner sep=2pt}, empty nodes, no placement, edges={thick}]{a[x=0,y=0]--b[x=0.8,y=0]--c[x=0.8,y=0.8]--d[x=0,y=0.8]--a;e[x=1.6,y=0.4]};
        \end{tikzpicture}}}$ \\[6pt]
        $H_5=\hspace{-12pt}\vcenter{\hbox{\begin{tikzpicture}
        \graph[nodes={draw, circle, fill=black, inner sep=2pt}, empty nodes, no placement, edges={thick}]{a[x=0,y=0]--[loop, in=45, out=135, looseness=15]a--[loop, in=315, out=225, looseness=15, draw=none]a};
        \end{tikzpicture}}}\hspace{8pt}
        H_{41}=\hspace{-12pt}\vcenter{\hbox{\begin{tikzpicture}
        \graph[nodes={draw, circle, fill=black, inner sep=2pt}, empty nodes, no placement, edges={thick}]{a[x=0,y=0]--[loop, in=45, out=135, looseness=15]a--[loop, in=315, out=225, looseness=15, draw=none]a;b[x=0.8,y=0]};
        \end{tikzpicture}}}\hspace{20pt}
        H_{32}=\vcenter{\hbox{\begin{tikzpicture}
        \graph[nodes={draw, circle, fill=black, inner sep=2pt}, empty nodes, no placement, edges={thick}]{a[x=0,y=0]--b[x=0.8,y=0]};
        \end{tikzpicture}}}\hspace{20pt}
        H_{2^21}=\vcenter{\hbox{\begin{tikzpicture}
        \graph[nodes={draw, circle, fill=black, inner sep=2pt}, empty nodes, no placement, edges={thick}]{a[x=0,y=0]--b[x=0.8,y=0];c[x=1.6,y=0]};
        \end{tikzpicture}}}$
        $H_{31^2}=\hspace{-12pt}\vcenter{\hbox{\begin{tikzpicture}
        \graph[nodes={draw, circle, fill=black, inner sep=2pt}, empty nodes, no placement, edges={thick}]{a[x=0,y=0]--[loop, in=45, out=135, looseness=15]a--[loop, in=315, out=225, looseness=15, draw=none]a--b[x=0.8,y=0];c[x=1.6,y=0]};
        \end{tikzpicture}}}\hspace{20pt}
        H_{21^4}=\vcenter{\hbox{\begin{tikzpicture}
        \graph[nodes={draw, circle, fill=black, inner sep=2pt}, empty nodes, no placement, edges={thick}]{a[x=0,y=0]--b[x=0.8,y=0]--c[x=1.6,y=0];d[x=2.4,y=0]};
        \end{tikzpicture}}}\hspace{20pt}
        H_{1^5}=\vcenter{\hbox{\begin{tikzpicture}
        \graph[nodes={draw, circle, fill=black, inner sep=2pt}, empty nodes, no placement, edges={thick}]{a[x=0,y=0]--b[x=0.8,y=0]--c[x=0.8,y=0.8]--d[x=0,y=0.8]--a;e[x=1.6,y=0.4]};
        \end{tikzpicture}}}$
    \end{center}
    Again, if we order the $X_G^{H_\lam}$'s in the order shown, then for every $\lam \vdash 5,$ $X_G^{H_\lam}$ will be the first $H$-CSF on our list with an $m_\lam$ term. Thus, the $X_G^{H_\lam}$'s form a basis for $\Lam^5$ by the same triangularity argument.
\end{example}

It might seem that this pattern of choosing $G$ and the $H_\lam$'s should easily generalize, but alas, no further bases with triangular transition matrices were found. Despite that, we found bases for $n=6$ and $n=7$:

\begin{example}
    For $n=6$, we can use $G=\vcenter{\hbox{\begin{tikzpicture}
        \graph[nodes={draw, circle, fill=black, inner sep=2pt}, empty nodes, no placement, edges={thick}]{a[x=0.4,y=0]--b[x=0,y=0.8]--c[x=1.2,y=0]--d[x=0.8,y=0.8]--a--e[x=1.6,y=0.8]--c;f[x=2.2,y=0.4]};
        \end{tikzpicture}}}$ and the following graphs for the $H_\lam$'s:
    \begin{center}
        $\hspace{-12pt}\vcenter{\hbox{\begin{tikzpicture}
        \graph[nodes={draw, circle, fill=black, inner sep=2pt}, empty nodes, no placement, edges={thick}]{a[x=0,y=0]--[loop, in=45, out=135, looseness=15]a--[loop, in=315, out=225, looseness=15, draw=none]a};
        \end{tikzpicture}}}\hspace{-10pt},\hspace{18pt}
        \vcenter{\hbox{\begin{tikzpicture}
        \graph[nodes={draw, circle, fill=black, inner sep=2pt}, empty nodes, no placement, edges={thick}]{a[x=0,y=0]--[loop, in=45, out=135, looseness=15]a--[loop, in=315, out=225, looseness=15, draw=none]a;b[x=0.8,y=0]};
        \end{tikzpicture}}}\hspace{2pt},\hspace{28pt}
        \vcenter{\hbox{\begin{tikzpicture}
        \graph[nodes={draw, circle, fill=black, inner sep=2pt}, empty nodes, no placement, edges={thick}]{a[x=0,y=0]--b[x=0.8,y=0]};
        \end{tikzpicture}}}\hspace{2pt},\hspace{18pt}
        \vcenter{\hbox{\begin{tikzpicture}
        \graph[nodes={draw, circle, fill=black, inner sep=2pt}, empty nodes, no placement, edges={thick}]{a[x=0,y=0]--[loop, in=45, out=135, looseness=15]a--[loop, in=315, out=225, looseness=15, draw=none]a--b[x=0.8,y=0]};
        \end{tikzpicture}}}\hspace{2pt},\hspace{28pt}
        \vcenter{\hbox{\begin{tikzpicture}
        \graph[nodes={draw, circle, fill=black, inner sep=2pt}, empty nodes, no placement, edges={thick}]{a[x=0,y=0]--b[x=0.8,y=0];c[x=1.6,y=0]};
        \end{tikzpicture}}}\hspace{2pt},\hspace{18pt}
        \vcenter{\hbox{\begin{tikzpicture}
        \graph[nodes={draw, circle, fill=black, inner sep=2pt}, empty nodes, no placement, edges={thick}]{a[x=0,y=0]--[loop, in=45, out=135, looseness=15]a--[loop, in=315, out=225, looseness=15, draw=none]a--b[x=0.8,y=0];c[x=1.6,y=0]};
        \end{tikzpicture}}}\hspace{2pt},$ \\[-4pt]
        $\hspace{-12pt}\vcenter{\hbox{\begin{tikzpicture}
        \graph[nodes={draw, circle, fill=black, inner sep=2pt}, empty nodes, no placement, edges={thick}]{a[x=0,y=0]--[loop, in=45, out=135, looseness=15]a--[loop, in=315, out=225, looseness=15, draw=none]a--b[x=0.8,y=0]--[loop, in=45, out=135, looseness=15]b--c[x=1.6,y=0]--[loop, in=45, out=135, looseness=15]c};
        \end{tikzpicture}}}\hspace{-10pt},\hspace{28pt}
        \vcenter{\hbox{\begin{tikzpicture}
        \graph[nodes={draw, circle, fill=black, inner sep=2pt}, empty nodes, no placement, edges={thick}]{a[x=0,y=0]--b[x=0.8,y=0]--[loop, in=45, out=135, looseness=15]b--[loop, in=315, out=225, looseness=15, draw=none]b--c[x=1.6,y=0];d[x=2.4,y=0]};
        \end{tikzpicture}}}\hspace{2pt},\hspace{28pt}
        \vcenter{\hbox{\begin{tikzpicture}
        \graph[nodes={draw, circle, fill=black, inner sep=2pt}, empty nodes, no placement, edges={thick}]{a[x=0,y=0]--b[x=0.8,y=0]--c[x=0.8,y=0.8]--d[x=0,y=0.8]--a;e[x=1.6,y=0.4]};
        \end{tikzpicture}}}\hspace{2pt},\hspace{28pt}
        \vcenter{\hbox{\begin{tikzpicture}
        \graph[nodes={draw, circle, fill=black, inner sep=2pt}, empty nodes, no placement, edges={thick}]{a[x=0.4,y=0]--b[x=0,y=0.8]--c[x=1.2,y=0]--d[x=0.8,y=0.8]--a--e[x=1.6,y=0.8]--c};
        \end{tikzpicture}}}\hspace{2pt},\hspace{28pt}
        \vcenter{\hbox{\begin{tikzpicture}
        \graph[nodes={draw, circle, fill=black, inner sep=2pt}, empty nodes, no placement, edges={thick}]{a[x=0.4,y=0]--b[x=0,y=0.8]--c[x=1.2,y=0]--d[x=0.8,y=0.8]--a--e[x=1.6,y=0.8]--c;f[x=2.2,y=0.4]};
        \end{tikzpicture}}}\hspace{2pt}.
        $
    \end{center}
\end{example}

\begin{example}
    For $n=7$, we can use $G=\vcenter{\hbox{\begin{tikzpicture}
        \graph[nodes={draw, circle, fill=black, inner sep=2pt}, empty nodes, no placement, edges={thick}]{a[x=0,y=0]--b[x=0,y=0.8]--c[x=0.8,y=0]--d[x=0.8,y=0.8]--a--e[x=1.6,y=0.8]--f[x=1.6,y=0]--d;c--e;b--f;g[x=2.3,y=0.4]};
        \end{tikzpicture}}}$ and the following graphs for the $H_\lam$'s:
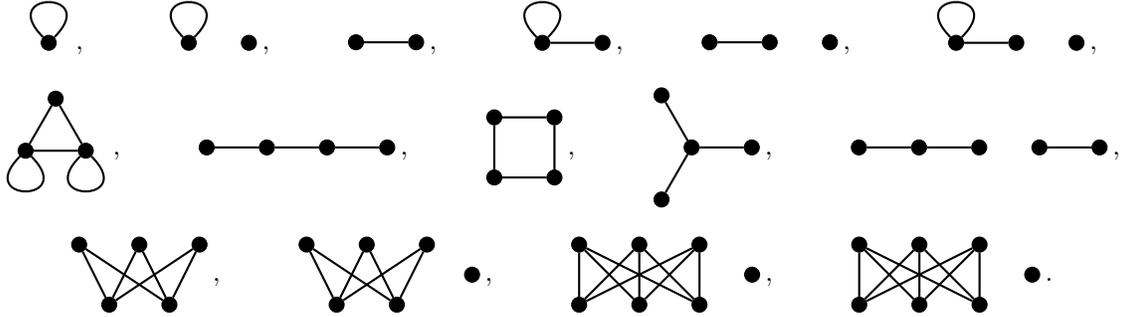
    \begin{center}
        $\hspace{-12pt}\vcenter{\hbox{\begin{tikzpicture}
        \graph[nodes={draw, circle, fill=black, inner sep=2pt}, empty nodes, no placement, edges={thick}]{a[x=0,y=0]--[loop, in=45, out=135, looseness=15]a--[loop, in=315, out=225, looseness=15, draw=none]a};
        \end{tikzpicture}}}\hspace{-10pt},\hspace{18pt}
        \vcenter{\hbox{\begin{tikzpicture}
        \graph[nodes={draw, circle, fill=black, inner sep=2pt}, empty nodes, no placement, edges={thick}]{a[x=0,y=0]--[loop, in=45, out=135, looseness=15]a--[loop, in=315, out=225, looseness=15, draw=none]a;b[x=0.8,y=0]};
        \end{tikzpicture}}}\hspace{2pt},\hspace{28pt}
        \vcenter{\hbox{\begin{tikzpicture}
        \graph[nodes={draw, circle, fill=black, inner sep=2pt}, empty nodes, no placement, edges={thick}]{a[x=0,y=0]--b[x=0.8,y=0]};
        \end{tikzpicture}}}\hspace{2pt},\hspace{18pt}
        \vcenter{\hbox{\begin{tikzpicture}
        \graph[nodes={draw, circle, fill=black, inner sep=2pt}, empty nodes, no placement, edges={thick}]{a[x=0,y=0]--[loop, in=45, out=135, looseness=15]a--[loop, in=315, out=225, looseness=15, draw=none]a--b[x=0.8,y=0]};
        \end{tikzpicture}}}\hspace{2pt},\hspace{28pt}
        \vcenter{\hbox{\begin{tikzpicture}
        \graph[nodes={draw, circle, fill=black, inner sep=2pt}, empty nodes, no placement, edges={thick}]{a[x=0,y=0]--b[x=0.8,y=0];c[x=1.6,y=0]};
        \end{tikzpicture}}}\hspace{2pt},\hspace{18pt}
        \vcenter{\hbox{\begin{tikzpicture}
        \graph[nodes={draw, circle, fill=black, inner sep=2pt}, empty nodes, no placement, edges={thick}]{a[x=0,y=0]--[loop, in=45, out=135, looseness=15]a--[loop, in=315, out=225, looseness=15, draw=none]a--b[x=0.8,y=0];c[x=1.6,y=0]};
        \end{tikzpicture}}}\hspace{2pt},$ \\[-4pt]
        $\hspace{-12pt}\vcenter{\hbox{\begin{tikzpicture}
        \graph[nodes={draw, circle, fill=black, inner sep=2pt}, empty nodes, no placement, edges={thick}]{a[x=0,y=0]--[loop, in=315, out=225, looseness=15,]a--b[x=0.8,y=0]--[loop, in=315, out=225, looseness=15]b--c[x=0.4,y=0.693]--a};
        \end{tikzpicture}}}\hspace{-10pt},\hspace{28pt}
        \vcenter{\hbox{\begin{tikzpicture}
        \graph[nodes={draw, circle, fill=black, inner sep=2pt}, empty nodes, no placement, edges={thick}]{a[x=0,y=0]--b[x=0.8,y=0]--c[x=1.6,y=0]--d[x=2.4,y=0]};
        \end{tikzpicture}}}\hspace{2pt},\hspace{28pt}
        \vcenter{\hbox{\begin{tikzpicture}
        \graph[nodes={draw, circle, fill=black, inner sep=2pt}, empty nodes, no placement, edges={thick}]{a[x=0,y=0]--b[x=0.8,y=0]--c[x=0.8,y=0.8]--d[x=0,y=0.8]--a};
        \end{tikzpicture}}}\hspace{2pt},\hspace{28pt}
        \vcenter{\hbox{\begin{tikzpicture}
        \graph[nodes={draw, circle, fill=black, inner sep=2pt}, empty nodes, no placement, edges={thick}]{a[x=0,y=0]--b[x=-0.4,y=0.693];a--c[x=-0.4,y=-0.693];a--d[x=0.8,y=0]};
        \end{tikzpicture}}}\hspace{2pt},\hspace{28pt}
        \vcenter{\hbox{\begin{tikzpicture}
        \graph[nodes={draw, circle, fill=black, inner sep=2pt}, empty nodes, no placement, edges={thick}]{a[x=0,y=0]--b[x=0.8,y=0]--c[x=1.6,y=0];d[x=2.4,y=0]--e[x=3.2,y=0]};
        \end{tikzpicture}}}\hspace{2pt},$ \\[10pt]
        $\vcenter{\hbox{\begin{tikzpicture}
        \graph[nodes={draw, circle, fill=black, inner sep=2pt}, empty nodes, no placement, edges={thick}]{a[x=0.4,y=0]--b[x=0,y=0.8]--c[x=1.2,y=0]--d[x=0.8,y=0.8]--a--e[x=1.6,y=0.8]--c};
        \end{tikzpicture}}}\hspace{2pt},\hspace{28pt}
        \vcenter{\hbox{\begin{tikzpicture}
        \graph[nodes={draw, circle, fill=black, inner sep=2pt}, empty nodes, no placement, edges={thick}]{a[x=0.4,y=0]--b[x=0,y=0.8]--c[x=1.2,y=0]--d[x=0.8,y=0.8]--a--e[x=1.6,y=0.8]--c;f[x=2.2,y=0.4]};
        \end{tikzpicture}}}\hspace{2pt},\hspace{28pt}
        \vcenter{\hbox{\begin{tikzpicture}
        \graph[nodes={draw, circle, fill=black, inner sep=2pt}, empty nodes, no placement, edges={thick}]{a[x=0,y=0]--b[x=0,y=0.8]--c[x=0.8,y=0]--d[x=0.8,y=0.8]--a--e[x=1.6,y=0.8]--f[x=1.6,y=0]--d;c--e;b--f;g[x=2.3,y=0.4]};
        \end{tikzpicture}}}\hspace{2pt},\hspace{28pt}
        \vcenter{\hbox{\begin{tikzpicture}
        \graph[nodes={draw, circle, fill=black, inner sep=2pt}, empty nodes, no placement, edges={thick}]{a[x=0,y=0]--b[x=0,y=0.8]--c[x=0.8,y=0]--d[x=0.8,y=0.8]--a--e[x=1.6,y=0.8]--f[x=1.6,y=0]--d;c--e;b--f;g[x=2.3,y=0.4]};
        \end{tikzpicture}}}\hspace{2pt}.$
    \end{center} \vspace{6pt}
\end{example}

However, due to computational constraints, we were unable to figure out whether the same is true for $n=8$. This leaves open the following question:

\begin{question}
    Does there exist a choice of $n$-vertex graphs $G$ and $\{H_\lam\mid\lam\vdash n\}$ such that the $X_G^{H_\lam}$'s span $\Lam^n$ for $n=8$ through $11$?
\end{question}

\section{Recursive formula for \texorpdfstring{$H$}{H}-CSFs}\label{sec:recursive}

The ordinary CSF satisfies several nice recursive properties, such as multiplicativity over disjoint unions $X_{G_1\sqcup G_2} = X_{G_1} X_{G_2}$, multiplicativity over joins using the modified $\odot$ multiplication from \cite{tsujie2018chromatic}, and the contraction-deletion relation $X_G = X_{G\backslash e}-X_{G/e}$ if one uses the weighted CSF from \cite{crew2020deletion}. Unfortunately, most of these properties cannot be recovered in the case of $H$-CSFs. However, we look in this section at a few recursive relations that do work for $H$-CSFs.

\subsection{Breaking down \texorpdfstring{$H$}{H} as a disjoint union}\label{subsec:disjoint_union_H}

One thing we can do is to recursively compute $H$-CSFs over a disjoint union of $H$'s. The authors of \cite{EFHKY22} point out that if $G$ is connected but $H$ is not, the coefficients of $X_G^H$ can be calculated in terms of the coefficients of the $X_G^{H_i}$'s, where the $H_i$'s are the connected components of $H$. We can extend this to a similar relation in cases where $G$ is not connected. This relation can be stated more cleanly if we define a modified $H$-CSF by $\til{X}_G^H$ by replacing all the $m_\lam^{|V(H)|}$'s with $\til{m}_\lam$'s, i.e. $$[\til{m}_\lam]\til{X}_G^H := [m_\lam^{|V(H)}]X_G^H,$$ where $\til{m}_\lam$ is the \emph{\tb{\tcb{augmented monomial symmetric function}}} $$\til{m}_\lam := m_\lam \cdot \prod_{i\ge 1}r_i(\lam)!,$$ where we recall that $r_i(\lam)$ is the number of parts of $\lam$ equal to $i$. Thus, $[\til{m}_\lam]\til{X}_G^H$ counts the number of maps from $G$ to $H$ of type $\lam$. Tusjie \cite{tsujie2018chromatic} defines a modified multiplication $\odot$ on symmetric functions by $\til{m}_\lam \odot \til{m}_\mu := \til{m}_{\lam\sqcup \mu}$, where $\lam\sqcup \mu$ is the partition obtained by taking all parts of $\lam$ together with all parts of $\mu$, counted with multiplicity. For instance, $\til{m}_{5422}\sqcup \til{m}_{321}= \til{m}_{5432221}.$

\begin{prop}\label{prop:H_disjoint_union}
    Suppose $H = H_1\sqcup \dots \sqcup H_\ell$, where the $H_i$'s are disjoint (i.e. there are no edges in $H$ between $H_i$ and $H_j$ for $i\ne j$) but each individual $H_i$ may or may not be connected. Then $$\til{X}_G^H = \sum_{G_1\sqcup \dots \sqcup G_\ell=G}\til{X}_{G_1}^{H_1}\odot \dots \odot \til{X}_{G_\ell}^{H_\ell},$$ where the sum ranges over all ways to partition the connected components of $G$ among $\ell$ possibly empty graphs $G_1,\dots,G_\ell$.
\end{prop}

\begin{proof}
    The coefficient $[\til{m}_\lam]\til{X}_G^H$ counts the maps from $G$ to $H$ of type $\lam$. Each connected component of $G$ must map onto a single $H_i$, so we can use casework based on which components of $G$ map onto each $H_i$. The possible cases correspond to the ways to split the connected components of $G$ into $\ell$ sets $G_1,\dots, G_\ell$ such that $G_i$ is the disjoint union of the connected components of $G$ that map onto $H_i.$ Then for each map falling under a particular case, we can write $\lam$ as the disjoint union of $\ell$ partitions $\lam^1,\dots,\lam^\ell$, where $\lam^i$ consists of the parts of $\lam$ corresponding to vertices of $H$ in $H_i$. Then for each $i$, the map from $G$ to $H$ restricts to a map from $G_i$ to $H_i$ of type $\lam^i$, so the number of choices for such a map is $[\til{m}_{\lam^i}]\til{X}_{G_i}^{H_i}.$ The total number of such maps for a particular choices of $\lam^1,\dots,\lam^\ell$ is then the product of these coefficients, and thus the $\til{m}_{\lam^i}$ terms from the $\til{X}_{G_i}^{H_i}$'s multiply to give the contribution to the $\til{m}_\lam$ term in $\til{X}_G^H$ coming from this particular choice of $\lam^1,\dots,\lam^\ell$, since the coefficients multiply, and the monomials themselves also multiply to give $\til{m}_\lam = \til{m}_{\lam^1}\odot \dots \odot \til{m}_{\lam^\ell}$. Summing over all cases for $G_1,\dots,G_\ell$ and $\lam^1,\dots,\lam^\ell$ and using the distributive property of $\odot$ multiplication over addition proves the claimed formula.
\end{proof}

\subsection{Computing \texorpdfstring{$K_n^1$}{Kn1}-CSFS in terms of classic CSFs}\label{subsec:K_n^1}

Next we look at $K_n^1$-CSFs, where $K_n^1$ is the graph formed by attaching a loop to one vertex of the complete graph $K_n$. We will show that for a graph $G$ on $n$ or fewer vertices, its $K_n^1$-CSF can be computed in terms of its CSF and the CSFs of graphs formed by contracting parts of it. For this, we will want to use \emph{\tb{\tcb{weighted graphs}}}, where we have a weight function $w:G\to\{1,2,3,\dots\}$ assigning a positive integer valued weight to each vertex. For a weighted graph $G$, the weighted version of the CSF from \cite{crew2020deletion} is defined by $$X_G := \sum_\kappa \prod_{v\in V(G)}x_{v(G)}^{\kappa(v)},$$ where $\kappa$ ranges over proper colorings of $V(G)$, like for the ordinary CSF. The difference is that the exponent on each color variable is the sum of the weights of the vertices assigned that color, instead of just the number of such vertices.

For a subset $W\se V(G)$ (not necessarily connected), write $G/W$ for the graph formed by contracting all vertices in $W$ into a single vertex whose weight is the sum of the original vertices, such that the new vertex in $G/W$ is connected to all vertices in $G$ with at least one neighbor in $W$. We can now state our result about $K_n^1$-CSFs:

\begin{prop}\label{prop:K_n^1}
    For a graph $G$ on $n$ or fewer vertices with no loops, the $K_n^1$-CSF of $G$ is $$X_G^{K_n^1} = n!\cdot X_G +(n-1)!\cdot \sum_W X_{G/W},$$ where the sum is over all nonempty subsets $W\se V(G)$ such that every vertex in $W$ is adjacent to at least one other vertex in $W$.
\end{prop}

\begin{proof}
    If we restrict to $n$ variables $x_1,\dots,x_n$, then $X_G^{K_n^1}$ is the sum of all terms corresponding to ways to map $G$ to $K_n^1$ and then label the vertices of $K_n^1$ with $1,2,\dots,n.$ We can split these terms into cases based on which sets of connected vertices in $G$ map onto the loop vertex. 
    
    If there are no such vertices, the loop vertex is essentially no different from any other vertex, so the terms we get correspond to all ways to map the vertices of $G$ to the $n$ vertices of $K_n$ such that adjacent vertices in $G$ do not map onto the same vertex, and then to choose a labeling of the vertices of $K_n$. The sum of these terms is just $n!$ times $X_G$ restricted to the same $n$ variables, since instead of choosing a map from $G$ to $K_n$, we could instead first choose a proper coloring of $G$ using the $n$ colors $1,2,\dots,n$, and then choose which color corresponds to which vertex of $K_n$ in $n!$ ways.

    On the other hand, if we have some connected subsets $W_1,\dots,W_k$ each of size greater than 1 all of which map onto the loop vertex $v$ of $K_n^1$, then our map from $G$ to $K_n^1$ can equivalently be thought of as a map from $G/W$ to $K_n$ where $W=W_1\sqcup \dots \sqcup W_k$, except that we force the contracted vertex to map to $v$. Thus, we could instead start by choosing a proper coloring of $G/W$ and then choosing which color corresponds to each vertex of $K_n^1$, except that there will only be $(n-1)!$ valid labelings of $K_n^1$ in this case, since $v$ needs to map to the color assigned to $W$. Thus, the sum of the terms for this case is equal to $(n-1)!\cdot X_{G/W}$.

    This shows that all terms on the left side of our expression correspond to terms on the right side. On the other hand, if we start with a term on the right side, it represents a proper coloring of some $G/W$, so it does correspond to a term on the left side sending all vertices in $W$ to $v$. Also, the same term on the right cannot correspond to multiple terms on the left, because $W$ can be uniquely recovered from a given map from $G$ to $K_n^1$ as the set of all vertices mapping to $v$ that also have at least one neighbor mapping to $v$, thus a given map from $G$ to $K_n^1$ can correspond to only one choice of $W$. 
    
    Since we have a 1-to-1 correspondence between the terms, the two sides are equal when we restrict to the first $n$ variables, and thus since they both can only involve monomials $m_\lam$ of length up to $n$, the two sides are also equal when we extend to the full set of variables by replacing $m_\lam(x_1,\dots,x_n)$ with $m_\lam(x_1,x_2,\dots)$, as needed.
\end{proof}

\section{\texorpdfstring{$H$}{H}-chromatic polynomials}\label{sec:H-chromatic_poly}

\subsection{Definition and polynomiality}\label{subsec:polynomiality}

The regular CSF $X_G$ has the property that evaluating the symmetric function at $1^k$ (i.e. setting the first $k$ indeterminates to 1 and the rest to 0) returns the chromatic polynomial $\chi_G$ evaluated at $k$. We generalize this property to $H$-CSFs by defining an $H$-chromatic polynomial analogously:

\begin{definition}
    Let $G$ and $H$ be finite graphs. Then the \emph{\tb{\tcb{$H$-chromatic polynomial}}} of $G$, denoted $\chi_G^H$, is the function $k\mapsto X_G^H(1^k)$.
\end{definition}

Before we proceed, we shall explain why calling it a ``polynomial'' is still appropriate in this case:

\begin{lemma}\label{lem:chi_G^H_is_poly}
    The $H$-chromatic polynomial is always a polynomial.
\end{lemma}

\begin{proof}
    We show a more general result, namely that the evaluation of any symmetric function at $1^k$ is a polynomial in $k$. By linearity of the evaluation operator, it suffices to show that the evaluation of any monomial symmetric function $m_\lam$ at $1^k$ is a polynomial. We claim that if $\lam\vdash n$, then $$m_\lam(1^k)=\binom{k}{\ell(\lam)}\binom{\ell(\lam)}{r_1(\lam),\dots,r_n(\lam)}.$$ With this claim, the result follows since the first binomial coefficient is a polynomial in $k$ while the second is just a constant. Indeed, $m_\lam$ is the symmetric closure of the term $x_1^{\lam_1}x_2^{\lam_2}\cdots x_{\ell}^{\lam_\ell}$, i.e. the sum of the elements in the orbit of $x_1^{\lam_1}x_2^{\lam_2}\cdots x_{\ell}^{\lam_\ell}$ under the action of the infinite symmetric group permuting the indeterminates $x_i$, and evaluation at $1^k$ corresponds to counting the total number of terms using only the indeterminates $x_1$ through $x_k$, since these terms map to 1 and all others map to 0 under this evaluation map. When restricted to $\ell$ of the indeterminates, the stabilizer has size $r_1(\lam)!\cdots r_n(\lam)!$, since permuting $x_i$ and $x_j$ leads to different terms if and only if their exponents are different, and there are $r_i(\lam)!$ ways to permute the indeterminates with the same exponent $i$ independently for each $i$. As such, by the orbit-stabilizer theorem, the number of terms in the orbit under $S_\ell$ is just $\ell!/(r_1(\lam)!\cdots r_n(\lam)!)$, which is the latter multinomial coefficient in the above. But the evaluation at $1^k$ concerns the size of the orbit under the action of $S_k$ rather than that of $S_\ell$, and to account for that, we multiply by the number of ways to pick $\ell$ indeterminates from $k$, which explains the first binomial coefficient.
\end{proof}

\begin{prop}\label{prop:chi_G^H_formula}
    For any graphs $G$ and $H$ and positive integer $k$, we have
    $$\chi_G^H(k)=|V(H)|!\sum_{L=1}^{\min\{|V(G)|,|V(H)|\}}\frac{\binom{k}{L}}{\binom{|V(H)|}{L}}\cdot\#\{f\in\tn{Hom}(G,H)\mid\ell(\tn{type}(f))=L\}.$$
\end{prop}

\begin{proof}
    From the argument in the proof of Proposition \ref{lem:chi_G^H_is_poly}, we have
    $$\chi_G^H(k)=X_G^H(1^k)=\sum_{\lam\vdash|V(G)|}m_\lam(1^k)[m_\lam]X_G^H=\sum_{\lam\vdash|V(G)|}\binom{k}{\ell(\lam)}\binom{\ell(\lam)}{r_1(\lam),\dots,r_n(\lam)}[m_\lam]X_G^H.$$ We know from the definition of the $H$-CSF that
    $$[m_\lam]X_G^H=\frac{n!}{\binom{n}{r_1(\lam),\dots,r_n(\lam),n-\ell(\lam)}}[m^n_\lam]X_G^H=\frac{n!}{\binom{n}{r_1(\lam),\dots,r_n(\lam),n-\ell(\lam)}}\cdot\#\{f\in\tn{Hom}(G,H)\mid\tn{type}(f)=\lam\}$$
    where $n=|V(H)|$. Substituting into the equation above and noting that $$\binom{\ell(\lam)}{r_1(\lam),\dots,r_n(\lam)}\bigg/\binom{n}{r_1(\lam),\dots,r_n(\lam),n-\ell(\lam)}=\frac{\ell(\lam)!(n-\ell(\lam))!}{n!}=\binom{n}{\ell(\lam)}^{-1},$$
    we see that
    $$\chi_G^H(k)=|V(H)|!\sum_{\lam\vdash|V(G)|}\frac{\binom{k}{\ell(\lam)}}{\binom{|V(H)|}{\ell(\lam)}}\cdot\#\{f\in\tn{Hom}(G,H)\mid\tn{type}(f)=\lam\}.$$
    Lastly, since the fraction $\binom{k}{\ell(\lam)}/\binom{|V(G)|}{\ell(\lam)}$ depends only on the value of $\ell(\lam)$, we can reindex the sum over all partitions $\lam$ by summing over partitions of each length separately. Doing so and noting that
    $$\sum_{\substack{\lam\vdash|V(G)|\\\ell(\lam)=L}}\#\{f\in\tn{Hom}(G,H)\mid\tn{type}(f)=\lam\}=\#\{f\in\tn{Hom}(G,H)\mid\ell(\tn{type}(f))=L\},$$
    and that the maximum value of $L$ for which this quantity is positive is bounded above by both $|V(G)|$ and $|V(H)|$, the result follows.
\end{proof}

\begin{cor}\label{cor:H-chromatic_poly_1}
    The sequence $\big\{\#\{f\in\tn{Hom}(G,H)\mid\ell(\tn{type}(f))=L\}\big\}_{L=1}^{\min\{|V(G)|,|V(H)|\}}$ uniquely determines $\chi_G^H$.
\end{cor}

In fact, since $\#\{f\in\tn{Hom}(G,H)\mid\ell(\tn{type}(f))=L\}$ is just the same as
$$\#\{f\in\tn{Hom}(G,H)\mid\ell(\tn{type}(f))\leq L+1\}-\#\{f\in\tn{Hom}(G,H)\mid\ell(\tn{type}(f))\leq L\},$$
the corollary above is equivalent to the following.

\begin{cor}\label{cor:H-chromatic_poly_2}
    The sequence $\big\{\#\{f\in\tn{Hom}(G,H)\mid\ell(\tn{type}(f))\leq L\}\big\}_{L=1}^{\min\{|V(G)|,|V(H)|\}}$ uniquely determines $\chi_G^H$.
\end{cor}

This characterization makes it apparent that the property that the regular chromatic polynomial is determined by the sequence $\big\{\#\{\tn{colorings of }G\tn{ using at most }L\tn{ colors}\}\big\}_{L=1}^{|V(G)|}$ generalizes to $H$-chromatic polynomials in the desired way.

\begin{example}\label{ex:4cycle_polynomial}
    Let $H$ be the 4-cycle and $G$ be a tree with bipartition sizes $a\geq1$ and $b\geq1$. Then
    $$\chi_G^H(k)=16k(k-1)+(2^{a+2}+2^{b+2}-16)k(k-1)(k-2)+(2^{a+b+1}-2^{a+2}-2^{b+2}+8)k(k-1)(k-2)(k-3).$$
\end{example}

\begin{proof}
    We claim that
    \begin{align*}
        \#\{f\in\tn{Hom}(G,H)\mid\ell(\tn{type}(f))\leq 1\}&=0,\\
        \#\{f\in\tn{Hom}(G,H)\mid\ell(\tn{type}(f))\leq 2\}&=8,\\
        \#\{f\in\tn{Hom}(G,H)\mid\ell(\tn{type}(f))\leq 3\}&=2^{a+2}+2^{b+2}-8,\\
        \#\{f\in\tn{Hom}(G,H)\mid\ell(\tn{type}(f))\leq 4\}&=2^{a+b+1}.
    \end{align*}
    
    Indeed, the first equation follows from the fact that $G$ has an edge and therefore so does $f(G)$, so the vertices of $G$ cannot all map to the same vertex in $H$, but this is the only way for $\tn{type}(f)$ to have length 1.
    
    The second equation follows from the fact that the only way for $\tn{type}(f)$ to have length 2 is for all vertices in one bipartition of $G$ to map to the same vertex $v$ and all vertices in the other bipartition to map to a vertex $w$ adjacent to $v$, and there are 8 possible choices (since there are 4 edges in $H$, multiplied by the 2 possible orientations of each edge) for the ordered pair $(v,w)$ from $V(H)$.

    To explain the third equation, first observe that for $\tn{type}(f)$ to have length at most 3, $f(G)$ must map to a subgraph of $H$ isomorphic to the 3-vertex path. In such a case, either all vertices in the part of size $a$ will map to the central vertex $c$ of said path, in which case there are $2^b$ choices for the images of the remaining vertices since there are 2 choices (the two neighbors of $c$) for each of these $b$ vertices, or all vertices in the part of size $b$ will map to $c$, with $2^a$ choices for the images of the remaining $a$ vertices. Since the choice of $c$ uniquely determines the choice of said 3-vertex path (due to the fact that the two other vertices must be the two neighbors of $c$), and there are 4 such choices (namely the 4 vertices of $H$), we get a total of $4(2^a+2^b)=2^{a+2}+2^{b+2}$ choices. However, this double-counts the 8 cases where $\ell(\tn{type}(f))=2$, since in such cases, $f(G)$ belongs to two different 3-vertex paths in $H$, so we subtract this from the total.

    The fourth equation follows from the fact that a homomorphism $f:G\rightarrow H$ can be inductively constructed as follows. We first order the elements of $V(G)$ into a sequence $v_1,v_2,\dots,v_{|V(G)|}$ such that $v_{i+1}$ is adjacent to the induced subtree of $G$ formed from the vertices $v_1$ through $v_i$ for each $1\leq i\leq|V(G)|-1$. Then there are 4 possible choices for $f(v_1)$ (namely the 4 vertices of $H$), and 2 possible choices for the image of each of the subsequent $a+b-1$ vertices of $G$, since given a choice of $f(v_1),\dots,f(v_i)$, there are 2 choices for $f(v_{i+1})$, namely the two neighbors of $f(v_j)$ where $v_j$ is the unique neighbor of $v_{i+1}$ whose image has already been picked. This gives a total of $4\cdot 2^{a+b-1}=2^{a+b+1}$ choices of $f$.

    Lastly, plugging in these values of $\#\{f\in\tn{Hom}(G,H)\mid\ell(\tn{type}(f))= L\}$ into the formula for $\chi_G^H$ obtained in Proposition \ref{prop:chi_G^H_formula} gives
    $$\chi_G^H(k)=4!\bigg(\frac{8}{\binom{4}{2}}\binom{k}{2}+\frac{2^{a+2}+2^{b+2}-16}{\binom{4}{3}}\binom{k}{3}+\frac{2^{a+b+1}-2^{a+2}-2^{b+2}+8}{\binom{4}{4}}\binom{k}{4}\bigg),$$
    which simplifies to the desired result.
\end{proof}

\subsection{Graphs with equal \texorpdfstring{$H$}{H}-chromatic polynomial but different \texorpdfstring{$H$}{H}-CSF}\label{subsec:equal_poly}

As a corollary of Example \ref{ex:4cycle_polynomial}, any two trees with the same bipartition sizes would have the same $C_4$-chromatic polynomial. But this is as expected: the authors of \cite{EFHKY22} showed more generally that any two connected bipartite graphs with the same part sizes in their bipartitions have the same $K_{m,n}$-CSF, and hence any two such graphs also have the same $H$-chromatic polynomial. However, the $H$-CSF contains more information than the $H$-chromatic polynomial, so we can also ask when the $H$-CSF distinguishes graphs that the $H$-chromatic polynomial does not:

\begin{question}
    Given fixed $H$, which collections of graphs $G$ have the same $\chi_G^H$ but distinct $X_G^H$'s? 
\end{question}

If $H=K_n$ and $|V(G)|\le n$, the $H$-CSF is just a scalar multiple of the ordinary CSF, and hence the $H$-chromatic polynomial is a corresponding scalar multiple of the ordinary chromatic polynomial, so the question becomes which graphs have the same chromatic polynomial but different CSF. It is known, for instance, that all $n$-vertex trees have the same chromatic polynomial $\chi_T(k) = k(k-1)^{n-1},$ while it is famously believed that no two nonisomorphic trees have the same CSF.

However, we can also find other examples that are specific to $H$-CSFs of graphs with the same $H$-chromatic polynomial but different $H$-CSF. For instance, let $S_n^1$ be the $n$-vertex star graph with an extra loop attached to the center vertex, a case for which the $H$-CSF was studied in \cite{EFHKY22}. In that case we can always find graphs distinguished by the $H$-CSF but not the $H$-chromatic polynomial:

\begin{prop}\label{prop:star_plus_loop_polynomial}
    For every $n$, there exist graphs with the same $S_n^1$-chromatic polynomial but different $S_n^1$-CSF.
\end{prop}

\begin{proof}
    The authors of \cite{EFHKY22} show that knowing the $S_n^1$-CSF of a graph is equivalent to knowing the number of independent sets of every size. On the other hand, Corollary \ref{cor:H-chromatic_poly_1} shows that knowing the $S_n^1$ chromatic polynomial is equivalent to knowing the number of maps from $G$ to $H$ using exactly $L$ vertices of $H$ for each $L\le n$. Let $i_j$ be the number of independent sets of size $j$ in $G$ for each $j\ge 1$. If $G$ is not edgeless, then any map from $G$ to $S_n^1$ must use the center vertex of $S_n^1$, and the only restriction is that the vertices of $G$ mapping onto other vertices of $S_n^1$ form an independent set in $G$. Then if we want to count maps using $L$ vertices of $S_n^1$, we can first choose which of the remaining vertices to use in $\binom{n-1}{L-1}$ ways. Then we need to choose some independent set in $G$, and distribute its vertices among those $L-1$ vertices such that all $L-1$ vertices have at least one vertex of $G$ mapping to them. The number of ways to do this is $$\binom{n-1}{L-1}\sum_{j\ge L-1}i_j\cdot (L-1)!\cdot S(j,L-1),$$ where $S(j,L-1)$ is the Stirling number counting set partitions of a size $j$ set into $L-1$ nonempty parts, and the $(L-1)!$ is because we care which part of the set partition maps to which vertex in $S_n^1$. Thus, if we want a graph $G'$ to have the same $S_n^1$-CSF as $G$, it suffices to ensure that for every $j$, $$\binom{n-1}{L-1}\sum_{j\ge L-1}(i_j-i_j')\cdot (L-1)!\cdot S(j,L-1)=0.$$ This can be thought of as a triangular system of $n$ equations in the variables $i_j-i_j'$ (since $L$ can range from 1 to $n$), so if we let $i_1-i_1',\dots,i_{n+1}-i_{n+1}'$ be our variables, we have more variables than equations and thus can find a nontrivial solution, and we can require it to have all rational entries since our matrix has all integer coefficients. Then we can scale the entries up so they are in fact all integers. Now, the $i_j$'s will represent the independent set counts of an actual graph as long as the earlier $i_j$'s are sufficiently large compared to the later ones, since in that case we can start with a large complete graph and then deleting all edges from disjoint size $n+1$ cliques until we get the desired value of $i_{n+1}$, then delete edges from size $n$ cliques until we get the desired value of $i_n$, and so on. Thus, we can simply take each $i_j'$ to be a positive integer that is sufficiently large compared to the later $i_\ell'$'s, and then we can add the differences $i_j-i_j'$ to get the corresponding $i_j$ values, and we will be able to construct graphs $G$ and $G'$ with the desired independent set counts. Then these graphs $G$ and $G'$ will have the same $S_n^1$-chromatic polynomial by construction. However, they will not have the same $S_n^1$-CSF, since for that we would need $i_j=i_j'$ for every $j$, and we have chosen the $i_j$ and $i_j'$ values such that $i_j\ne i_j'$ for at least one value of $j$. Thus, $G$ and $G'$ will have the same $H$-chromatic polynomial but different $H$-CSF, as needed.
\end{proof}

\section*{Acknowledgments}

We thank Logan Crew and Sophie Spirkl for getting us started on this project and for many helpful discussions and suggestions related to it.

\printbibliography

\section*{Appendix: Sage Function to Compute H-CSFs}

See \href{https://github.com/t0tientqu0tient/H-CSFs-of-Graphs}{GitHub repository}.

\end{document}